\documentclass[a4paper,12pt]{amsart}

\title[Galois deformation 
and the pseudo-isomorphism class]{
Euler systems for Galois deformations and 
the pseudo-isomorphism class of the dual 
of fine Selmer groups}
\author{Tatsuya Ohshita}
\address{Science Department, Ehime University 
2--5, Bunkyo-cho,  Matsuyama-shi, 
Ehime  790--8577, Japan}
\email{ohshita.tatsuya.nz@ehime-u.ac.jp}
\date{\today}

\usepackage{amsmath}
\usepackage{amsmath}
\usepackage{amssymb}
\usepackage{amscd}
\usepackage{mathrsfs}
\usepackage[all]{xy}
\usepackage[dvipdfmx]{graphicx}
\usepackage{enumerate}

\newtheorem{thm}{Theorem}[section]
\newtheorem{prop}[thm]{Proposition}
\newtheorem{cor}[thm]{Corollary}
\newtheorem{lem}[thm]{Lemma}

\theoremstyle{definition}
\newtheorem{dfn}[thm]{Definition}
\newtheorem{exa}[thm]{Example}
\newtheorem{rem}[thm]{Remark}

\def\Gal{\mathop{\mathrm{Gal}}\nolimits}
\def\cha{\mathop{\mathrm{char}}\nolimits}
\def\Fitt{\mathop{\mathrm{Fitt}}\nolimits}
\def\det{\mathop{\mathrm{det}}\nolimits}
\def\Im{\mathop{\mathrm{Im}}\nolimits}
\def\Ker{\mathop{\mathrm{Ker}}\nolimits}
\def\Coker{\mathop{\mathrm{Coker}}\nolimits}
\def\Hom{\mathop{\mathrm{Hom}}\nolimits}

\def\ann{\mathop{\mathrm{ann}}\nolimits}
\def\Frob{\text{\rm Frob}}

\def\Het[#1](#2){H_{\text{\'et}}^{#1}(#2)}
\def\Tor{\text{\rm Tor}}

\newcommand{\mf}[1]{{\mathfrak{#1}}}

\newcommand{\bb}[1]{{\mathbb{#1}}}
\newcommand{\mca}[1]{{\mathcal{#1}}}
\newcommand{\mrm}[1]{{\mathrm{#1}}}

\begin{document}

\begin{abstract}
In this article, 
we study the pseudo-isomorphism class of 
the dual fine Selmer group $X$
attached to a $p$-adic Galois deformation
whose deformation ring $\Lambda$ is isomorphic to 
the ring of formal power series.
By using the ``Kolyvagin system" arising from
a given Euler system $\boldsymbol{c}$,
we shall construct a collection 
$\{ \mf{C}_i(\boldsymbol{c}) \}_{i \ge 0}$
of ideals of $\Lambda$, and prove that
the ideals $\mf{C}_i(\boldsymbol{c})$ approximate 
the higher Fitting ideals of $X$ 
under suitable hypothesis.
In particular, we shall prove that 
the ideals $\mf{C}_i(\boldsymbol{c})$ 
arising from the Euler system of Beilinson--Kato elements
determine the pseudo-isomorphism classes of
the dual fine Selmer groups 
attached to ordinary and nearly ordinary 
Hida deformations satisfying certain conditions.

\end{abstract}

\maketitle

\section{Introduction}\label{secintro}

\subsection{Setting and main results}\label{ssintro}

Let $p$ be an odd prime number.
For any $n \in \bb{Z}_{> 0}$, 
we denote by $\mu_{n}$
the group roots of unity in $\overline{\bb{Q}}$,
and put $\mu_{p^\infty}:=\bigcup_{m \ge 0} \mu_{p^m}$.
We  fix a finite set $\Sigma$ 
of prime numbers containing $p$. 
We denote  by $\bb{Q}_\Sigma$ 
the maximal Galois extension field of $\bb{Q}$ 
unramified outside $\Sigma$,
and put $G_{\bb{Q},\Sigma}:=\Gal(\bb{Q}_{\Sigma}/\bb{Q})$.
For a topological $G_{\bb{Q},\Sigma}$-module $M$ 
and any $i \in \bb{Z}_{\ge 0}$, let
\begin{align*}
H^i(M):&=H^i(\bb{Q},M), \\
H^i_{\Sigma}(M):&=H^i(\bb{Q}_{\Sigma}/\bb{Q},M)
\end{align*}
be the $i$-th continuous Galois cohomology groups.
We denote the inertia subgroup of 
$G_{\bb{Q}_\ell}:= 
\Gal(\overline{\bb{Q}}_\ell/\bb{Q}_\ell) $ by $I_\ell$
for any prime number $\ell$.

Let $F$ be a finite extension field of $\bb{Q}_p$,
and $\mca{O}:=\mca{O}_F$ the ring of integers of $F$.
We fix a uniformizer $\varpi$ of $\mca{O}$, and put
$k:=\mca{O}/\varpi\mca{O}$.
Let $r \in \bb{Z}_{>0}$ be a positive integer, and 
$\Lambda:=\Lambda^{(r)}:=\mca{O}[[x_1,\dots , x_{r}]]$
the ring of formal power series.
(We write $\Lambda^{(0)}:=\mca{O}$.)
We denote the maximal ideal of $\Lambda$ by
$\mf{m}:=\mf{m}_{\Lambda}$.

We consider a free $\Lambda$-module
of finite rank $d$ equipped with 
a continuous Galois action  
\[
\rho_{\bb{T}} \colon \Gal(\bb{Q}_{\Sigma}/\bb{Q})
\longrightarrow \mathrm{Aut}_{\Lambda}(\bb{T}) 
\simeq \mathrm{GL}_d(\Lambda),
\]
and put 
$\mca{A}^*:=\Hom_{\mathrm{cont}}(\bb{T},\mu_{p^\infty})$.
(In this article, our main interest is the case 
when $\bb{T}$ is an ordinary or nearly ordinary Hida deformation 
of elliptic modular forms. 
For details, see \S \ref{ssHida}.)
For any ring homomorphism
$f\colon \Lambda \longrightarrow A$,
we define $f^* \bb{T}:=A \otimes_{\Lambda,f} \bb{T}$.
We put $\bar{T}:=\bb{T}/\mf{m}\bb{T}$, and 
denote the maximal $\mf{m}$-torsion $\Lambda$-submodule 
of $\mca{A}$ by $\mca{A}^*[\mf{m}]$. 
In this article, we assume the following conditions.
\begin{itemize}
\item[(A1)] The representation $\bar{T}$ 
of $G_{\bb{Q},\Sigma}$ is absolutely irreducible over $k$.
\item[(A2)] There exists an element 
$\tau \in G_{\bb{Q}(\mu_{p^\infty})}
:=\Gal(\overline{\bb{Q}}/\bb{Q}(\mu_{p^\infty}))$ which makes
the $\Lambda$-module $\bb{T}/(\tau-1)\bb{T}$ free of rank one.
\item[(A3)] We have $H^0(\bb{Q},\bar{T})
=H^0(\bb{Q},\mca{A}^*[\mf{m}])=0$.
\item[(A4)] At least one of the following is satisfied:
\begin{itemize}
\item[(a)] $p \ge 5$, or
\item[(b)] $\Hom_{\bb{F}_p[G_{\bb{Q},\Sigma}]}
(\bar{T},\mca{A}^*[\mf{m}])=0$.
\end{itemize}
\item[(A5)] Let $\Omega$ be the maximal subfield of $\bb{Q}$ 
which is fixed by 
$\ker(\rho_{\bb{T}} \vert_{G_{\bb{Q}(\mu_{p^\infty})}})$. 
Then, we have 
\[
H^1(\Omega/\bb{Q},\bar{T})=H^1(\Omega/\bb{Q},\mca{A}^*[\mf{m}])=0.
\]
\item[(A6)] For any $\ell \in \Sigma \setminus \{ p \}$, we have 
$H^0(I_\ell,\mca{A}^*[\mf{m}])=0$.
\item[(A7)] We have $H^0(\bb{Q}_p,\mca{A}^*[\mf{m}])=0$.
\item[(A8)] Let $\bb{T}^-$ be the maximal 
$\Lambda$-submodule of $\bb{T}$
on which the complex conjugate acts via $-1$.
Then, the $\Lambda$-module $\bb{T}^-$ is free of rank one.
\end{itemize}
Note that we need 
the assumptions (A1)--(A5) and (A8) 
in order to apply the theory of 
Kolyvagin systems established by 
Mazur and Rubin in \cite{MR}.
The assumptions (A6) and (A7)
are technical ones which simplify
the local conditions at bad primes.

We define a $\Lambda$-module $\mathrm{Sel}_p(\bb{Q},\mca{A}^*)$ by
\[
\mathrm{Sel}_p(\bb{Q},\mca{A}^*):= 
\ker \left(
H_{\Sigma}^1(\mca{A}^*)
\longrightarrow H^1(\bb{Q}_p,\mca{A}^*) \times
\prod_{p \ne \ell \in \Sigma}
H^1(I_\ell,\mca{A}^*)
\right),
\]
and define a $\Lambda$-module $X:=X(\bb{T})$ by
\[
X:=\Hom_{\mathrm{cont}}(\mathrm{Sel}_p(\bb{Q},\mca{A}^*), 
\bb{Q}_p/\bb{Z}_p).
\]
For any $i \in \bb{Z}_{\ge 0}$,
the $\Lambda$-module $H_{\Sigma}^i(\mca{A}^*)$ 
(resp.\ $H_{\Sigma}^i(\bb{T})$) is cofinitely generated
(resp.\ finitely generated).
In particular, 
$X$ is a finitely generated $\Lambda$-module.
In this article, we are interested in the pseudo-isomorphism class
of the $\Lambda$-module $X$.
Note that the pseudo-isomorphism class of 
a finitely generated $\Lambda$-module $M$ is 
determined by higher Fitting ideals 
$\{ \Fitt_{\Lambda_{\mf{p}},i}(M_{\mf{p}})\}_{i \in \bb{Z}_{\ge 0}}$
of localizations $M_{\mf{p}}$
at height one primes $\mf{p}$.
We shall study the higher Fitting ideals 
of (localizations of) $X$.

In order to study higher Fitting ideals, 
we assume the existence of an Euler system 
\[
\boldsymbol{c}:=\{c(n) \in  
H^1(\bb{Q}_{\Sigma}/\bb{Q}(\mu_n),\bb{T}) \}_n
\]
of $\bb{T}$. 
Moreover, we assume that the Euler system  
$\boldsymbol{c}$
can be extended to cyclotomic direction. 
(For details, see \S \ref{SSec_defES}.)
Note that practical Euler systems (like 
circular units and Beilinson--Kato elements)
can be extended to cyclotomic direction. 
(See Lemma \ref{lempropcycext}.)
We define an ideal $\mathrm{Ind}_{\Lambda}(\boldsymbol{c})$ 
of $\Lambda$ defined by
\[
\mathrm{Ind}_{\Lambda}(\boldsymbol{c}):= 
\left\{ f(c(1))\mathrel{\vert} 
f \in \Hom_{\Lambda} \left(
H_{\Sigma}^1(\bb{T}),\Lambda \right)
\right\}.
\]
Here, we assume the following ``non-vanishing" conditions on $\boldsymbol{c}$ .
\begin{itemize}
\item[(NV)] \textit{The element $c(1) \in H^1_{\Sigma}(\bb{T})$ is not 
$\Lambda$-torsion.}
\end{itemize}
Note that if $\boldsymbol{c}$
is the Euler system of circular units or Beilinson--Kato elements 
corresponding to a certain $p$-adic $L$-function via 
the Coleman map, 
then the property (NV) for $\boldsymbol{c}$ follows from 
the non-vanishing of the $p$-adic $L$-function.
In \cite{Oc2}, 
Ochiai proved that under the assumption (NV), 
the $\Lambda$-module $X$ is torsion, and satisfies 
\[
\cha_{\Lambda}(X) \supseteq 
\mathrm{Ind}_{\mca{R}}(\boldsymbol{c}),
\]
where $\cha_{\Lambda}(X)$ is the characteristic 
ideal of the $\Lambda$-module $X$.
In order to state some pieces of our results, we also consider  
the following  condition (MC) on the pair $(\mathbb{T},\boldsymbol{c})$.
\begin{itemize}
\item[(MC)] \textit{The Euler system $\boldsymbol{c}$ satisfies} (NV), 
\textit{and it holds that} 
\(
\cha_{\Lambda}(X) =
\mathrm{Ind}_{\Lambda}(\boldsymbol{c})
\).
\end{itemize}
If $\boldsymbol{c}$
is the Euler system of circular units or Beilinson--Kato elements 
corresponding to a certain $p$-adic $L$-function via 
the Coleman map, 
then the property (MC) for $\boldsymbol{c}$
is equivalent to the Iwasawa main conjecture for $\bb{T}$.

In \S \ref{ssconstC_i},
by using the ``universal Kolyvagin system" corresponding to 
$\boldsymbol{c}$, we shall construct an ideal $\mf{C}_i(c)$
of $\Lambda$, 
which is an analogue of Kurihara's 
higher Stickelberger ideal $\Theta_i^{\delta}$ in \cite{Ku}, 
for any $i \in \bb{Z}_{\ge 0}$.

The following is our first main theorem
which does not need the assumption (MC).

\begin{thm}\label{thmuncondresults}
Let $\bb{T}$ be a free $\Lambda$-module of finite rank 
with a continuous $\Lambda$-linear action of $\Gal(\bb{Q}_{\Sigma}/\bb{Q})$.
We assume that $\bb{T}$ satisfies 
the assumptions {\rm (A1)--(A8)}. 
Let $\boldsymbol{c}$ be an Euler system for $\bb{T}$ 
satisfying the condition {\rm (NV)}
which can be extended to cyclotomic direction.
Then, for any height one prime ideal $\mf{p}$ of $\Lambda$ 
and for any $i \in \bb{Z}_{\ge 0}$,
we have 
\[
\Fitt_{\Lambda_\mf{p},i}
(X_\mf{p}) \supseteq 
\mf{C}_i(\boldsymbol{c})\Lambda_\mf{p}. 
\]
\end{thm}

We shall state stronger results under the assumption (MC).
Here, we treat the cases when $\bb{T}$ is 
a one variable deformation, or the cases when 
$\bb{T}$ is the cyclotomic deformation of 
a one variable deformation.
In such cases, we can deduce finer results than 
Theorem \ref{thmcondresults1} and Theorem \ref{thmcondresults2}
under the assumption (MC). 

The results
for one variable cases are as follows.

\begin{thm}\label{thmcondresults1}
Suppose $\Lambda=\mca{O}[[x_1]]$.
Let $(\bb{T},\boldsymbol{c})$ be as in Theorem \ref{thmuncondresults}.
We assume that the Euler system $\boldsymbol{c}$ for $\bb{T}$ 
satisfies the condition {\rm (MC)}.
Let  $\mf{p}$ be a height one prime ideal of $\Lambda$.
Then, for any $i \in \bb{Z}_{\ge 0}$,
we have 
\[
\Fitt_{\Lambda_\mf{p},i}
(X_\mf{p}) =
\mf{C}_i(\boldsymbol{c})\Lambda_\mf{p}. 
\]
In particular, the pseudo-isomorphism class 
of the $\Lambda$-module
$X$ is determined by the collection
$\{ \mf{C}_i(\boldsymbol{c}) \}_{i \in \bb{Z}_{\ge 0}}$
of ideals of $\Lambda$.
\end{thm}

Now, let us consider the case when 
$\bb{T}$ is the cyclotomic deformation of 
a one variable deformation.
We put $\Gamma:=\Gal(\bb{Q}_\infty/\bb{Q})$. 
Let $\chi_{\mathrm{cyc}}\colon \Gamma \longrightarrow \bb{Z}_p^\times$
be the cyclotomic character, 
and $\gamma \in \Gamma$ the topological generator 
given by $\chi_{\mathrm{cyc}}(\gamma)=1+p$.
We put $\Lambda_0=\Lambda^{(1)}$.
From now on, we assume that $r=2$, and that we have 
\[
\Lambda=\Lambda_0[[\Gamma]]= \Lambda^{(2)}, 
\]
where the variable $x_2$ is identified with
$\gamma-1$.
Let $\bb{T}_0$ be a free $\Lambda_0$-module
of finite rank $d$ equipped with 
a continuous $\Lambda_0$-linear action  
$\rho_{\bb{T}_0}$ of $\Gal(\bb{Q}_{\Sigma}/\bb{Q})$
satisfying the conditions (A1)--(A8).
Here, we assume 
\[
(\bb{T},\rho_{\bb{T}})=(\bb{T}_0^{\mathrm{cyc}},\rho_{\bb{T}}^{\mathrm{cyc}})
:=(\bb{T}\otimes_{\Lambda_0}\Lambda,
\rho_{\bb{T}}\otimes \chi_{\mathrm{taut}}),
\]
where $\chi_{\mathrm{taut}} 
\colon \Gal(\bb{Q}_{\Sigma}/\bb{Q}) 
\longrightarrow \Gamma \subseteq \Lambda^\times$ 
is the tautological character.
Note that the assumptions (A1)--(A8) for $\bb{T}_0$
imply that $\bb{T}$ satisfies (A1)--(A8).
(Moreover, we can easily show the converse:  
the conditions (A1)--(A8) for $\bb{T}_0$
are valid if its cyclotomic deformation satisfies them.)
Let $\boldsymbol{c}$ be an Euler system of $\bb{T}$.
Note that $\boldsymbol{c}$
can be extended to cyclotomic direction automatically.
(See Lemma \ref{lempropcycext}.)

\begin{thm}\label{thmcondresults2}
Let $(\bb{T},\boldsymbol{c})$ be as above.
Namely, we set
\(
\Lambda=\Lambda_0[[\Gamma]]= \Lambda^{(2)}
\)
and $(\bb{T},\rho_{\bb{T}})=
(\bb{T}_0^{\mathrm{cyc}},\rho_{\bb{T}}^{\mathrm{cyc}})$.
We assume that the Euler system $\boldsymbol{c}$ for $\bb{T}$ 
satisfies the condition {\rm (MC)}.
Let  $\mf{p}$ be a height one prime ideal of $\Lambda$.
Then, for any $i \in \bb{Z}_{\ge 0}$,
we have 
\[
\Fitt_{\Lambda_\mf{p},i}
(X_\mf{p}) =
\mf{C}_i(\boldsymbol{c})\Lambda_\mf{p}. 
\]
In particular, the pseudo-isomorphism class 
of the $\Lambda$-module 
$X$ is determined by the collection
$\{ \mf{C}_i(\boldsymbol{c}) \}_{i \in \bb{Z}_{\ge 0}}$
of ideals of $\Lambda$.
\end{thm}

\subsection{Strategy}
Here, we introduce the strategy of 
the proof of our main results.

Theorem \ref{thmuncondresults} is proved 
by the induction on the number $r$
of variables in $\Lambda=\Lambda^{(r)}$.
When $r=0$, namely the case when $\Lambda$ is a DVR,
the assertion like Theorem \ref{thmuncondresults} follows 
from the theory of Kolyvagin systems
established by Mazur and Rubin in \cite{MR}.
When $r=1$, 
by using the method developed in \cite{MR} \S 5.3, 
we can reduce the proof to 
the non-variable cases.
Namely, for each hight one prime ideal $\mf{p}$ of $\Lambda$,
we take a sequence $\{ \mf{p}_n \}_{n \ge 0}$
which is a ``perturbation" of the prime ideal $\mf{p}$,
and observe asymptotic behavior of 
the reduction by $\mf{p}_n$.
For $r \ge 2$, we use the method developed by Ochiai in \cite{Oc2}.
We take a suitable ``linear element" $g \in \Lambda$ of $\Lambda$,
and reduce the proof of Theorem \ref{thmuncondresults} for 
the pair $(\bb{T},\boldsymbol{c})$ over $\Lambda^{(r)}$
to the proof of that for $(\pi_g^*\bb{T},
\pi_g^*\boldsymbol{c})$
over $\Lambda^{(r-1)}$, where 
$\pi_g \colon \Lambda \longrightarrow 
\Lambda/g\Lambda=\Lambda^{(r-1)}$
denotes the reduction map.
In the induction arguments, 
a property called ``a weak specialization compatibility"
which says that the image of the ideal
$\mf{C}_i(\boldsymbol{c})$ 
for $(\bb{T},\boldsymbol{c})$
by the reduction map $\pi_g$
is contained in the ideal 
$\mf{C}_i(\pi_g^*\boldsymbol{c})$
for $(\pi_g^*\bb{T},
\pi^*\boldsymbol{c})$
becomes a key.
(For the weak specialization compatibility,
see Proposition \ref{propredonevar}.)

Theorem \ref{thmcondresults1}
and 
Theorem \ref{thmcondresults2}
are also proved by 
reduction arguments like 
\cite{MR} \S 5.3 and \cite{Oc2},
but we need more careful arguments.
In order to prove
these theorems, we need to show
a property on 
$\mf{C}_i(\boldsymbol{c})$
called a strong specialization compatibility
which says that 
that the image of the ideal
$\mf{C}_i(\boldsymbol{c})$ 
by the reduction map $\pi_g$
coincides with the ideal 
$\mf{C}_i(\pi_g^*\boldsymbol{c})$.
(See Theorem \ref{thmonevarcompletered} and 
Theorem \ref{thmcyclotcompletered}.)
The proof of 
the strong specialization compatibility
is the most technical part in our article.
The difficulty to prove  
the strong specialization compatibility
is as follows.
For a positive integer $n$,
we denote by $\mathrm{Prime}(n)$
the set of prime divisors of $n$.
Roughly speaking, 
the ideal
$\mf{C}_i(\pi_g^*\boldsymbol{c})$
is a projective limit of 
certain ideals of quotient rings $\Lambda/I$
generated by the images of
(modified) Kolyvagin derivatives
$\kappa^{\mathrm{univ}}_n(\boldsymbol{c})_{I}$
where $n$ runs through square-free positive integer 
satisfying $\# \mathrm{Prime}(n) \le i$
and contained in a certain set $\mca{N}(\bb{T},I)$.
(For the definition of 
$\mf{C}_i(\pi_g^*\boldsymbol{c})$, see 
Definition \ref{dfnC_i}, and  
for the definition of $\mca{N}(\bb{T},I)$, 
see \S \ref{ssKD}.)
By the definition of $\mca{N}(\bb{T},I)$, 
a prime divisor $\ell$ of an element 
$n \in \mca{N}(\bb{T},I)$
makes 
$(\bb{T}/(\Frob_\ell -1)\bb{T})\otimes_{\Lambda} \Lambda/I$
be a free $\Lambda/I$-module of rank one.
This implies that the set 
$\mca{N}(\bb{T},I)$ is smaller than 
$\mca{N}(\pi_g^*\bb{T},\pi_g(I))$.
So, the ideal
$\pi_g(\mf{C}_i(\boldsymbol{c}))$
may be smaller than 
$\mf{C}_i(\pi_g^*\boldsymbol{c})$.
We can overcome this difficulty 
in the situations of 
Theorem \ref{thmcondresults1}
and 
Theorem \ref{thmcondresults2} as follows.
\begin{itemize}
\item When $\Lambda=\Lambda^{(1)}$, 
we can embed $\Lambda/I$ in to 
a certain quotient ring of DVR.
By using the theory of Kolyvagin systems over DVR,
we can deduce that 
$\pi_g(\mf{C}_i(\boldsymbol{c}))$
is not small.
(For details, see the proof of 
See Theorem \ref{thmonevarcompletered} 
in \S \ref{ssrmkononevarcase}.)
\item When $\bb{T}$ is the cyclotomic deformation of 
a Galois deformation $\bb{T}_0$ over 
$\Lambda^{(1)}$,
the proof of  the strong specialization compatibility
for $\bb{T}$ can be reduced to that
for $\bb{T}_0$, 
namely Theorem \ref{thmcondresults1}.
(For details, see the proof of 
Theorem \ref{thmcondresults2}
in \S \ref{ssrmkoncyclotcase}.)
\end{itemize}

The contents of our article is as follows.
In \S \ref{secprel}, we introduce 
some basic notion and
preliminary results.
In \S \ref{secES}, we review 
the theory of 
Euler systems for Galois deformations
and Kolyvagin systems over DVR.
In \S \ref{secCi}, we define the ideal  
$\mf{C}_i(\boldsymbol{c})$, 
and prove their basic properties, 
that is, independence of the choice of 
a certain system of parameters of $\Lambda$
(Proposition \ref{propindepofparam})
and the stability under 
scalar extensions (Proposition \ref{corOO'}).
In \ref{secred}, 
we prove 
the weak/strong  specialization compatibility
of $\mf{C}_i(\boldsymbol{c})$.
In \S \ref{secpfmr}, 
we prove our main results.
In \S \ref{ssHida}, 
we apply our results to 
ordinary and nearly ordinary 
Hida deformation.

\section*{Notation}

We put $\overline{\bb{N}}:= \bb{Z}_{\ge 0}\cup \{ \infty \}$
and $\overline{\bb{N}}_{>0 }:= \bb{Z}_{> 0}\cup \{ \infty \}$.

Let $n \in \bb{Z}_{>0}$ be any positive integer. 
We put $H_n:= \Gal(\bb{Q}(\mu_n)/\bb{Q})$. 
We define
\(
\Lambda_{/I,[n]}:=\Lambda/I[H_n],
\)
and denote by 
\(
\pi_{I,[n]} \colon \Lambda \longrightarrow 
\Lambda_{/I,[n]}
\)
the natural ring homomorphism.
For simplicity, we write 
$\pi_I:=\pi_{I,[1]}\colon \Lambda \longrightarrow 
\Lambda/I$. We also write  
$\Lambda_{[n]}:=\Lambda_{(0),[n]}=\Lambda[H_n]$ and 
$\pi_{[n]}:=\pi_{[n]}$.
The tautological action of $H_n$ 
on $\Lambda_{/I,n}$ induces the action of $H_n$ 
on $H^1(\pi_{I,[n]}^*\bb{T})$.
By Shapiro's lemma, we have a natural 
$H_n$-equivariant homomorphism
\[
H^1(\bb{Q}(\mu_n),\pi_I^* \bb{T})
\simeq H^1(\pi_{I,[n]}^*\bb{T})
\]
if $n$ is prime to $\Sigma$.

Let $R$ be a DVR,
and $v_R\colon R \longrightarrow \bb{Z} \cup \{ \infty \}$ 
the additive valuation on $R$.
Then, for any ideal $I$ of $R$
generated by an element $a \in R$,
we define $v_R(I):=v_R(a)$.

Let $S$ be a commutative ring, and $I$ an ideal of $S$.
Let $a$ be an element of $S$, 
and $C$ a subset of $\Lambda$.
Then, we denote by $a_I$ (resp.\ $C_I$) 
the image of $a$ (resp.\ C) in $S/I$.

Let $A$ be a set, and 
$\boldsymbol{a}:=(a_0, \dots, a_{r}) \in A^{r+1}$ any element.
For any $i \in \bb{Z}$ with $0 \le i \le r$, 
we define truncated systems $\boldsymbol{a}_{\le i}$ 
and $\boldsymbol{a}_{\ge i}$ by
\begin{align*}
\boldsymbol{h}_{\le i}:&=(a_0, \dots , a_{i}) \in A^{i+1}, \\
\boldsymbol{h}_{\ge i}:&=(a_i, \dots , a_{r}) \in A^{r-i+1}.
\end{align*}

\section*{Acknowledgment}

The author would like to thank 
Tadashi Ochiai for his helpful advices.
This work is supported by JSPS
KAKENHI Grant Number 26800011.

\section{Preliminaries}\label{secprel}

Here, we recall some basic notion and preliminary results.
In \S \ref{ssStr}, we recall structure theorem 
of $\Lambda$-modules, and we also recall  
the definition and basic properties of Fitting ideals.
In \S \ref{ssMPS}, we define the notion of 
monic parameter systems which is a system of 
parameters of $\Lambda$ consisting of 
``monic polynomials" in certain sense.
The quotient of $\Lambda$ by 
ideals generated by powers of elements in  
a fixed monic parameter system
have some useful properties. 
For instance, such rings become 
$0$-dimensional Gorenstein rings.
(See Lemma \ref{leminjmod}.)
In \S \ref{ssCT}, we recall some control theorems
for Galois cohomology groups.

We keep the notation introduced in \S \ref{ssintro}.
In particular, we fix an odd prime number $p$, and 
we put $\Lambda:=\mca{O}[[x_1, \dots, x_{r}]]$. 
Let $\bb{T}$ be a free $\Lambda$-module 
of finite rank with a continuous $\Lambda$-linear action of 
$G_{\bb{Q},\Sigma}$
satisfies the conditions (A1)--(A8).

\subsection{Higher Fitting ideals and structure theorem of $\Lambda$-modules}\label{ssStr}

First, let us recall the definition and basic properties 
of higher Fitting ideals.

\begin{dfn}\label{defFittid}
Let $R$ be a commutative ring, and
$M$ a finitely presented $R$-module.
Suppose that we have an exact sequence
\begin{equation}\label{finpresM}
R^m \xrightarrow{\ A \ } R^n \longrightarrow M
\longrightarrow 0
\end{equation}
of $R$-modules.
Then, for any $i \in \bb{Z}$, 
we denote by $\Fitt_{R,i}(M)$ 
the ideal of $R$ generated by 
all $(n-i)\times (n-i)$-minors of $A$.
Note that if $n-i > m$ (resp.\ $n-i \le 0$), 
then we define $\Fitt_{R,i}(M):=\{ 0\}$
(resp.\ $\Fitt_{R,i}(M)=R$). 
We call $\Fitt_{R,i}(M)$ 
\textit{the $i$-th Fitting ideal of $R$}.
Not that the ideals $\Fitt_{R,i}(M)$ is 
independent of the choice of the exact sequence 
{\rm (\ref{finpresM})}.
\end{dfn}

We shall review some basic properties 
on higher Fitting ideals briefly.
Let $R$ and $M$ be as in 
Definition \ref{defFittid}.
Then, by definition, we can verify 
the following properties easily.
\begin{enumerate}[(i)]
\item Higher Fitting ideals $\{ \Fitt_{R,i}(M) \}$
forms an ascending filtration of $R$.
\item For any ring homomorphism $f\colon R \longrightarrow R'$,
we have $\Fitt_{R',i}(f^*M)=f(\Fitt_{R,i}(M))R'$.
Namely, the higher Fitting ideals are 
{\it compatible with base change}.
\item Let $\mathrm{ann}_R(M)$ be the annihilator ideal of 
the $R$-module $M$. Then, we have 
$\Fitt_{R,i}(M) \subseteq \mathrm{ann}_R(M)$.
(This is a remarkable property of Fitting ideals
though we do not use it in this article.)
\end{enumerate}

Now let us introduce 
some important examples for higher Fitting ideals.

\begin{exa}
Let $R$ be a PID,
and $M$ a finitely generated $R$-module.
Then, by the structure theorem, 
we have an isomorphism
\[
M \simeq \bigoplus_{i=1}^s R/d_iR,
\]
where $\{d_i \}_i=1^s$ are sequence contained in 
$R \setminus R^\times$ satisfying $d_i \mid d_{i+1}$
for any $i$.
Hence by the definition of higher Fitting ideals,
we have 
\[
\Fitt_{R,i}(M)=\begin{cases}
\left(
\displaystyle \prod_{j=1}^{s-i}d_j
\right)R & (\text{if $i<s$}) \\
R & (\text{if $i \ge s$}).
\end{cases}
\]
This implies that the isomorphism class of
the $R$-module $M$ is determined by 
the higher Fitting ideals $\{ \Fitt_{R,i}(M) \}_{i \ge 0}$.
\end{exa}

\begin{exa}\label{exaFittlocring}
Let $R$ be a local ring 
with the maximal ideal $\mf{m}_R$, 
and $M$ the finitely generated $R$-module.
By the base change property of Fitting ideals, 
we have 
\[
\min\{i \in \bb{Z}_{\ge 0}
\mathrel{\vert} \Fitt_{R,i}(M)=R \}
= \dim_{R/\mf{m}_R} M \otimes_R R/\mf{m}_R.
\]
Namely, the minimal number of generators of $M$
is determined by higher Fitting ideals $\{ \Fitt_{R,i}(M) \}$.
In particular, we can easily verify that
the $R$-module $M$ is free of rank one 
if and only if $\Fitt_{R,i}(M)=0$ and $\Fitt_{R,1}(M)=R$.
\end{exa}

Now, let us consider 
the ring
$\Lambda:=\Lambda^{(r)}_{\mca{O}}=\mca{O}[[x_1, \dots, x_r]]$,
where $\mca{O}$ is the integer ring of
a finite extension field $F$ of $\bb{Q}_p$.
Note that the ring $\Lambda$ is Noetherian UFD.

First, we recall the notion of pseudo-null modules and 
pseudo-isomorphisms.
Let $f \colon M \longrightarrow N$ be 
a homomorphism of finitely generated $\Lambda$-modules.
We say that the $\Lambda$-module $M$ is pseudo-null
if and only if $M_{\mf{p}}=0$ for 
any height one prime  $\mf{p}$ of $\Lambda$.

Recall that we have the following 
structure theorem of 
finitely generated torsion
$\Lambda$-modules.

\begin{prop}\label{propstrthrm}
Let $M$ be finitely generated torsion 
$\Lambda$-module.
Then, we have a pseudo-isomorphism
\(
\iota_M \colon
M \longrightarrow \bigoplus_{i=1}^s \Lambda/d_i \Lambda
\)
of $\Lambda$-modules, where the following hold.
\begin{itemize}
\item We have $s \in \bb{Z}_{>0}$.
\item For each $ i \in \bb{Z}$ with $1 \le i \le s$, 
we have 
$d_i \in  \Lambda \setminus (\Lambda^\times \cup \{ 0 \})$.
\item For any $i,j \in \bb{Z}$ with $1 \le i < j \le s$, 
we have $d_j \in d_i \Lambda$.
\end{itemize}
\end{prop}

For the finitely generated torsion $\Lambda$-module $M$
in Proposition \ref{propstrthrm},
we define the characteristic ideal $\cha_{\Lambda}(M)$
of the $\Lambda$-module $M$ by
\[
\cha_{\Lambda}(M)=
\left(
\prod_{i=1}^s d_i 
\right)
\Lambda.
\]
Note that the higher Fitting ideals are 
independent of the choice of 
the pseudo-isomorphism $\iota_M$
in Proposition \ref{propstrthrm}.

\begin{exa}\label{exaFitt}
Let $M$ be a finitely generated torsion $\Lambda$-module, 
and $\{ d_i \}_{i \ge 0}$ the sequence 
in $\Lambda \setminus \Lambda^{\times}$
as in Proposition \ref{propstrthrm}.
Then, by the base change property 
of higher Fitting ideals, for any
any $i \in \bb{Z}_{\ge 0}$ and 
for any prime ideal $\mf{p}$ of $\Lambda$, 
we have
\[
\Fitt_{\Lambda_{\mf{p}},i}(M_{\mf{p}})=
\begin{cases}
\left(
\displaystyle \prod_{j=1}^{s-i}f_j
\right) \Lambda_{\mf{p}} & (\text{if $i<s$}) \\
\Lambda_{\mf{p}} & (\text{if $i \ge s$}).
\end{cases}
\]
Namely,
for any $i \in \bb{Z}_{\ge 0}$, 
there exists an ideal $\mf{A}_i$ of height at least two
satisfying
\[
\Fitt_{\Lambda,i}(M)=
\begin{cases}
\left(
\displaystyle \prod_{j=1}^{s-i}f_j
\right) \mf{A}_i & (\text{if $i<s$}) \\
\mf{A}_i & (\text{if $i \ge s$}).
\end{cases}
\]
In particular, the characteristic ideal $\cha_{\Lambda}(M)$
is the minimal principal ideal of $\Lambda$
containing $\Fitt_{\Lambda,0}(M)$.
The pseudo-isomorphism class of $M$
is determined by the higher Fitting ideals 
$\{ \Fitt_{\Lambda,i}(M) \}_{i \in \bb{Z}_{\ge 0}}$.
\end{exa}

\subsection{Monic parameter systems}\label{ssMPS}

Here, we use the notation introduced 
in \S Notation, namely, in the end of \S \ref{secintro}. 
For instance, 
for each $n \in \bb{Z}_{>0}$ and each ideal $I$ of $\Lambda$, 
we denote by 
\(
\pi_{I,[n]} \colon \Lambda \longrightarrow 
\Lambda_{/I,[n]}:=(\Lambda/I)[H_n]
\)
the natural map.

\begin{dfn}
Recall that we put $\overline{\bb{N}}:= \bb{Z}_{\ge 0}\cup \{ \infty \}$.
Let $\boldsymbol{h}:=(h_0, \dots , h_{r}) \in 
(\mf{m}_{\Lambda})^{r+1}$ and 
$\boldsymbol{m}:=(m_0, \dots , m_{r}), 
\boldsymbol{m}':=(m'_0, \dots , m'_{r}) \in 
(\overline{\bb{N}}_{>0})^{r+1}$
be arbitrary elements. 
\begin{enumerate}[(i)]
\item We write $\boldsymbol{m}' \ge \boldsymbol{m}$ if 
we have $m'_i \ge m_i$ for any $i$. 
\item We put 
\[
\boldsymbol{h}^{\boldsymbol{m}}:=(h_0^{m_0}, \dots h_{r}^{m_{r+1}})
\in (\mf{m}_{\Lambda})^{r}.
\]
(Here, we define $h_i^ \infty:=0$)
\item We regard $\overline{\bb{N}}_{>0}$ as a subset of
$(\overline{\bb{N}}_{>0})^{r+1}$ via the diagonal embedding.
In particular, for any integer $N$, we write 
$\boldsymbol{h}^N:=\boldsymbol{h}^{(N, \dots, N)}$.
Note that we write $N \ge \boldsymbol{m}$ if and only if 
$N \ge m_i$ for any $i$.
\item We denote by $I(\boldsymbol{h})$ 
the ideal of $\Lambda$ generated by 
$\{ h_i \mathrel{\vert} 0 \le i \le r \}$.
\end{enumerate}
\end{dfn}

Now let us introduce a notion 
\textit{monic parameter systems}, 
which plays an important role in our article.

\begin{dfn}
An element $\boldsymbol{h}:=(h_0, \dots , h_{r}) \in 
(\mf{m}_{\Lambda})^{r+1}$ is called 
\textit{a monic parameter system} of $\Lambda$
if it satisfies the following conditions (1)--(3):
\begin{enumerate}[(i)]
\item The $0$-th component $h_0$ is 
a non-zero element contained 
in $\mf{m}_{\mca{O}}=\varpi\mca{O}$.
\item For any $i \ge 1$, 
the $i$-th component $h_i$ 
is a monic polynomial in the variable $x_i$
whose coefficients of lower terms are contained in the maximal ideal of 
the ring
$\Lambda^{(i-1)}=\mca{O}[[x_1,\dots,x_{i-1}]]$.
\end{enumerate}
Note that $\boldsymbol{x}:=(\varpi, x_1, \dots , x_{r})$
forms a  monic parameter system  of $\Lambda$.
We call $\boldsymbol{x}$ \textit{the standard monic 
parameter system} of $\Lambda$.
\end{dfn}

Here, let us observe some basic properties of 
monic parameter systems.

\begin{lem}\label{lemmpr}
Let $\boldsymbol{h}$ be a monic parameter system of $\Lambda$.
Then, the following hold.
\begin{enumerate}[{\rm (i)}]
\item The $(r+1)$-ple 
$\boldsymbol{h}$ forms a system of parameters for 
the maximal ideal $\mf{m}_{\Lambda}$.
In particular, the order of 
$\Lambda/I(\boldsymbol{h}^{\boldsymbol{m}})$ is finite.
\item For any $\boldsymbol{m}:=(m_0, \dots, m_{r}) 
\in (\bb{Z}_{>0})^{r+1}$,
the $(r+1)$-ple $\boldsymbol{h}^{\boldsymbol{m}}$ 
is  a monic parameter system.
\item Let $i \in \bb{Z}$ be any integer satisfying $0 \le i \le r$, 
and put 
\[
\boldsymbol{h}(i):=(\boldsymbol{h}_{\le i-1},0, \boldsymbol{h}_{\ge i+1}).
\]
Then, there exists an exact sequence 
\[
0 \longrightarrow \pi_{I(\boldsymbol{h}(i))}^*\bb{T} 
\xrightarrow{\times h_i} \pi_{I(\boldsymbol{h}(i))}^*\bb{T} 
\longrightarrow \pi_{I(\boldsymbol{h})}^*\bb{T}
\longrightarrow 0
\]
of $\Lambda$-modules.
\end{enumerate}
\end{lem}

\begin{proof}
The assertions (i) and (ii) immediately follows from 
the definition of monic parameter systems.
Let $i \in \bb{Z}_{\ge 0}$ be any element. 
Since $\boldsymbol{h}=(h_i)_i$ is a monic parameter system, 
the $\Lambda^{(i)}/I(\boldsymbol{h}_{\le i-1})$-module
$\pi_{I(\boldsymbol{h}(i))}\bb{T}$ 
is free of finite rank. 
Moreover, since $\boldsymbol{h}$ is a monic parameter system,
the image of $h_i$ in $\Lambda^{(i)}/I(\boldsymbol{h}_{\le i-1})$ 
is not a zero divisor.
This implies that the scalar multiplication map
\(
\times h_i \colon \pi_{I(\boldsymbol{h}(i))}^*\bb{T} 
\longrightarrow \pi_{I(\boldsymbol{h}(i))}^*\bb{T}
\)
is injective.
Hence the assertion (iii) follows.
\end{proof}

Let $\boldsymbol{h}$ be a monic parameter system of $\Lambda$, 
and $\boldsymbol{m}=(m_0, \dots m_{r}) 
\in (\bb{Z}_{>0})^{r+1}$ any element.
Take any element $n \in \mca{N}_{\Sigma}$.
Then, by definition, the ring 
$\Lambda_{/I(\boldsymbol{h}^{\boldsymbol{m}}),[n]}$
becomes locally complete intersection.
In particular,  the ring 
$\Lambda_{/I(\boldsymbol{h}^{\boldsymbol{m}}),[n]}$
is Gorenstein.
Since 
$\Lambda_{/I(\boldsymbol{h}^{\boldsymbol{m}}),[n]}$
is a finite direct product of 
$0$-dimensional local rings, we obtain 
the following lemma which
becomes a key of some ring theoretic  arguments
in our article.

\begin{lem}\label{leminjmod}
The $\Lambda_{/I(\boldsymbol{h}^{\boldsymbol{m}})
,[n]}$-module $\Lambda_{/I(\boldsymbol{h}^{\boldsymbol{m}}),[n]}$
is injective.
\end{lem}

\begin{rem}
It is convenient to consider 
the notion of ``monic parameter systems"
to construct systems of parameters for the local ring $\Lambda$
systematically.
Note that 
Lemma \ref{lemmpr} (iii) and Lemma \ref{leminjmod}
are benefits of the fact that  
a monic parameter systems forms 
systems of parameters for $\Lambda$.
For instance, if we use a power of the maximal ideal of $\Lambda$
instead of $I(\boldsymbol{h}^{\boldsymbol{m}})$,
similar assertion to Lemma \ref{leminjmod} does not hold.
\end{rem}

\subsection{Control theorems of Galois cohomology groups}\label{ssCT}

Here, we introduce some preliminary results 
which are related to Iwasawa theoretic reduction arguments.
Let us consider the $\Lambda[\Gal(\bb{Q}_\Sigma/\bb{Q})]$-module $\bb{T}$ 
satisfying all the conditions in Theorem \ref{thmuncondresults}.
First, we give a description of our Selmer group $X=X(\bb{T})$
in terms of the Galois cohomology group of the second degree.

\begin{lem}\label{lemXH2}
We have a natural isomorphism
$X \simeq H_{\Sigma}^2(\bb{T})$
of $\Lambda$-modules.
\end{lem}

\begin{proof}
By the assumption (A6), we have 
\[
\mathrm{Sel}_p(\bb{Q},\mca{A}^*)= 
\ker \left(
H_{\Sigma}^1(\mca{A}^*)
\longrightarrow
\prod_{\ell \in \Sigma}
H^1(\bb{Q}_\ell,\mca{A}^*)
\right).
\]
So, by (the projective limit of) 
the Poitou--Tate exact sequences,
we obtain 
\[
X \simeq \ker \left(
H_{\Sigma}^2(\bb{T})
\longrightarrow \prod_{\ell \in \Sigma}
H^2(\bb{Q}_\ell,\bb{T})
\right).
\]
By the local duality theorem of 
Galois cohomology groups and 
the assumptions (A6) and (A7) 
imply $H^2(\bb{Q}_\ell,\bb{T})=0$
for any $\ell \in \Sigma$.
Hence we obtain our proposition.
\end{proof}

Let $\boldsymbol{h} \in \Lambda^{r+1}$ be a monic parameter system, 
and $\boldsymbol{m} \in \bb{Z}_{> 0}^{r+1}$ 
any element. 
We put $I:=I(\boldsymbol{h}^{\boldsymbol{m}})$.
By the similar arguments to that in the proof of 
\cite{Ne} (8.4.8.1) Proposition, 
we obtain the following proposition.

\begin{prop}\label{propspecseq}
We have a spectral sequence
\[
E_2^{p,q}:= \Tor^{\Lambda}_{-p}(H_{\Sigma}^q(\bb{T}),\Lambda/I)
\Longrightarrow H_{\Sigma}^{-p+q}(\pi^*_{I} \bb{T}).
\]
\end{prop}

\begin{rem}
In \cite{Ne} (8.4.8.1) Proposition, 
Nekov\'ar obtained the spectral sequence
in the case when $\Lambda$ is an Iwasawa algebra 
of a Galois group $\Gamma \simeq \bb{Z}_p^{r}$, and 
$I$ is the augmentation ideal of $\Lambda$.
But,  by Lemma \ref{lemmpr} (iii), 
similar arguments work in our situation.
\end{rem}

By Lemma \ref{lemXH2} and 
Proposition \ref{propspecseq},
we immediately obtain the following corollary.

\begin{cor}[Control Theorem]\label{corcontthm}
The following hold.
\begin{enumerate}[{\rm (i)}]
\item We have an exact sequence
\[
\hspace{-20mm}
\Tor^{\Lambda}_{2}(H_{\Sigma}^2(\bb{T}),\Lambda/I)
\longrightarrow H_{\Sigma}^1(\pi^* \bb{T}) 
\otimes_{\Lambda} \Lambda/I \]
\[\hspace{20mm}
\longrightarrow H_{\Sigma}^1(\pi_{I}^* \bb{T})
\longrightarrow 
\Tor^{\Lambda}_{1}(H_{\Sigma}^2(\bb{T}),\Lambda/I).
\]
\item We have a natural homomorphism
\[
X(\bb{T}) \otimes_{\Lambda} \Lambda/I \simeq 
X(\pi^*_I \bb{T}).
\]
\end{enumerate}
\end{cor}

\section{Euler systems}\label{secES}

Let $\bb{T}$ be as in Theorem \ref{thmuncondresults}.
In this section, we recall 
the definition and some basic preliminary results on 
Euler systems and Kolyvagin systems.
In \S \ref{SSec_defES}, 
we introduce the definition of Euler systems
for Galois deformations.
In \S \ref{ssKD}, 
we recall the definition and basic properties of 
Kolyvagin derivatives.
In \S \ref{ssKS}, we recall the theory of 
Kolyvagin systems over DVR 
established by Mazur and Rubin in \cite{MR}.
In \S \ref{ssunivKS}, we introduce a notion of 
universal Kolyvagin system, which is a system of 
linear combinations of Kolyvagin derivatives 
whose specialization to DVR forms 
a Kolyvagin system.

\subsection{Definition}\label{SSec_defES}
Here, we recall the notion of Euler systems 
for Galois deformations, 
which is a generalization of 
the Euler systems in the sense of \cite{MR}. 
(We adopt the axiom of Euler systems
slightly different from Ochiai's one in \cite{Oc2}.)
For any $\sigma \in G_{\bb{Q}, \Sigma}$, 
we define a polynomial $P(x;\bb{T} \vert \sigma)$ by 
\[
P(x;\bb{T}\vert \sigma ):=
\det_{\Lambda}(1- x\rho_{\bb{T}}(\sigma); \bb{T}),
\] 
where $\Frob_\ell \in  \Gal(\bb{Q}_{\Sigma}/\bb{Q}) $ 
is an arithmetic Frobenius element at $\ell$.
The definition of Euler system 
in this paper is as follows.

\begin{dfn}\label{dfn_ES}
We define the set $\mca{N}_{\Sigma}$ by
\[
\mca{N}_{\Sigma}:=
\left\{
n \in \bb{Z}_{>0} \mathrel{\vert}
\text{$n$ is square free and prime to $\Sigma$}
\right\}.
\]
We call a collection 
\[
\boldsymbol{c}:= \left\{
c(n) \in  H^1(\bb{Q}(\mu_n),\bb{T}) 
\mathrel{\vert}
n \in \mca{N}_{\Sigma}
\right\}
\]
of Galois cohomology classes 
{\em an Euler system} for $\bb{T}$
if the collection $\boldsymbol{c}$ satisfies the following conditions. 
\begin{enumerate}[{\rm (i)}]
\item For any $n \in \Sigma$, the Galois cohomology class
$c(n)$ is unramified outside $p$.
\item Let $\ell$ be any prime number not contained in $\Sigma$,
and $n \in \mca{N}_{\Sigma}$ any integer prime to $\ell$.
Then, we have 
\[
\mathrm{Cor}_{\bb{Q}(\mu_{n\ell})/\bb{Q}(\mu_n)}(c(n\ell))
=P(\Frob_\ell^{-1};\bb{T} \vert \Frob_\ell)c(n),
\]
where $\mathrm{Cor}_{\bb{Q}(\mu_{n\ell})/\bb{Q}(\mu_n)}$ is
the corestriction map of Galois cohomology.
\end{enumerate}
\end{dfn}

\begin{rem}\label{remES}
Our Euler system axiom is slightly different from 
that in \cite{Ru} or \cite{Oc2}.
Indeed, Euler systems $\boldsymbol{z}=\{z(n) \in
H^1(\bb{Q}(\mu_n),\bb{T})\}_n$
in the sense of \cite{Oc2} satisfies the following 
condition (ii)' instead of (ii) in Definition \ref{dfn_ES}.
\begin{itemize}
\item[(ii)'] Let $\ell$ be any prime number not contained in $\Sigma$,
and $n \in \mca{N}_{\Sigma}$ any integer prime to $\ell$.
Then, we have 
\[
\mathrm{Cor}_{\bb{Q}(\mu_{n\ell})/\bb{Q}(\mu_n)}(z(n\ell))
=P(\Frob_\ell^{-1};\bb{T}^* \vert \Frob_\ell^{-1})z(n).
\]
\end{itemize}
Here, we put $\bb{T}^*:=\Hom_{\Lambda}(\bb{T},\Lambda) 
\otimes_{\bb{Z}_p} \varprojlim_m \mu_{p^m}$.
We prefer the axiom (ii) to (ii)' 
since (ii)' is useful 
to apply the theory of Kolyvagin systems.
By Proposition \ref{propESmodify} below, 
we can construct an Euler system in our sense
by a canonical way
when an Euler system in the sense of \cite{Oc2} is given.
\end{rem}

\begin{prop}[\cite{Ru} Lemma 9.6.1 and Corollary 9.6.4]\label{propESmodify}
Let 
$\boldsymbol{z}=\{z(n) \in
H^1(\bb{Q}(\mu_n),\bb{T})\}_{n \in \mca{N}_\Sigma}$
be a collection of continuous Galois cohomology classes satisfying
{\rm (i)} in Definition \ref{dfn_ES} and 
{\rm (ii)'} in Remark \ref{remES}.
For each $n \in \mca{N}_{\Sigma}$ and 
each divisor $d$ of $n$ satisfying $d>1$,
we define an element 
$A(n,d;\bb{T}) \in \Lambda[\Gal(\bb{Q}(\mu_n)/\bb{Q}) ]$
by 
\[
A(n,d;\bb{T}):= \prod_{\ell \in \mathrm{Prime(n/d)}}
\frac{P(\Frob_\ell^{-1};\bb{T} \vert \Frob_\ell)-
P(\Frob_\ell^{-1};\bb{T}^* \vert \Frob_\ell^{-1})}{\varphi(n/d)} , 
\]
where $\varphi$ denotes Euler's $\varphi$ function.
Then, the system 
\[
\boldsymbol{c}:= \left\{
c(n):=z(n)+ \sum_{1< d \mid n} A(n,d;\bb{T})z(d)
 \in  H^1(\bb{Q}(\mu_n),\bb{T}) 
\mathrel{\bigg\vert}
n \in \mca{N}_{\Sigma}
\right\}
\]
forms an Euler system for $\bb{T}$
in the sense of Definition \ref{dfn_ES}.
\end{prop}

\begin{dfn}\label{dfnESalg}
Let $\mca{R}$ be a topological $\Lambda$-algebra
with  a structure map
$\pi \colon \Lambda \longrightarrow \mca{R}$.
\begin{enumerate}[{\rm (i)}]
\item If a collection 
\[
\boldsymbol{c}:= \left\{
c(n) \in  H^1(\bb{Q}(\mu_n),\bb{T}) 
\mathrel{\vert}
n \in \mca{N}_{\Sigma}
\right\}
\]
satisfies conditions (1) and (2)
in Definition \ref{dfn_ES}, 
we call $\boldsymbol{c}$ 
{\em an Euler system} for $\pi^*\bb{T}$
\item For any Euler system $\boldsymbol{c}$ for $\bb{T}$, 
we can define an Euler system $\boldsymbol{c}$ for $\pi^*\bb{T}$ by
\[
\pi^*\boldsymbol{c}:= \left\{
c(n) \in  H^1(\bb{Q}(\mu_n),\pi^*\bb{T}) 
\mathrel{\vert}
n \in \mca{N}_{\Sigma}
\right\}.
\]
\end{enumerate}
\end{dfn}

We consider the cyclotomic $\bb{Z}_p$-extension $\bb{Q}_\infty/\bb{Q}$, 
and put $\Gamma:=\Gal(\bb{Q}_\infty/\bb{Q})$.
Let $\chi_{\mrm{taut}}\colon  \Gal(\bb{Q}_{\Sigma}/\bb{Q}) 
\longrightarrow \Gamma \subseteq \Lambda^\times$ 
be the tautological character.
We define a $\Lambda[[\Gamma]]$-module 
$\bb{T}^{\mrm{cyc}}:=\bb{T} \otimes_{\Lambda} \Lambda[[\Gamma]]$
on which $ \Gal(\bb{Q}_{\Sigma}/\bb{Q}) $ acts via 
$\rho_{\bb{T}}\otimes \chi_{\mrm{taut}}$.
The augmentation $\Lambda[[\Gamma]] \longrightarrow \Lambda$
induces a $ \Gal(\bb{Q}_{\Sigma}/\bb{Q}) $-equivariant map
\[
\mathrm{aug}_{\bb{T}}\colon 
\bb{T}^{\mrm{cyc}} \longrightarrow \bb{T}.
\]

\begin{dfn}\label{dfncycdir}
In this article, we say that an Euler system 
$\boldsymbol{c}:=\{ c(n) \}_n$ for $\bb{T}$
{\em can be extended to the cyclotomic direction} 
if there exists an Euler system 
$\tilde{\boldsymbol{c}}:=\{ \tilde{c}(n) \}_n$ 
for $\bb{T}^{\mrm{cyc}}$ such that 
$\mathrm{aug}_{\bb{T}}(\tilde{\boldsymbol{c}})
:=\{ \mathrm{aug}_{\bb{T}}(\tilde{c}(n)) \}_n$
coincides with $\boldsymbol{c}$.
The Euler system $\tilde{\boldsymbol{c}}$ is called 
an extension of $\boldsymbol{c}$ to the cyclotomic direction.
\end{dfn}

\begin{rem}
Let $\mca{R}$ be a topological $\Lambda$-algebra. 
If an Euler system $\boldsymbol{c}$ 
for $\mca{R}\otimes_{\Lambda}\bb{T}$ satisfies
similar conditions to that in Definition \ref{dfncycdir},
we say that $\boldsymbol{c}$ 
can be extended to the cyclotomic direction. 
\end{rem}

\begin{lem}\label{lempropcycext}
Let $\boldsymbol{c}$ be 
an Euler system for $\bb{T}$.
\begin{enumerate}[{\rm (i)}]
\item Let $\mca{R}$ be any topological $\Lambda$-algebra,
and $\pi \colon \Lambda \longrightarrow \mca{R}$
the structure map of the $\Lambda$-algebra $\mca{R}$.
If an Euler system $\boldsymbol{c}$
for $\bb{T}$ 
can be extended to the cyclotomic direction,
then $\pi^*\boldsymbol{c}$
can be extended to the cyclotomic direction.
\item Let $\tilde{\boldsymbol{c}}$ be 
an Euler system for $\widetilde{\bb{T}}$
satisfying $\mathrm{aug}_{\bb{T}}
(\tilde{\boldsymbol{c}})=\boldsymbol{c}$.
Then, $\tilde{\boldsymbol{c}}$ can be 
extended to the cyclotomic direction.
\end{enumerate}
\end{lem}

\begin{proof}
First, Let us show the first assertion.
The map $\pi$ induces 
a continuous homomorphism 
$\tilde{\pi}\colon \Lambda[[\Gamma]]
\longrightarrow \mca{R}[[\Gamma]]$
Note that we have a natural isomorphism
$\tilde{\pi}^*\bb{T}^{\mathrm{cyc}}
\simeq (\pi^*\bb{T})^\mathrm{cyc}$.
By definition, 
the Euler system $\tilde{\pi}^*(\tilde{\boldsymbol{c}})$ 
for $(\pi^*\bb{T})^{\mathrm{cyc}}$ 
satisfies 
\[
\mathrm{arg}_{\pi^*\bb{T}}(\tilde{\pi}^*(\tilde{\boldsymbol{c}}))
=\pi^*\boldsymbol{c}.
\]
So $\pi^*\boldsymbol{c}$
can be extended to the cyclotomic direction.

Next, Let us show the second assertion.
The diagonal embedding 
$\Delta\colon \Gamma \longrightarrow \Gamma \times \Gamma$
induces a continuous homomorphism
\[
\Delta\colon \Lambda[[\Gamma]]
\longrightarrow \Lambda[[\Gamma \times \Gamma]]
\]
of $\Lambda$-algebras.
We define an isomorphism 
$\iota \colon \Lambda[[\Gamma \times \Gamma]] 
\xrightarrow{\ \simeq\ } \Lambda[[\Gamma]]$
of topological rings by composite
\begin{align*}
\Lambda[[\Gamma \times \Gamma]]
& \simeq \varprojlim_{n_1, n_2}
\Lambda \left[
\left(
\Gamma/\Gamma^{p^{n_1}}
\right) \times 
\left(
\Gamma/\Gamma^{p^{n_2}}
\right)
\right] \\
& \simeq \varprojlim_{n_2}\left(
\varprojlim_{n_1}\Lambda[\Gamma/\Gamma^{p^{n_1}}]
\right)[\Gamma/\Gamma^{p^{n_2}}] \\
& \simeq 
(\Lambda[[\Gamma]])[[\Gamma]].
\end{align*}
Then we have a commutative diagram
\begin{equation}\label{diagdiagcyclot}
\xymatrix{
& (\Lambda[[\Gamma \times \Gamma]] \ar[d]^{\mathrm{pr}_1} 
\ar[r]^{\iota}_{\simeq} & 
(\Lambda[[\Gamma]])[[\Gamma]] \ar[d]^{\mathrm{aug}} \\
\Lambda[[\Gamma]] \ar[ur]^{\Delta}
\ar@{=}[r] & \Lambda[[\Gamma]]
\ar@{=}[r] & \Lambda[[\Gamma]],
}
\end{equation}
where $\mathrm{aug}$ is the augmentation map, and
$\mathrm{pr}_1$ is the ring homomorphism
induced by the first projection map
$\Gamma \times \Gamma \longrightarrow \Gamma$.
By th commutative diagram (\ref{diagdiagcyclot}), 
we can immediately verify that
the Euler system 
\[
\tilde{\tilde{\boldsymbol{c}}}:=
(\iota \circ \Delta)^*(\tilde{\boldsymbol{c}})
\] 
for $\left((\bb{T})^{\mathrm{cyc}}\right)^{\mathrm{cyc}}$ 
satisfies 
\[
\mathrm{arg}_{\bb{T}^{\mathrm{cyc}}}\left(
\tilde{\tilde{\boldsymbol{c}}}
\right)
=\tilde{\boldsymbol{c}}.
\]
This completes the proof of Lemma \ref{lempropcycext}.
\end{proof}

\subsection{Kolyvagin derivatives}\label{ssKD}
Here, we fix an Euler system 
$\boldsymbol{c}=\{ c(n) \}_n$ for $\bb{T}$.
Let us recall the construction of
Kolyvagin derivative arising from the Euler system
$\boldsymbol{c}$.

First, we introduce some notation
in general setting.
Let $R$ be any pro-finite commutative ring, and
$T$ a free $R$ module of finite rank
with continuous $R$-linear $G_{\bb{Q},\Sigma}$-action $\rho_T$.
Let $I$ be an ideal of $R$ of finite index.
For any prime number $\ell \notin \Sigma$,
we set the ideal $\mca{I}_{T,\ell }$ of $R$ by
$\mca{I}_{\bb{T},\ell }:=(\ell -1, P_\ell(1;\bb{T}))$,
where for any prime number $\ell \notin \Sigma$,
we define a polynomial $P_\ell(x;T)$ by 
\[
P_\ell(x;T):=
\det_{R}(1- x\rho_{T}(\Frob_\ell); T).
\] 
For any $n \in \mca{N}_{\Sigma}$, 
we define the ideal 
$\mca{I}_{T,n}$ of $\Lambda$ by 
\[
\mca{I}_{T,n}:= \sum_{\ell \in \mathrm{Prime}(n)}
\mca{I}_{T,\ell}.
\]
We define a set $\mca{P}(T;I)$ of prime numbers by
\begin{equation}\label{eqPTI}
\mca{P}(T;I):= \left\{
\ell \mathrel{\bigg\vert}
\begin{array}{l}
\text{$\ell \notin \Sigma$, 
$\mca{I}_{T,\ell} \subseteq I$, and the 
$\Lambda/I$-module} \\
\text{$T/(IT+(\Frob_\ell -1)T)$ 
is free of rank one}
\end{array}
\right\},
\end{equation}
and define a set $\mca{N}(T,I)$ of positive integers by 
\[
\mca{N}(T,I):= \left\{
n \in \bb{Z}_{>0} \mathrel{\bigg\vert}
\begin{array}{l}
\text{$n$ is a square free integer, and all prime } \\
\text{divisors of $n$  
are contained in $\mca{P}(T,I)$}
\end{array}
\right\}.
\]
For simplicity, when $(R,T)=(\Lambda,\bb{T})$,
we write $\mca{I}_n:=\mca{I}_{\bb{T},n}$.

\begin{rem}
Let $I$ be any ideal of $\Lambda$ of finite index.
Here, we give some remarks on the sets
$\mca{P}(\bb{T},I)$ and $\mca{N}(\bb{T},I)$.
\begin{enumerate}[(i)]
\item The assumption (A2) and 
the Chebotarev density theorem 
imply that the sets $\mca{P}(\bb{T},I)$ and 
$\mca{N}(\bb{T},I)$ are not empty.
\item Let $\ell$ be an integer not contained in $\Sigma$. 
Suppose that $(\ell-1) \in I$, and  
the $\Lambda/I$-module $\bb{T}/(I\bb{T}+(\mathrm{Frob}_\ell-1)\bb{T})$
is free of rank one. 
Then, we have $\ell \in \mca{P}(\bb{T},I)$.
\item We note that $1 \in \mca{N}(\bb{T},I)$.
\end{enumerate} 
\end{rem}

Let $\ell \in \mca{P}(\bb{T},I)$ be any element.
and fix a generator $\sigma_\ell$ of 
$H_\ell \simeq (\bb{Z}/\ell\bb{Z})^\times$. 
We define an element $D_\ell \in \bb{Z}[H_n]$ by
\[
D_\ell := \sum_{\nu=1}^{\ell-2}\nu \sigma_\ell ^\nu.
\]

Let $n \in \mca{N}(\bb{T},I)$ be any element.
We denote the set of prime numbers 
dividing $n$ by $\mathrm{Prime}(n)$.
Then, we have 
$H_n=\Gal(\bb{Q}(\mu_n)/\bb{Q})
\simeq \prod_{\ell \in \mathrm{Prime}(n)} H_{\ell}$.
We put
\[
H_n^\otimes:=
\bigotimes_{\ell \in \mathrm{Prime}(n)} H_{\ell},
\]
where tensor products are taken over $\bb{Z}$.
We define 
\[
D_n := \prod_{\ell \in \mathrm{Prime}(n)} D_{\ell} \in \bb{Z}[H_n].
\]
By the similar arguments to 
the proof of \cite{Ru} Lemma 4.4.2, we deduce that
\begin{eqnarray}\label{Dnfixed}
(1-\sigma)D_n c(n)  \in 
IH^1(\pi_{I,[n]}^*\bb{T})
\end{eqnarray}
for any $\sigma \in H_n$. 
So, we obtain
\[
D_n c(n)_I \in H^1(\pi_{I,[n]}^*\bb{T})^{H_n},
\]
where $c(n)_I$ is the image of $c(n)$
in $H^1(\pi_{I,[n]}^*\bb{T})$.
By the assumptions (A1) and (A3), 
the restriction map 
\[
H^1(\pi_{I,[n]}^*\bb{T}) \longrightarrow 
H^1(\pi_{I,[n]}^*\bb{T})^{H_n}
\]
becomes an isomorphism.
Hence we obtain the following definition.

\begin{dfn}\label{dfnKD}
We denote by $\kappa(\boldsymbol{c};n)_I$ 
the unique element of 
$H^1(\pi_I^*\bb{T}) \otimes_{\bb{Z}} H_{n}^\otimes$
whose image by the restriction map 
\begin{eqnarray}\label{resisom}
H^1(\pi_I^*\bb{T})\otimes_{\bb{Z}} 
H^\otimes_n \longrightarrow 
H^1(\pi_{I,[n]}^*\bb{T})^{H_n}
\otimes_{\bb{Z}} H^\otimes_n
\end{eqnarray}
coincides with $D_n c(n)_I
\otimes \bigotimes_{\ell \in \mathrm{Prime}(n)} \sigma_{\ell}$.
Note that the element $\kappa(\boldsymbol{c};n)_I$ is
independent of the choice of generators 
$\sigma_\ell \in H_\ell$.
The element $\kappa(\boldsymbol{c};n)_I$ is called
{\em the Kolyvagin derivative} 
of the Euler system $\boldsymbol{c}$.
\end{dfn}

\subsection{Mazur-Rubin theory over DVR}\label{ssKS}

Let us recall the theory of Kolyvagin systems 
over DVR, namely 
the case when $\Lambda=\Lambda^{(0)}$,
established by Mazur and Rubin in the book \cite{MR}. 

First, let us recall basic settings in the book \cite{MR}.
Let $R$ be the integer ring of 
a finite extension field of $\bb{Q}_p$.
We denote the maximal ideal of $R$ by $\mf{m}_R$.
Let $T$ be a free $R$ module of finite rank
with continuous $R$-linear $G_{\bb{Q},\Sigma}$-action.
We put $A^*:=\Hom_{\bb{Z}_p}(T,\mu_{p^\infty})$.
Let $\mca{F}_{\mathrm{can}}$ be 
the canonical local condition for $T$
in the sense of \cite{MR}, namely 
\[
H^1_{\mathcal{F}_{\mathrm{can}}}(\bb{Q}_v,T)=
\begin{cases}
H^1(\bb{Q}_v,T) & (v \in \{p, \infty\}) \\
H^1_f(\bb{Q}_v,T) & (v \notin \{ p,\infty \})
\end{cases}
\]
in our setting,
where $H^1_f(\bb{Q}_v,T)$ denotes  
Bloch--Kato's finite local condition.
We denote by $\mca{F}^*_{\mathrm{can}}$ 
the \textit{dual} local condition for $\mca{F}_{\mathrm{can}}$ 
in the sense of \cite{MR} Definition 2.3.1.
Let $\mca{P}$ be a set of prime numbers
contained in $\mca{P}(T,\mf{m}_R)$.
Here, the set $\mca{P}(T,\mf{m}_R)$
is defined by similar manor to (\ref{eqPTI}).
We denote by $\mca{N}(\mca{P})$ the set of positive integers
which are square free products of several prime numbers in $\mca{P}$.
We assume that the triple $(T,\mca{F}_{\mathrm{can}},\mca{P})$
satisfies the hypotheses (H1)--(H6) in \cite{MR}.
Mazur and Rubin introduced the notion of 
{\it the core rank} $\chi (T):=\chi(T,\mca{F}_{\mathrm{can}},\mca{P}) 
\in \bb{Z}_{\ge 0} \cup \{ \infty \}$ 
of the triple $(T,\mca{F}_{\mathrm{can}},\mca{P})$.
(See \cite{MR} Definition 4.1.11 and  Definition 5.2.4.)
In our situation, the core rank $\chi(T)$ is 
given by the formula
\[
\chi(T)= \mathrm{rank} T^{-} + \mathrm{corank}H^0(\bb{Q}_p,A^*),
\]
where $T^-$ is the maximal $R$-submodule of $T$
where the complex conjugate acts via 
the scalar multiplication by $-1$.
(See \cite{MR} Theorem 5.2.15.)

\begin{exa}
Let $\bb{T}$ be as in Theorem \ref{thmuncondresults}.
In particular, we assume that 
$\bb{T}$ satisfies conditions (A1)--(A8).
Let $f \colon \Lambda \longrightarrow R$ 
be any continuous ring homomorphism.
Here, we put $T:=f^*\bb{T}$.
Take any $n \in \bb{Z}_{>0}$, 
and put $\mca{P}:=\mca{P}(T,\mf{m}_R^n)$.
Then, we can easily verify that the triple 
$(T,\mca{F}_{\mathrm{can}},\mca{P})$
satisfies hypotheses (H1)--(H6), and 
$\chi(T)=1$. Moreover, we have a canonical isomorphism
\[
\Hom_{\bb{Z}_p}(H^1_{\mca{F}_{\mathrm{can}^*}}(\bb{Q},A^*)
,\bb{Q}_p/\bb{Z}_p) \simeq
X(T).
\]
\end{exa}

Now let us recall the definition of Kolyvagin systems.
Let $(T,\mca{F}_{\mathrm{can}},\mca{P})$ be as above.
For each $n \in \mca{P}$,
We define a new local condition $\mca{F}_{\mathrm{can}}(n)$ by
\[
H^1_{\mca{F}_{\mathrm{can}}(n)}(\bb{Q}_\ell,T/I_nT)
=
\begin{cases}
H^1_{\mca{F}_{\mathrm{can}}}(\bb{Q}_\ell,T/I_nT)
& (\text{if $\ell \nmid n$}) \\
\Ker \left(
H^1(\bb{Q}_\ell,T/I_nT) \rightarrow 
H^1(\bb{Q}_\ell(\mu_\ell),T/I_nT) 
\right)
& (\text{if $\ell \mid n$}).
\end{cases}
\]

\begin{dfn}[\cite{MR} Definition 3.1.3]
A collection
\[
\boldsymbol{\kappa}=\{
\kappa_n \in H^1(T/I_nT) \otimes_{\bb{Z}} H_n^{\otimes}
\}_{n \in \mca{N}(\mca{P})}
\]
is called a Kolyvagin system for the triple
$(T,\mca{F}_{\mathrm{can}},\mca{P})$
if it satisfies the following two conditions: 
\begin{enumerate}[{\rm (i)}]
\item For any $n \in \mca{P}$, 
we have 
$\kappa_n \in H^1_{\mca{F}_{\mathrm{can}}(n)}
(\bb{Q},T/I_nT)$.
\item Take any $n \in \mca{N}(\mca{P})$,
and let $\ell \in \mca{P}$ be an element prime to $n$.
We put 
\[
H^1_s(\bb{Q}_\ell,T/I_nT):=H^1(\bb{Q}_\ell,T/I_nT)/
H^1_f(\bb{Q}_\ell,T/I_nT).
\]
Then, we have 
\[
\psi^{\overline{n,n\ell}}_{n\ell}(\kappa_{n\ell})=
\psi^{\overline{n,n\ell}}_n(\kappa_{n})
\in 
H^1_s(\bb{Q}_\ell,T/I_nT)\otimes_{\bb{Z}} H_{n\ell}^{\otimes}
\]
where 
\begin{align*}
\psi^{\overline{n,n\ell}}_{n\ell} \colon 
H^1_{\mca{F}_{\mathrm{can}}(n\ell)}
(\bb{Q}, T/I_{n\ell}T)\otimes_{\bb{Z}} H_{n\ell}^{\otimes}
&\longrightarrow H^1(\bb{Q}_\ell,T/I_{n\ell}T)
\otimes_{\bb{Z}} H_{n\ell}^{\otimes} \\
&\longrightarrow H^1_s(\bb{Q}_\ell,T/I_{n\ell}T)
\otimes_{\bb{Z}} H_{n\ell}^{\otimes}
\end{align*}
is a composite of the localization 
(namely, restriction) map and the natural surjection, 
and  
\begin{align*}
\psi^{\overline{n,n\ell}}_n \colon
H^1_{\mca{F}_{\mathrm{can}}(n)}(\bb{Q}, T/I_nT)\otimes_{\bb{Z}} H_n^{\otimes}
& \longrightarrow 
H^1_f(\bb{Q}_\ell,T/I_nT)\otimes_{\bb{Z}} H_n^{\otimes} \\
& \longrightarrow 
H^1_s(\bb{Q}_\ell,T/I_nT)\otimes_{\bb{Z}} H_{n\ell}^{\otimes}
\end{align*}
is the composite of the localization map and 
the finite singular comparison map (\cite{MR} Definition 1.2.2).
\end{enumerate}
We denote the set of Kolyvagin systems for
$(T,\mca{F}_{\mathrm{can}},\mca{P})$ by
$\mathbf{KS}(T,\mca{F}_{\mathrm{can}},\mca{P})$.
The set $\mathbf{KS}(T,\mca{F}_{\mathrm{can}},\mca{P})$
has a natural $R$-module structure.
\end{dfn}

The following are main results on Kolyvagin systems (of rank one) 
over DVR proved in \cite{MR} \S 5.2.

\begin{thm}[\cite{MR} Theorem 5.2.10]\label{thmMRuniqueness}
Suppose that $\chi(T)=1$.
Then, the $R$-module
$\mathbf{KS}(T,\mca{F}_{\mathrm{can}},\mca{P})$
is free of rank one.
\end{thm}

Let 
$\boldsymbol{\kappa}=\{\kappa_n\}_n
\in \mathbf{KS}(T,\mca{F}_{\mathrm{can}},\mca{P})$
be any Kolyvagin system.
For each $n \in \mca{N}(\mca{P})$, we define
\[
\partial(\boldsymbol{\kappa};n)
:=
\max\{
j \in \bb{Z}_{\ge 0} \mathrel{\vert}
\kappa_n \in \mf{m}_R^jH^1_{\mca{F}_{\mathrm{can}}(n)}(\bb{Q},T/I_nT)
\}.
\]
Let $t \in \bb{Z}_{>0}$, and put
$\mca{P}(t):=\{
\ell \in \mca{P} \mathrel{\vert} I_\ell \subseteq \mf{m}_R^t
\}$.
Then, for each $i \in \bb{Z}_{\ge 0}$, 
we put
\[
\partial_i(\boldsymbol{\kappa})_t:=
\min\{\partial(\boldsymbol{\kappa};n)
\mathrel{\vert} \# \mathrm{Prime}(n)=i,\ 
\mathrm{Prime}(n) \subseteq \mca{P}(t) \}.
\]
In particular,
we put $\partial_0(\boldsymbol{\kappa})_t:=
\partial(\boldsymbol{\kappa};1)_t$.

\begin{thm}[\cite{MR} Theorem 5.2.12]\label{thmMRstrthm}
Let $t \in \bb{Z}_{>0}$ be any positive integer.
Suppose that $\chi(T)=1$.
Let $\boldsymbol{\kappa}=\{\kappa_n\}_n$ be
a generator of the $R$-module
$\mathbf{KS}(T,\mca{F}_{\mathrm{can}},\mca{P})$.
Then, for any $i \in \bb{Z}_{\ge 0}$, we have
\[
\Fitt_{R,i}(\Hom_{\bb{Z}_p}(H^1_{\mca{F}_{\mathrm{can}^*}}(\bb{Q},
A^*),\bb{Q}_p/\bb{Z}_p))
=\mf{m}_R^{\partial_i(\boldsymbol{\kappa})_t}.
\]
\end{thm}

\subsection{The universal Kolyvagin system}\label{ssunivKS}
Here, by using Kolyvagin derivatives
of a fixede Euler system $\boldsymbol{c}$, 
we shall construct 
a collection of Galois cohomology classes  
called a ``universal Kolyvagin system",
whose specializations to DVRs become Kolyvagin systems.
From now on, we assume that 
the Euler system $\boldsymbol{c}$
extends to cyclotomic direction. 
Fix an Euler system $\widetilde{\boldsymbol{c}}
=\{ \tilde{c}(n) \}_{n \in \mca{N}_{\Sigma}}$
on $\bb{T}^{\mathrm{cyc}}$
which is an extension of $\boldsymbol{c}$ to the cyclotomic direction.

We use the following notation.
Recall that we put $\Gamma:=\Gal(\bb{Q}_\infty / \bb{Q})$.
We denote by $\mf{a}_{\Gamma}$ the augmentation ideal of 
the completed group ring $\Lambda[[\Gamma]]$.
Let $\tilde{I}$ be an ideal of $\Lambda[[\Gamma]]$, 
and $n \in \mca{N}_{\Sigma}$ any element.
Then, we put
$\Lambda[[\Gamma]]_{\tilde{I},[\ell]}
:=(\Lambda[[\Gamma]]/\tilde{I})[H_\ell]$, and let
\(
\tilde{\pi}_{\tilde{I}}\colon
\Lambda[[\Gamma]] \longrightarrow \Lambda[[\Gamma]]/\tilde{I}
\)
be the natural projection.
We write
$\widetilde{\mca{I}}_n:=\mca{I}_{\bb{T}^{\mathrm{cyc}},n}$.
If the ideal $\tilde{I}$ contains $\widetilde{\mca{I}}_n$,
then as in Definition \ref{dfnKD}, we can define
the Kolyvagin derivative 
\[
\kappa(\widetilde{\boldsymbol{c}};n)_{\tilde{I}}
\in H^1(\tilde{\pi}_{\tilde{I}}^*\bb{T}^{\mathrm{cyc}}) 
\otimes_{\bb{Z}} H^\otimes_{n}
\]
arising from the Euler system $\widetilde{\boldsymbol{c}}$
on $\widetilde{\bb{T}}$.

Let $n \in \mca{N}_\Sigma$ be any element.
We denote by $\mf{S}_{\mathrm{Prime}(n)}$ 
the set of Permutations on $\mathrm{Prime}(n)$.
Let $\alpha \in \mf{S}_{\mathrm{Prime}(n)}$ be any element.
the sign of $\alpha$ is denoted by $\mathrm{sign}(\alpha)$, 
and the set of elements fixed by $\alpha$ is denoted by 
$\mathrm{Prime}(n)^\alpha$.
We put 
\[
d_\alpha := \prod_{\ell \in \mathrm{Prime}(n)^\alpha} \ell
\in \mca{N}_{\Sigma}.
\]

Let $\ell$ be a prime number not contained in $\Sigma$,
and $\tilde{I}$ an ideal of $\Lambda[[\Gamma]]$. 
We denote by $\widetilde{\mf{a}}_{I,H_\ell}$ 
the augmentation ideal of the group ring
$\Lambda[[\Gamma]]_{/\tilde{I},[\ell]}
=(\Lambda[[\Gamma]]/\tilde{I})[H_\ell]$ 
on the group $H_\ell$. Namely, we put
\[
\widetilde{\mf{a}}_{\tilde{I},H_\ell}:= \Ker \left( 
\mathrm{aug}\colon  \Lambda[[\Gamma]]_{/\tilde{I},[\ell]}
 \longrightarrow \Lambda[[\Gamma]]/\tilde{I}
\right).
\]
We have an isomorphism
\[
e_{\tilde{I},H_\ell}\colon 
\widetilde{\mf{a}}_{\tilde{I},H_\ell}/
\widetilde{\mf{a}}_{\tilde{I},H_\ell }^2
\longrightarrow (\Lambda[[\Gamma]]/\tilde{I})\otimes_{\bb{Z}} H_\ell
\]
of $\Lambda[[\Gamma]]/\tilde{I}$-modules defined by
$e_{\tilde{I},H_\ell}(\sigma-1)= 1 \otimes \sigma$
for any $\sigma \in H_\ell$.
Similarly, for any ideal $I$ of $\Lambda$ of finite index,
we denote the augmentation ideal of 
$\Lambda_{/I,[\ell]}=(\Lambda/I)[H_\ell]$ 
by $\mf{a}_{I,\ell}$, and define the isomorphism
\[
e_{I,H_\ell}\colon 
\mf{a}_{I,H_\ell}/\mf{a}_{I,H_\ell }^2
\longrightarrow (\Lambda/I)\otimes_{\bb{Z}} H_\ell
\]
of $\Lambda/I$-modules.

\begin{dfn}
Let $n  \in \mca{N}_{\Sigma}$ be any element.
\begin{enumerate}[(i)]
\item For any ideal $\tilde{I}$ of $\Lambda[[\Gamma]]$ of finite index 
containing $\widetilde{\mca{I}}_n$,
we define an element $\kappa^{\mathrm{univ}}_n
(\widetilde{\boldsymbol{c}})_{\tilde{I}}$ of 
\[
H^1(\tilde{\pi}_{\tilde{I}}^*\bb{T}^{\mathrm{cyc}}) 
\otimes_{\bb{Z}} H^\otimes_{n}
=H^1(\tilde{\pi}_{\tilde{I}}^*\bb{T}^{\mathrm{cyc}}) 
\otimes_{\Lambda[[\Gamma]]/\tilde{I}}
\bigotimes_{\ell \mid n}
\left((\Lambda[[\Gamma]]/\tilde{I})
\otimes_{\bb{Z}} H_{\ell}\right)
\] 
by
\[
\kappa^{\mathrm{univ}}_n(\widetilde{\boldsymbol{c}})_{\tilde{I}} 
:= \sum_{\alpha \in \mf{S}_{\mathrm{Prime}(n)}} 
\mathrm{sign}(\alpha) 
\kappa(\widetilde{\boldsymbol{c}};d_\alpha)_{\tilde{I}} \otimes
\bigotimes_{\ell \mid (n/d_\alpha)} 
e_{\tilde{I},H_\ell}(P_\ell(\Frob_{\alpha(\ell)};\bb{T}^{\mathrm{cyc}})).
\]
\item Similarly to (i), for any ideal $I$ of $\Lambda$ of finite index 
containing $\mca{I}_n$,
we define an element $\kappa^{\mathrm{univ}}_n
(\boldsymbol{c})_{I} 
\in H^1(\pi_{I}^*\bb{T}) 
\otimes_{\bb{Z}} H^\otimes_{n}$ by
\[
\kappa^{\mathrm{univ}}_n(\widetilde{\boldsymbol{c}})_{I} 
:= \sum_{\alpha \in \mf{S}_{\mathrm{Prime}(n)}} 
\mathrm{sign}(\alpha) 
\kappa(\boldsymbol{c};d_\alpha)_{I} \otimes
\bigotimes_{\ell \mid (n/d_\alpha)} 
e_{I,H_\ell}(P_\ell(\Frob_{\alpha(\ell)};\bb{T})).
\]
\item For any $\alpha \in  \mf{S}_{\mathrm{Prime}(n)}$ and 
any prime divisor $\ell$ of $n/d_\alpha$, 
we fix an element $A(n;\alpha,\ell) \in \Lambda$ satisfying
\[
\left(
A(\alpha,\ell) \ 
\mathrm{mod}\ \mca{I}_n
\right)
\otimes \sigma_\ell = 
e_{\mca{I}_n,H_\ell}(P_\ell(\Frob_{\alpha(\ell)};\bb{T})) \in 
(\Lambda_{/\mca{I}_n})\otimes_{\bb{Z}} H_\ell.
\]
Then, we define an element $d^{\mathrm{univ}}_n(\boldsymbol{c}) 
\in H^1(\pi_{[n]}^*\bb{T})$ by
\[
d^{\mathrm{univ}}_n(\boldsymbol{c})_I 
:= \sum_{\alpha \in \mf{S}_{\mathrm{Prime}(n)}} 
\mathrm{sign}(\alpha) \left( 
\prod_{\ell \mid (n/d_\alpha)}  A(\alpha,\ell) \right) 
\cdot D_{d_{\alpha}} c(d_{\alpha}).
\]
\end{enumerate}
\end{dfn}

\begin{rem}
Let $n  \in \mca{N}_{\Sigma}$ be any element, and
$I$ any ideal of $\Lambda$ of finite index 
containing $\mca{I}_n$.
Then, by definition, the following hold.
\begin{enumerate}[(i)]
\item When $n=1$, we have
\[
d^{\mathrm{univ}}_1(\boldsymbol{c})_I=
\kappa^{\mathrm{univ}}_1
(\boldsymbol{c})_{I}=c(1)_I
\in H^1(\pi_I^*\bb{T}).
\]
\item The image of 
$d^{\mathrm{univ}}_n(\boldsymbol{c})_I \otimes 
\bigotimes_{\ell \in \mathrm{Prime}(n)} \sigma_\ell$ by
the inverse map of the restriction map 
\[
H^1(\pi_I^*\bb{T})\otimes_{\bb{Z}} 
H^\otimes_n \longrightarrow 
H^1(\pi_{I,[n]}^*\bb{T})^{H_n}
\otimes_{\bb{Z}} H^\otimes_n
\]
coincides with $\kappa^{\mathrm{univ}}_n
(\boldsymbol{c})_{I}$.
\item We identify the topological $(\Lambda/I)[G_{\bb{Q},\Sigma}]$-module
$\tilde{\pi}_{\mf{a}_{\Gamma}+I\Lambda[[\Gamma]]}^*
\bb{T}^{\mathrm{cyc}}$
with $\pi_I^*\bb{T}$ via the natural isomorphism.
Then, we have
\[
\kappa^{\mathrm{univ}}_n
(\widetilde{\boldsymbol{c}})_{\mf{a}_{\Gamma}+I\Lambda[[\Gamma]]} 
=\kappa^{\mathrm{univ}}_n
(\boldsymbol{c})_{I}
\in H^1(\pi_I^*\bb{T})\otimes_{\bb{Z}} H^\otimes_n.
\]
\end{enumerate}
\end{rem}

\begin{rem}
An axiomatic framework of
Kolyvagin systems for Galois deformations 
are studied in \cite{Bu}.
Note that our notion of ``universal Kolyvagin system"
is slightly different from that in \cite{Bu}. 
On the one hand, in \cite{Bu}, 
a ``universal Kolyvagin system in \cite{Bu} is 
a system of Galois cohomology classes
satisfying certaion axioms. 
We have not proved that our ``universal Kolyvagin systems"
satisfies the axioms required in \cite{Bu}. 
(In our article, we can omit to check it since 
our main results follow from the theory of Kolyvagin systems over DVRs.)
\end{rem}

Recall that we have fixed an Euler system $\widetilde{\boldsymbol{c}}
=\{ \tilde{c}(n) \}_{n \in \mca{N}_{\Sigma}}$
on the cyclotomic deformation $\bb{T}^{\mathrm{cyc}}$.
Let $\chi \in \Hom_{\mathrm{cont}}(\Gamma,\bb{Z}_p^\times) 
\simeq \bb{Z}_p$
be any element.
The character $\chi$ induces a continuous 
$\Lambda$-algebra homomorphism
$e_\chi \colon \Lambda[[\Gamma]] \longrightarrow \Lambda$
given by $e_\chi(g)=\chi(g)$ for each $g \in \Gamma$.
We put $\bb{T}\otimes \chi:=
e_{\chi}\bb{T}^{\mathrm{cyc}}$, and consider the  Euler system 
$c\otimes \chi:=e_\chi^*\widetilde{\boldsymbol{c}}$ on $\bb{T}\otimes \chi$. 
Roughly speaking, the following proposition
implies that a specialization of the universal Kolyvagin systems 
becomes Kolyvagin systems.

\begin{prop}\label{propspecialunivKolyv}
Let $\mca{O}'$ be the ring of integers
of a finite extension field $F'$ of $F$, 
and 
$f \colon \Lambda \longrightarrow \mca{O}'$
any continuous ring homomorphism. 
Then, for any open subgroup $U$ of 
$\Hom_{\mathrm{cont}}(\Gamma,\bb{Z}_p^\times)$,
there exists a character 
$\chi \in U$
such that the system
\[
f^*\kappa^{\mathrm{univ}}(\boldsymbol{c} \otimes \chi)
= \left\{ 
f^*\kappa^{\mathrm{univ}}_n(\boldsymbol{c} \otimes \chi)
\in H^1\left(
f^*_{\mca{I}_{\bb{T} \otimes \chi ,n}}(\bb{T} \otimes \chi)
\right) \otimes_{\bb{Z}} H^\otimes_{n}
\right\}_{n \in \mca{N}(\bb{T},\mf{m}_\Lambda)}
\]
forms a Kolyvagin system for the Selmer triple
$(f^*(\bb{T}\otimes \chi),\mca{F}_{\mathrm{can}},
\mca{P}(\bb{T},\mf{m}_\Lambda))$ over $\mca{O}'$,
where 
$f^*\kappa^{\mathrm{univ}}_n(\boldsymbol{c} \otimes \chi)$
is the image of 
$\kappa^{\mathrm{univ}}_n(\widetilde{\boldsymbol{c}})_{
\mca{I}_{\bb{T} \otimes \chi ,n}+\Ker (f \otimes e_\chi)}$
by the  map
\[
H^1\left(
\tilde{\pi}^*_{\widetilde{\mca{I}}_n+ \Ker e_\chi}(\bb{T}^{\mathrm{cyc}})
\right) \otimes_{\bb{Z}} H^\otimes_{n}
\longrightarrow 
H^1\left(
f^*_{\mca{I}_{\bb{T} \otimes \chi ,n}}(\bb{T} \otimes \chi)
\right) \otimes_{\bb{Z}} H^\otimes_{n}
\]
induced by the continuous ring homomorphism 
$f \otimes e_{\chi} \colon \Lambda[[\Gamma]] \longrightarrow \mca{O}'$,
and 
\[
f_{\mca{I}_{\bb{T} \otimes \chi ,n}} \colon \Lambda \longrightarrow 
\mca{O}'/f(\mca{I}_n) \mca{O}'
\]
denotes the continuous ring homomorphism induced by $f$.
\end{prop}

\begin{proof}
For each $\ell \in \mca{P}(\bb{T},\mf{m}_\Lambda)$
and $N \in \bb{Z}_{>0}$, 
we denote by 
$\Xi_{\ell,N}$ the subset of
$\Hom_{\mathrm{cont}}(\Gamma,\bb{Z}_p^\times)$
which consists of characters 
$\psi \in \Hom_{\mathrm{cont}}(\Gamma,\bb{Z}_p^\times)$
such that the operator
$f^*(\rho \otimes \psi)(\Frob_\ell)^N-\mathrm{id}$ does not 
acts on $f^* (\bb{T} \otimes \chi)$ invectively.
Note that the set $\Xi_{\ell,N}$ is finite.
We put
\[
\Xi:= \bigcup_{\ell \in \mca{P}(\bb{T},\mf{m}_\Lambda)}
\bigcup_{N>0} \Xi_{\ell,N}.
\]
Since $\Xi$ is a countable set, 
the set $U \setminus \Xi$ is not empty.

Take any character $ \chi \in U \setminus \Xi$.
Then, the pair $(f^*(\bb{T}\otimes \chi), 
\mca{P}(\bb{T},\mf{m}_\Lambda))$
satisfies the assumptions in \cite[Theorem 3.2.4 ]{MR}.
So by the construction of 
the Kolyvagin system $\boldsymbol{\kappa}$
in Theorem 3.2.4, which is explained in \cite[Appendix A]{MR},
the system 
$f^*\kappa^{\mathrm{univ}}(\boldsymbol{c} \otimes \chi)$
becomes a Kolyvagin system for the Selmer triple
$(f^*(\bb{T}\otimes \chi),\mca{F}_{\mathrm{can}},
\mca{P}(\bb{T},\mf{m}_\Lambda))$.
\end{proof}

Let $\mca{O}'$ be the ring of integers
of a finite extension field $F'$ of $F$, 
and 
$f \colon \Lambda \longrightarrow \mca{O}'$
any continuous ring homomorphism. 
For each $n \in \mca{N}_\Sigma$, we define
\[
\partial(\boldsymbol{c},f;n)
:=
\max\{
j \in \bb{Z}_{\ge 0} \mathrel{\vert}
f^*\kappa^{\mathrm{univ}}(\boldsymbol{c}) 
\in \mf{m}_R^jH^1_{\mca{F}_{\mathrm{can}}(n)}(\bb{Q},
f^*\bb{T}/\mca{I}_{f^*\bb{T},n} f^* \bb{T})
\}.
\]
For each $t \in \bb{Z}_{>0}$ and $i \in \bb{Z}_{\ge 0}$, 
we put
\[
\partial_i(\boldsymbol{c},f)_t:=
\min\{\partial(\boldsymbol{c},f;n)
\mathrel{\vert} \# \mathrm{Prime}(n)=i,\ 
\mathrm{Prime}(n) \subseteq \mca{P}(t) \}.
\]

\begin{cor}\label{corspecialunivKolyv}
Let 
$f \colon \Lambda \longrightarrow \mca{O}'$
be as in Proposition 
\ref{propspecialunivKolyv},
and $\mf{m}_{\mca{O}'}$ the maximal ideal of $\mca{O}'$.
Assume that 
the order of $X(f^*\bb{T})$ is finite, and we have
\[
\Fitt_{\mca{O}',0}(X(f^*\bb{T}))
=\mf{m}_{\mca{O}'}^{N+\partial(\boldsymbol{c},f;1)} 
\]
for some $N \in \bb{Z}_{\ge 0}$.
Let $t \in \bb{Z}_{>0}$ be a positive integer
satisfying $t > \mathrm{length}_{\mca{O}'}(X(f^*\bb{T}))$.
Then, for any $i \in \bb{Z}_{\ge 0}$, we have 
\[
\mf{m}_{\mca{O}'}^N \Fitt_{\mca{O}',i}(X(f^*\bb{T}))
=\mf{m}_{\mca{O}'}^{\partial_i(\boldsymbol{c},f)_t}.
\]
\end{cor}

\begin{proof}
Recall we have fixed a topological generator 
$\gamma \in \Gamma$.
By proposition 
\ref{propspecialunivKolyv},
there exists a continuous homomorphism
$\chi \colon \Gamma \longrightarrow \bb{Z}_p^\times$
satisfying $\chi(\gamma)-1 \in \mf{m}^t_{\mca{O}'}$
such that the system
\[
f^*\kappa^{\mathrm{univ}}(\boldsymbol{c} \otimes \chi)
= \left\{ 
f^*\kappa^{\mathrm{univ}}_n(\boldsymbol{c} \otimes \chi)
\in H^1\left(
f^*_{\mca{I}_{\bb{T} \otimes \chi ,n}}(\bb{T} \otimes \chi)
\right) \otimes_{\bb{Z}} H^\otimes_{n}
\right\}_{n \in \mca{N}(\bb{T},\mf{m}_\Lambda)}
\]
becomes a Kolyvagin system for the Selmer triple
$(f^*(\bb{T}\otimes \chi),\mca{F}_{\mathrm{can}},
\mca{P}(\bb{T},\mf{m}_\Lambda))$.
By Theorem \ref{thmMRuniqueness}
and Theorem \ref{thmMRstrthm},
there exists $N \in \bb{Z}_{>0}$ such that
\begin{equation}\label{eqMRunique}
\mca{O}'\cdot
f^*\kappa^{\mathrm{univ}}(\boldsymbol{c} \otimes \chi)
=\mf{m}_{\mca{O}'}^N \cdot
\mathbf{KS}(f^*(\bb{T}\otimes \chi),\mca{F}_{\mathrm{can}},
\mca{P}(\bb{T},\mf{m}_\Lambda)).
\end{equation}

Fix a uniformizer $\varpi' \in \mca{O}'$.
Since $X(f^*\bb{T})$ is annihilated 
by $\mf{m}_{\mca{O}'}^t$,
the cohomological exact sequence arising from
the short exact sequence
\[
0 \longrightarrow 
f^*\bb{T} \xrightarrow{\ \times \varpi'^t \ }
f^*\bb{T} \longrightarrow 
f^*\bb{T}/\mf{m}_{\mca{O}'}^t f^*\bb{T} \longrightarrow 0
\]
implies that we have a natural isomorphism
\[
X(f^*\bb{T})  \simeq H^2_{\Sigma}(f^*\bb{T})
\simeq H^2_{\Sigma}(f^*\bb{T} /\mf{m}_{\mca{O}'}^t f^*\bb{T} ). 
\]
We also have a natural isomorphism
\[
X(f^* (\bb{T} \otimes \chi))  
/\mf{m}_{\mca{O}'}^t X(f^* (\bb{T} \otimes \chi))  
\simeq H^2_{\Sigma}(f^* (\bb{T} \otimes \chi)
/\mf{m}_{\mca{O}'}^t  f^* (\bb{T} \otimes \chi)).
\]
Since we have a natural $\mca{O}'[G_{\bb{Q},\Sigma}]$-equivariant isomorphism
\[
f^* (\bb{T} \otimes \chi)/\mf{m}_{\mca{O}'}^t  f^* (\bb{T} \otimes \chi)
\simeq f^*\bb{T} /\mf{m}_{\mca{O}'}^t f^*\bb{T}, 
\]
we obtain the isomorphism
\[
X(f^* (\bb{T} \otimes \chi))  
/\mf{m}_{\mca{O}'}^t X(f^* (\bb{T} \otimes \chi))  
\simeq X(f^* \bb{T}).
\]
Hence by (\ref{eqMRunique}) and 
Theorem \ref{thmMRstrthm},
for any $i \in \bb{Z}_{\ge 0}$, we have
\begin{align*}
\mf{m}_{\mca{O}'}^N\Fitt_{\mca{O}',i}(X(f^*\bb{T}))
&=\mf{m}_{\mca{O}'}^N\Fitt_{\mca{O}',i}
\left(
X(f^* (\bb{T} \otimes \chi))  
/\mf{m}_{\mca{O}'}^t X(f^* (\bb{T} \otimes \chi))  
\right) \\
& =
\mf{m}_{\mca{O}'}^N \cdot 
\left( \Fitt_{\mca{O}',i}(X(f^* (\bb{T} \otimes \chi)))
+ \mf{m}_{\mca{O}'}^{t} \right) \\
& = \mf{m}_{\mca{O}'}^{N+ \min \{ \partial_i(f^*\kappa^{\mathrm{univ}}
(\boldsymbol{c} \otimes \chi))_t, t \}} \\
&= \mf{m}_{\mca{O}'}^{N+ \partial_i(f^*\kappa^{\mathrm{univ}}
(\boldsymbol{c} \otimes \chi))_t} \\
&=  \mf{m}_{\mca{O}'}^{N+ \partial_i(\boldsymbol{c},f)_t}.
\end{align*}
This completes the proof.
\end{proof}

\section{The ideal $\mf{C}_i(\boldsymbol{c})$}\label{secCi}

Let $(\bb{T}, \boldsymbol{c})$ be as in Theorem \ref{thmuncondresults}.
In this section, we shall construct the ideals 
$\mf{C}_i(\boldsymbol{c})$, 
and prove their basic properties.
In \S \ref{ssconstC_i}, we fix a monic parameter system 
$\boldsymbol{h}$, and construct 
the ideals 
$\mf{C}_i(\boldsymbol{c};\boldsymbol{h})$. 
(Note that we shall define 
$\mf{C}_i(\boldsymbol{c}):=\mf{C}_i(\boldsymbol{c};\boldsymbol{x})$, 
where $\boldsymbol{x}$ denotes the standard monic parameter system.) 
In \S \ref{ssvarmps},
we vary the monic parameter system 
$\mf{C}_i(\boldsymbol{c};\boldsymbol{h})$, 
and prove the independence of the ideal
$\mf{C}_i(\boldsymbol{c};\boldsymbol{h})$
of the choice of the monic parameter system $\boldsymbol{h}$.
In \S \ref{ssscalarext}, 
we prove a basic property of the ideals 
$\mf{C}_i(\boldsymbol{c};\boldsymbol{h})$,
namely the stability under 
scalar extensions (Proposition \ref{corOO'}).
This property plays an important role in the reduction arguments 
in  \S \ref{secpfmr} based on Ochiai's work \cite{Oc2}.
In \S \ref{ssvarmps}, we show 
another basic property of 
of the ideals 
$\mf{C}_i(\boldsymbol{c};\boldsymbol{h})$,
that is, the stability under 
affine transformations (Proposition \ref{lemafftrans}).
This stability will not be used in the proof of our main results, 
but it seems to be important ingredient to deal 
with concrete problems since this stability implies that 
in some sense,  
the definition of $\mf{C}_i(\boldsymbol{c};\boldsymbol{h})$
does not depend on parameters $x_1 , \dots , x_r$ of $\Lambda$.

\subsection{The construction}\label{ssconstC_i}

In this subsection, we fix 
a monic parameter system $\boldsymbol{h}$ of $\Lambda$.
We shall construct an ideal 
$\mf{C}_i(\boldsymbol{c};\boldsymbol{h})$ 
of $\Lambda$ for any $i \in \bb{Z}_{\ge 0}$.

\begin{dfn}\label{dfnidealC(n)}
Let $I$ and $I'$ be ideals of $\Lambda$ satisfying 
$I \subseteq I'$, and $n \in \mca{N}_{\Sigma}$ any element.
\begin{enumerate}[(i)]
\item We define the ideal 
$\mca{KI}(\boldsymbol{c};I;n)$ of $\Lambda_{/I',[n]}$ by
\[
\mca{KI}(\boldsymbol{c};I';n):=
\left\{
f(d^{\mathrm{univ}}_n(\boldsymbol{c})_{I'}) 
\mathrel{\bigg\vert} 
f \in \Hom_{\Lambda_{/I',[n]}}
(H^1(\pi_{I',[n]}^* \bb{T}),
\Lambda_{/I',[n]})
\right\}.
\]
We call $\mca{KI}(\boldsymbol{c};I;n)$ 
\textit{the ideal of Kolyvagin images}. 
\item We denote by $\mca{KI}(\boldsymbol{c};I';n)_I$ the image of 
$\mca{KI}(\boldsymbol{c};I';n)$ in $\Lambda_{/I,[n]}$.
\item Assume that $n \in \mca{N}(\bb{T},I)$.
Then, by (\ref{Dnfixed}) in \S \ref{ssKD}, 
the set $\mca{KI}(\boldsymbol{c};I';n)_I$
is fixed by the action of $H_n$.
We define the ideal $\mf{C}(\boldsymbol{c};I';n)_I$ of $\Lambda/I$ 
by the inverse image of 
$\mca{KI}(\boldsymbol{c};I';n)_I$
by the isomorphism
\[
\Lambda/I \xrightarrow{\ \simeq \ } 
(\Lambda_{/I,[n]})^{H_n}=N_{H_n}\Lambda_{/I,[n]};\ 
x \longmapsto N_{H_n} x,
\]
where we put 
\[
N_{H_n}:= \sum_{\sigma \in H_n} \sigma.
\]
\item If $I=I'$, then 
we put $\mf{C}(\boldsymbol{c};I;n):=\mf{C}(\boldsymbol{c};I;n)_I$.
\end{enumerate}
\end{dfn}

\begin{rem}
Let $\boldsymbol{h}$ be a monic parameter system of $\Lambda$, 
and $\boldsymbol{m} \in (\bb{Z}_{>0})^{r+1}$ any element.
We put $I:=I(\boldsymbol{h}^{\boldsymbol{m}})\Lambda$. 
Let $n \in \mca{N}(\bb{T},I)$ be any element.
Then, by the isomorphism (\ref{resisom}) 
in \S \ref{dfnKD} and the injectivity of 
the $\Lambda_{I,[n]}$-module $\Lambda_{I,[n]}$
(see Lemma \ref{leminjmod}), 
we have
\[
\mf{C}(\boldsymbol{c};I;n)=
\left\{
f(\kappa^{\mathrm{univ}}_n(\boldsymbol{c})_{I}) 
\mathrel{\bigg\vert} 
f \in \Hom_{\Lambda/I}
(H^1(\pi_{I}^* \bb{T}) \otimes H_n^{\otimes},
\Lambda/I)
\right\}.
\]
Note that the ideal $\mf{C}(\boldsymbol{c};I;n)$ is independent of 
the choice of generators $\sigma_\ell$ and 
elements $A(n; \alpha, \ell) \in \Lambda$.
\end{rem}

In order to define the ideals 
$\mf{C}_i(\boldsymbol{c};\boldsymbol{h})$ 
the following lemma becomes a key.

\begin{lem}\label{lemascdesc}
Let $\boldsymbol{h}$ be a monic parameter system of $\Lambda$, 
and $\boldsymbol{m}, \boldsymbol{m}' \in (\bb{Z}_{>0})^{r+1}$ 
elements satisfying $\boldsymbol{m}' \ge \boldsymbol{m}$.
We put $I:=I(\boldsymbol{h}^{\boldsymbol{m}})$ and 
$I':=I(\boldsymbol{h}^{\boldsymbol{m}'})$. 
Let $n \in \mca{N}_{\Sigma}$ be any element.
\begin{enumerate}[{\rm (i)}]
\item For any $f' \in \Hom_{\Lambda_{/I',[n]}}
(H^1(\pi_{I',[n]}^* \bb{T}),
\Lambda_{/I',[n]})$, 
there exists a homomorphism
$f \in \Hom_{\Lambda_{/I,[n]}}
(H^1(\pi_{I,[n]}^* \bb{T}),
\Lambda_{/I,[n]})$
which makes the diagram
\[
\xymatrix{
H^1(\pi_{I',[n]}^* \bb{T}) 
\ar[r]^(0.58){f'} \ar[d]_{\pi_1} & \Lambda_{/I',[n]} \ar[d]^{\pi_2} \\
H^1(\pi_{I,[n]}^* \bb{T}) 
\ar@{-->}[r]^(0.58){f} & \Lambda_{/I,[n]}
}
\]
commute, where $\pi_1$ and $\pi_2$ are natural maps. 
\item For any $f \in \Hom_{\Lambda_{/I,[n]}}
(H^1(\pi_{I,[n]}^* \bb{T}) ,
\Lambda_{/I,[n]})$, 
there exists a homomorphism
$f' \in \Hom_{\Lambda_{/I',[n]}}
(H^1(\pi_{I',[n]}^* \bb{T}) ,
\Lambda_{/I',[n]})$
which makes the diagram
\[
\xymatrix{
H^1(\pi_{I',[n]}^* \bb{T}) 
\ar@{-->}[r]^(0.58){f'} \ar[d]_{\pi_1} & \Lambda_{/I',[n]} \ar[d]^{\pi_2} \\
H^1(\pi_{I,[n]}^* \bb{T}) 
\ar[r]^(0.58){f} & \Lambda_{/I,[n]}
}
\]
commute.
\end{enumerate}
\end{lem}

\begin{proof}
We put $\boldsymbol{h}=(h_0, \dots , h_{r})$ and 
$\boldsymbol{h}=(m_0, \dots , m_{r})$. 
In order to show our lemma, 
we may and do assume that $\boldsymbol{m}'$ is written in the form
$\boldsymbol{m}'=(\boldsymbol{m}_{\le i-1},m_i +1, \boldsymbol{m}_{\ge i+1})$
for some integer $i$ with $0 \le i \le r$.
Then, we have an isomorphism
\[
\nu_2 \colon \Lambda_{/I,[n]} \xrightarrow{\ \simeq \ } 
\Lambda_{/I',[n]}[I]=h_i \Lambda_{/I',[n]};\ 
x \longmapsto h_i x.
\]
We denote by
\[
\nu_1 \colon H^1(\pi_{I,[n]}^* \bb{T}) 
\longrightarrow 
H^1(\pi_{I',[n]}^* \bb{T}) 
\]
the map induced by 
$\nu_2\otimes \mathrm{id}_{\bb{T}}
\colon  \pi_{I,[n]}^* \bb{T}
\longrightarrow 
\pi_{I',[n]}^* \bb{T}$.
By definition, the maps $\nu_1 \circ \pi_1$ 
and $\nu_2 \circ \pi_2$ are endomorphisms 
defined by the scalar multiplication 
by $h_i$ respectively.
Since the Galois representation $\bb{T}$ 
is unramified at each prime divisors of $n$,
the assumption (A3) and the short exact sequence
\[
0 \longrightarrow \pi_{I,[n]}^* \bb{T} 
\xrightarrow{\ \times h_i \ } 
\pi_{I',[n]}^* \bb{T} \longrightarrow 
\pi_{I(\boldsymbol{m}_{\le i-1},1, \mf{m}_{\ge i+1}),[n]}^* \bb{T}
\longrightarrow 0
\]
imply that the map $\nu_1$ is injective.

Let us prove the first assertion.
Note that the image of $f' \circ \nu_1$
is annihilated by $I$.
We define a homomorphism
\[
f \colon 
H^1(\pi_{I,[n]}^* \bb{T})
\longrightarrow 
\Lambda_{/I,[n]} 
\]
of $\Lambda_{/I,[n]}$-modules 
by $f:= \nu_2^{-1} \circ f' \circ \nu_1$.
Then, we have 
\[
f \circ \pi_1 = \nu_2^{-1} \circ f' \circ \nu_1 \circ \pi_1
=\nu_2^{-1} \circ h_i f' = \nu_2^{-1} \circ \nu_2 \circ \pi_2 \circ f'
=\pi_2 \circ f',
\]
so the map $f$ satisfies the required properties.

Next, let us show the second assertion.
Since the map $\nu_1$ is an injection,
we can define the $\Lambda_{/I',[n]}$-linear map
\[
f'_0 \colon 
\Im \nu_1
\longrightarrow 
\Lambda_{/I',[n]} 
\]
by $f'_0:= \nu_2 \circ f \circ \nu_1^{-1}$.
Recall that by lemma \ref{leminjmod}, 
the $\Lambda_{/I',[n]}$-module $\Lambda_{/I',[n]}$
is injective.
So we have a homomorphism
$f' \colon 
H^1(\pi_{I,[n]}^* \bb{T}) 
\longrightarrow 
\Lambda_{/I,[n]}$ whose restriction to $\Im \nu_1$
coincides with $f'_0$.
By definitions of $f'_0$ and $f'$, 
it holds that
\[
\nu_2 \circ \pi_2 \circ f' = h_i f' 
=f' \circ \nu_1 \circ \pi_1 =  
\nu_2 \circ f \circ \pi_1.
\]
Since $\nu_2$ is an injection, we obtain 
$\pi_2 \circ f' = f \circ \pi_1$.
This completes the proof.
\end{proof}

By Lemma \ref{lemascdesc}, we obtain the following corollary, 
which is plays an important role in \S \ref{secred}.  

\begin{cor}\label{corascdesc}
Let $\boldsymbol{h}$ be a monic parameter system of $\Lambda$, 
and $\boldsymbol{m}, \boldsymbol{m}' \in (\bb{Z}_{>0})^{r+1}$ 
elements satisfying $\boldsymbol{m}' \ge \boldsymbol{m}$.
We put $I:=I(\boldsymbol{h}^{\boldsymbol{m}})$ and 
$I':=I(\boldsymbol{h}^{\boldsymbol{m}'})$. 
Let $n \in \mca{N}_{\Sigma}$ be any element.
Then, we have
\[
\mca{KI}(\boldsymbol{c};I';n)_I=\mca{KI}(\boldsymbol{c};I;n)
\]
In particular, if $n \in \mca{N}(\bb{T},I)$, 
then we have
\[
\mf{C}(\boldsymbol{c};I';n)_I=
\mf{C}(\boldsymbol{c};I;n).
\]
\end{cor}

Let $\boldsymbol{h}$ be a monic parameter system of $\Lambda$,
and fix a collection of non-empty subsets
\[
\mf{N}:=
\{
\mca{N}_{\boldsymbol{m}}
\}_{\boldsymbol{m} \in (\bb{Z}_{>0})^{r}}
\]
satisfying 
$\mca{N}_{\boldsymbol{m}'} \subseteq 
\mca{N}_{\boldsymbol{m}}
\subseteq \mca{N}(\bb{T},I(\boldsymbol{h}^{\boldsymbol{m}}))$
for any $\boldsymbol{m}, \boldsymbol{m}' \in (\bb{Z}_{>0})^{r+1}$ 
with $\boldsymbol{m}' \ge \boldsymbol{m}$.
We put 
\[
\mf{N}(\bb{T};\boldsymbol{h}):=
\{
\mca{N}(\bb{T},I(\boldsymbol{h}^{\boldsymbol{m}}))
\}_{\boldsymbol{m} \in (\bb{Z}_{>0})^{r}}
\]
Let $\boldsymbol{m} \in (\bb{Z}_{>0})^{r+1}$  and
$i \in \bb{Z}_{\ge 0}$ be arbitrary elements.
We define
\[
\mca{N}_{\boldsymbol{m}}(i):=
\left\{
n \in \mca{N}_{\boldsymbol{m}} \mathrel{\vert}
\# \mathrm{Prime}(n) \le i
\right\}.
\]
Then, we denote by $\mf{C}_i(\boldsymbol{c};
I(\boldsymbol{h}^{\boldsymbol{m}});\mf{N})$ 
the ideal of $\Lambda_{/I(\boldsymbol{h}^{\boldsymbol{m}})}$ 
generated by 
\[
\bigcup_{n \in \mca{N}_{\boldsymbol{m}}(i)}
\mf{C}(\boldsymbol{c};I(\boldsymbol{h}^{\boldsymbol{m}});n).
\]

Let $\boldsymbol{m}, \boldsymbol{m}' \in (\bb{Z}_{>0})^{r+1}$ 
be elements satisfying $\boldsymbol{m}' \ge \boldsymbol{m}$.
Put $I:=I(\boldsymbol{h}^{\boldsymbol{m}})$, and 
$I':=I(\boldsymbol{h}^{\boldsymbol{m}'})$. 
The image of $\mf{C}_i(\boldsymbol{c};I';\mf{N})$ 
in $\Lambda/I$ is denoted by $\mf{C}_i(\boldsymbol{c};I';\mf{N})_I$.
We write 
$\mf{C}_i(\boldsymbol{c};I'):=\mf{C}_i(\boldsymbol{c};I';
\mf{N}(\bb{T};\boldsymbol{h}))$ and 
$\mf{C}_i(\boldsymbol{c};I')_I:=
\mf{C}_i(\boldsymbol{c};I';\mf{N}(\bb{T};\boldsymbol{h}))_I$.
By Corollary \ref{corascdesc}, we have
\begin{eqnarray}\label{compC_iI'C_iI}
\mf{C}_i(\boldsymbol{c};I';\mf{N})_I \subseteq \mf{C}_i
(\boldsymbol{c};I;\mf{N}).
\end{eqnarray}
Hence we obtain a projective system
$\{ \mf{C}_i(\boldsymbol{c};I(\boldsymbol{h}^{\boldsymbol{m}});\mf{N}) 
\}_{\boldsymbol{m} \in (\bb{Z}_{>0})^{r+1}}$.

\begin{dfn}\label{dfnC_i}
Let $\boldsymbol{h}$ be a monic parameter system of $\Lambda$.
Then, we define an ideal $\mf{C}_i(\boldsymbol{c};\boldsymbol{h})$ 
of the ring $\Lambda$ by
\[
\mf{C}_i(\boldsymbol{c};\boldsymbol{h};\mf{N}):=
\varprojlim_{\boldsymbol{m}}\mf{C}_i(\boldsymbol{c};I(\boldsymbol{h}^{\boldsymbol{m}});\mf{N}) 
\subseteq 
\varprojlim_{\boldsymbol{m}} \Lambda_{/I(\boldsymbol{h}^{\boldsymbol{m}})}
\simeq \Lambda. 
\]
Especially, we put 
$\mf{C}_i(\boldsymbol{c};\boldsymbol{h}):=
\mf{C}_i(\boldsymbol{c};\boldsymbol{h};\mf{N}(\bb{T};\boldsymbol{h}))$
and $\mf{C}_i(\boldsymbol{c}):=
\mf{C}_i(\boldsymbol{c};\boldsymbol{x};\mf{N}(\bb{T};\boldsymbol{x}))$, 
where $\boldsymbol{x}$ denotes the standard monic parameter system.
\end{dfn}

\subsection{Varying monic parameter systems}\label{ssvarmps}

In the previous subsection, we have constructed 
the collection $\{ \mf{C}_i(\boldsymbol{c};\boldsymbol{h}) \}_{i \ge 0}$
of ideals of $\Lambda$ by using the ``universal Kolyvagin system" 
arising from the Euler system $\boldsymbol{c}$.
Note that a priori, the definition of 
the ideals $\mf{C}_i(\boldsymbol{c};\boldsymbol{h})$
depends on the choice of 
the monic parameter system $\boldsymbol{h}$.
Here, we shall prove Proposition \ref{propindepofparam}
which assert that the ideals 
$\mf{C}_i(\boldsymbol{c};\boldsymbol{h})$
are independent of $\boldsymbol{h}$.
This independence becomes a key of the induction arguments
in the proof of our main results.

\begin{lem}\label{lemforindepmps}
Let $\boldsymbol{h}$ be a monic parameter system of $\Lambda$, 
and $i \in \bb{Z}$ an integer satisfying 
$0 \le i \le r$. 
We assume that 
$\boldsymbol{h}_{\le i-1} = \boldsymbol{x}_{\le i-1}$
if $i \ge 1$.
We put 
\[
\tilde{\boldsymbol{h}}:=
(\boldsymbol{x}_{ \le i}, \boldsymbol{h}_{ \ge i+1}) 
\in (\Lambda)^{r+1}. 
\]
Then, for any 
$\boldsymbol{m},\boldsymbol{m}',\boldsymbol{m}'' \in (\bb{Z}_{>0})^{r+1}$
satisfying $I(\boldsymbol{h}^{\boldsymbol{m}''}) \subseteq 
I(\tilde{\boldsymbol{h}}^{\boldsymbol{m}'}) \subseteq I(\boldsymbol{h}^{\boldsymbol{m}})$, 
we have 
\[
\mca{KI}(\boldsymbol{c};I(\boldsymbol{h}^{\boldsymbol{m}''});n)_{
I(\tilde{\boldsymbol{h}}^{\boldsymbol{m}'})}
=\mca{KI}(\boldsymbol{c};I(\tilde{\boldsymbol{h}}^{\boldsymbol{m}'});n)
\]
and 
\[
\mca{KI}(\boldsymbol{c};I(\tilde{\boldsymbol{h}}^{\boldsymbol{m}'})
;n)_{I(\boldsymbol{h}^{\boldsymbol{m}})}
=\mca{KI}(\boldsymbol{c};I(\boldsymbol{h}^{\boldsymbol{m}});n).
\]
\end{lem}

\begin{proof}
First, suppose that $i=0$.
Let $v_{\mca{O}}$ be the additive valuation on $\mca{O}$
normalized by $v_{\mca{O}}(\varpi)=1$, and 
put $\widetilde{\boldsymbol{m}}'' = (m_0'' \cdot v_{\mca{O}}(h_0), 
\boldsymbol{m}''_{\ge 1})$.
Then we have
$I(\tilde{\boldsymbol{h}}^{\widetilde{\boldsymbol{m}}''})=
I(\boldsymbol{h}^{\boldsymbol{m}''})$.
This implies that we have
\[
\mca{KI}(\boldsymbol{c};
I(\tilde{\boldsymbol{h}}^{\widetilde{\boldsymbol{m}}''});n)
=\mca{KI}(\boldsymbol{c};I(\boldsymbol{h}^{\boldsymbol{m}''});n).
\]
So by Corollary \ref{corascdesc}, 
we deduce the assertion of Lemma \ref{lemforindepmps}
for $i=0$.

Next, let us assume that $i \ge 1$.
By the definition of monic parameter systems, 
there exist a  positive integer $\delta$ 
and elements $g_0, g_1, \dots , g_{\delta -1} \in 
\mf{m}_{\Lambda^{(i-1)}}$ satisfying
\[
h_i= x_i^{\delta}+ \sum_{j=0}^{\delta-1} g_j x_i ^{j}.
\] 
Here, $\mf{m}_{\Lambda^{(i-1)}}$ 
denotes the maximal ideal of $\Lambda^{(i-1)}=\mca{O}[[x_1,\dots, x_{i-1}]]$.
Let $M$ be a positive integer satisfying 
$p^M \ge \boldsymbol{m}''$, $\delta p^M > \boldsymbol{m}'$
and $\mf{m}_{\Lambda^{(i-1)}}^M \subseteq 
I(\boldsymbol{x}_{\le i-1}^{\boldsymbol{m}''_{\le i-1}})$. 
Then, we have
\begin{equation}\label{congindepMPS}
h_i^{p^M} \equiv x_i^{\delta p^M} \mod 
\mf{m}_{\Lambda^{i-1}}^M\Lambda.
\end{equation}
Define elements $\boldsymbol{\nu}, \widetilde{\boldsymbol{\nu}}
\in (\bb{Z}_{>0})^{r+1}$ by
\begin{align*}
\boldsymbol{\nu} &= (\boldsymbol{m}''_{\le i-1},p^M, 
\boldsymbol{m}''_{\ge i+1}), \\
\tilde{\boldsymbol{\nu}} &= (\boldsymbol{m}''_{\le i-1}, \delta p^M, 
\boldsymbol{m}''_{\ge i+1}).
\end{align*}
By the congruence (\ref{congindepMPS}),
we obtain $I(\tilde{\boldsymbol{h}}^{\tilde{\boldsymbol{\nu}}})
=I(\boldsymbol{h}^{\boldsymbol{\nu}})$ and
\[
\mca{KI}(\boldsymbol{c};
I(\tilde{\boldsymbol{h}}^{\tilde{\boldsymbol{\nu}}});n)
=\mca{KI}(\boldsymbol{c};I(\boldsymbol{h}^{\boldsymbol{\nu}});n).
\]
Hence the assertion of Lemma \ref{lemforindepmps}
follows from Corollary \ref{corascdesc}.
\end{proof}

\begin{lem}\label{lemforindepmps2}
Let $\boldsymbol{h}$ be a monic parameter system of $\Lambda$.
Then, for any 
$\boldsymbol{m},\boldsymbol{m}',\boldsymbol{m}'' \in (\bb{Z}_{>0})^{r+1}$
satisfying $I(\boldsymbol{h}^{\boldsymbol{m}''}) \subseteq 
I(\boldsymbol{x}^{\boldsymbol{m}'}) \subseteq I(\boldsymbol{h}^{\boldsymbol{m}})$, 
we have 
\[
\mca{KI}(\boldsymbol{c};I(\boldsymbol{h}^{\boldsymbol{m}''});n)_{
I(\boldsymbol{x}^{\boldsymbol{m}'})}
=\mca{KI}(\boldsymbol{c};I(\tilde{\boldsymbol{h}}^{\boldsymbol{m}'});n)
\]
and 
\[
\mca{KI}(\boldsymbol{c};I(\boldsymbol{x}^{\boldsymbol{m}'})
;n)_{I(\boldsymbol{h}^{\boldsymbol{m}})}
=\mca{KI}(\boldsymbol{c};I(\boldsymbol{h}^{\boldsymbol{m}});n)
\]
\end{lem}

\begin{proof}
For each $i \in \bb{Z}$ with $0 \le i \le r$, 
we define 
$\boldsymbol{h}^{(i)}:=(\boldsymbol{x}_{\le r-i}, \boldsymbol{h}_{r-i+1})$. 
Note that in particular, we have $\boldsymbol{h}^{(0)}:=\boldsymbol{x}$ 
and $\boldsymbol{h}^{(r)}:=\boldsymbol{h}$.
Let $\{ N_i \}_{i=0}^{r}$ be a sequence of integers
satisfying 
$I(\boldsymbol{h}^{\boldsymbol{m}''}) \supseteq I(\boldsymbol{h}^{N_{r}})$
and $I((\boldsymbol{h}^{(i)})^{N_{i}}) 
\supseteq I((\boldsymbol{h}^{(i-1)})^{N_{i-1}})$
for any integer $i$ with $1 \le i \le r$.
Then, by using Lemma \ref{lemforindepmps}, 
we deduce, via induction on $i$, 
that 
\begin{eqnarray}\label{indepmprind}
\mca{KI}(\boldsymbol{c};I(\boldsymbol{x}^{N_0});n)_{I((\boldsymbol{h}^{(i)})^{N_i})}
=\mca{KI}(\boldsymbol{c};I((\boldsymbol{h}^{(i)})^{N_i});n)
\end{eqnarray}
for any $0 \le i \le r$. 
Hence, by Corollary \ref{corascdesc}, 
the equality (\ref{indepmprind}) for $i=r$ implies 
the assertion of our lemma.
\end{proof}

\begin{prop}\label{propindepofparam}
For any monic parameter system $\boldsymbol{h}$ of $\Lambda$ 
and any $i \in \bb{Z}_{\ge 0}$,
we have
\(
\mf{C}_i(\boldsymbol{c};\boldsymbol{h})= \mf{C}_i(\boldsymbol{c})
\). 
Namely, the ideal $\mf{C}_i(\boldsymbol{c};\boldsymbol{h})$ 
is independent of the choice of $\boldsymbol{h}$.
\end{prop}

\begin{proof}
Let $\{ N_\nu \}_{\nu>0}$ be an increasing sequence of 
positive integers 
which satisfies $I(\boldsymbol{x}^{N_{2j-1}}) \supseteq 
I(\boldsymbol{h}^{N_{2j}})
\supseteq I(\boldsymbol{x}^{N_{2j+1}})$ 
for any $j \in \bb{Z}_{\ge 0}$.
We put $I_{2j-1}:=I(\boldsymbol{x}^{N_{2j-1}})$ and 
$I_{2j}:=I(\boldsymbol{h}^{N_{2j}})$ for any $j \in \bb{Z}_{>0}$.
Then, Lemma \ref{lemforindepmps2} implies that
$\{ \mf{C}_i(\boldsymbol{c};I_\nu) \}_{\nu >0}$ forms a projective system. 
Hence we obtain
\[
\mf{C}_i(\boldsymbol{c};\boldsymbol{h})= \varprojlim_{j} \mf{C}_i(\boldsymbol{c};I_{2j})
=\varprojlim_{\nu} \mf{C}_i(\boldsymbol{c};I_{\nu})
=\varprojlim_{j} \mf{C}_i(\boldsymbol{c};I_{2j-1})
=\mf{C}_i(\boldsymbol{c};\boldsymbol{x}).
\]
This completes the proof of Proposition \ref{propindepofparam}.
\end{proof}

\subsection{Extension of Scalars}\label{ssscalarext}
Here, we observe the behavior of 
the ideals $\mf{C}_i(\boldsymbol{c})$
along extensions of the rings of constants.

Let $F'$ be a finite extension field of $F$, and 
$\mca{O}'$ the ring of integers of $F'$. 
We denote the ramification index of $F'/F$ by $e$.
Fix a uniformizer $\varpi' \in \mca{O}'$.
We put $\Lambda_\mca{O'}:=
\Lambda \otimes_{\mca{O}} \mca{O}'
=\mca{O}'[[x_1, \dots, x_{r}]]$, 
and $\Lambda_{\mca{O}'}:=
\Lambda \otimes_{\mca{O}} \mca{O}'$.
We define the new ``standard parameter system" 
\[
\boldsymbol{x}'= (x'_0,\dots, x'_{r}):=(\varpi', x_1, \dots, x_{r}).
\]
We define $\bb{T}_{{\mca{O}}'}:= 
\bb{T} \otimes_{\mca{O}} \mca{O}'$.
The Euler system $\boldsymbol{c}$ induces an Euler system $\boldsymbol{c}':=\boldsymbol{c} \otimes 1$
for $\bb{T}_{\mca{O}'}$.

Let $\boldsymbol{m}=(m_0, \dots, m_{r}) \in (\bb{Z}_{>0})^{r+1}$ be any element, 
and put $\boldsymbol{m}':=(em_0, \boldsymbol{m}_{\ge 1})$.
and denote by $I_{\mca{O}'}(\boldsymbol{x}'^{\boldsymbol{m}'})$ 
the ideal of $\Lambda$ generated by 
$\{ x'^{m_i}_i \mathrel{\vert} 0 \le i \le r \}$.
Then, clearly we have
\[
I_{\mca{O}'}(\boldsymbol{x}'^{\boldsymbol{m}'})=I(\boldsymbol{x}^{\boldsymbol{m}})\Lambda_\mca{O}'
=I(\boldsymbol{x}^{\boldsymbol{m}})\mca{O}'.
\]
Put $I:=I(\boldsymbol{x}^{\boldsymbol{m}})$ and 
$I':=I_{\mca{O}'}(\boldsymbol{x}'^{\boldsymbol{m}'})_{\mca{O}'}$.
Since $\mca{O}'$ is 
a free $\mca{O}$-module of finite rank, 
we have an isomorphism
\begin{eqnarray}\label{HiOO'}
H^i(\pi_{I'}^*\bb{T}_\mca{O}') \simeq
H^i(\pi_{I}^*\bb{T}\otimes_{\mca{O}} \mca{O}') 
\simeq H^i(\pi_{I}^*\bb{T}) \otimes_{\mca{O}}\mca{O}' 
\end{eqnarray}
of $\Lambda \otimes_{\mca{O}}\mca{O}'$-modules
for any $i \in \bb{Z}_{\ge 0}$.
Let $\ell$ be a prime number not contained in $\Sigma$.
Since $\mca{O}'$ is faithfully flat over $\mca{O}$, 
the $\mca{O}'$-module 
\[
\bb{T}_{\mca{O}'}/(I'\bb{T}_{\mca{O}'}+(\Frob_\ell -1)\bb{T}_{\mca{O}'})
\simeq \left( \bb{T}/(I\bb{T}+(\Frob_\ell -1)\bb{T}) 
\right) \otimes_{\mca{O}} \mca{O}'
\]
is free of rank one
if and only if the $\mca{O}$-module 
$\bb{T}/(I\bb{T}+(\Frob_\ell -1)\bb{T})$ 
is free of rank one. 
So, we obtain 
\begin{eqnarray}\label{NOO'comparison}
\mca{N}(\bb{T},I)=\mca{N}(\bb{T}_{\mca{O}'},I').
\end{eqnarray}
For any $n \in \mca{N}(\bb{T},I)$,
we denote by $\mf{C}(\boldsymbol{c}';I';n)_{\mca{O}'}$
the ideal of $\Lambda_{\mca{O}'}/I'$ 
constructed in Definition \ref{dfnidealC(n)} (iv) for the new data 
\[
(\Lambda_{\mca{O}'}, 
\Lambda_{\mca{O}'}, 
\bb{T}_{\mca{O}'} , \boldsymbol{c} \otimes 1,\boldsymbol{x}').
\]
Similarly, we define $\mf{C}_i(\boldsymbol{c}';I')_{\mca{O}'}$
and $\mf{C}_i(\boldsymbol{c})_{\mca{O}'}$.
We can easily check that the new data satisfies all conditions
required in Theorem \ref{thmuncondresults} 
if the old data 
$(\Lambda, \bb{T}, \boldsymbol{c})$ satisfies them.

\begin{lem}\label{lemCiOO'1}
The following hold.
\begin{enumerate}[{\rm (i)}]
\item For any $n \in \mca{N}(\bb{T},I)$,
we have $\mf{C}(\boldsymbol{c}';I';n)_{\mca{O}'}=
\mf{C}(\boldsymbol{c};I;n){\mca{O}'}$.
\item For any $i \in \bb{Z}_{\ge 0}$,
we have $\mf{C}_i(\boldsymbol{c}';I')_{\mca{O}'}=
\mf{C}_i(\boldsymbol{c};I){\mca{O}'}$.
\end{enumerate}
\end{lem}

\begin{proof}
By the equality (\ref{NOO'comparison}), 
the assertion (i) implies the assertion (ii).
So, it suffices to show the assertion (i).
We have 
\begin{align*}
\Hom_{\Lambda_{\mca{O}'}/I'}(H^1(\pi_{I'}^*\bb{T}_\mca{O}') 
,\Lambda_{\mca{O}'}/I')
&=\Hom_{\Lambda/I}(H^1(\pi_{I}^*\bb{T}) 
,\Lambda_{\mca{O}'}/I') \\
& =\Hom_{\Lambda/I}(H^1(\pi_{I}^*\bb{T}) 
,\Lambda/I)\otimes_{\mca{O}}\mca{O}'.
\end{align*}
Indeed, the first equality follows from (\ref{HiOO'})
for $i=1$, and
the second equality holds since
$\Lambda_{\mca{O}'}/I'=\Lambda/I\otimes_{\mca{O}}\mca{O}'$
is a free $\Lambda/I$-module.
Hence by the definition of the ideals
$\mf{C}(\boldsymbol{c};I;n)$ and
$\mf{C}(\boldsymbol{c}';I';n)_{\mca{O}'}$, 
we obtain the assertion (i).
\end{proof}

Here, let us prove Proposition \ref{corOO'} below, 
which states that the ideals $\mf{C}_i(\boldsymbol{c})$
is compatible with base change and descent arguments 
along extension of the coefficient rings. 
We need some lemmas.

\begin{lem}\label{lemROO'descent}
Let $J$ be an ideal of $\Lambda/I$.
Then, as a subset of $\Lambda_{\mca{O}'}/I'$
we have $J=J\Lambda_{\mca{O}'}/I' \cap \Lambda/I$.
\end{lem}

\begin{proof}
Put $\tilde{J}:=J\Lambda_{\mca{O}'}/I' \cap \Lambda/I$.
Clearly, we have  $J \subseteq \tilde{J}$, 
and $\tilde{J}\Lambda_{\mca{O}'}/I'=J\Lambda_{\mca{O}'}/I'$.
Since $\mca{O}'$ is flat over $\mca{O}$, we have
\[
J\Lambda_{\mca{O}'}/I' = J\otimes_{\mca{O}}\mca{O}' \subseteq
\Lambda_{\mca{O}'}/I'=\Lambda/I\otimes_{\mca{O}}\mca{O}'.
\]
Similarly, we have  
$\tilde{J}\Lambda_{\mca{O}'}/I' = \tilde{J}\otimes_{\mca{O}}\mca{O}'$, 
so we obtain 
$J\otimes_{\mca{O}}\mca{O}'= \tilde{J}\otimes_{\mca{O}}\mca{O}'$.
The ring $\mca{O}'$ is faithfully flat over $\mca{O}$, so we obtain 
$J= \tilde{J}$.
\end{proof}

By Lemma \ref{lemCiOO'1} and Lemma \ref{lemROO'descent}, 
we immediately obtain the following lemma.

\begin{lem}\label{lemCiOO'2}
The following hold.
\begin{enumerate}[{\rm (i)}]
\item For any $n \in \mca{N}(\bb{T},I)$,
we have $\mf{C}(\boldsymbol{c};I;n)=
\mf{C}(\boldsymbol{c}';I';n)_{\mca{O}'} \cap \Lambda/I$.
\item For any $i \in \bb{Z}_{\ge 0}$,
we have $\mf{C}_i(\boldsymbol{c};I)_{\mca{O}}=
\mf{C}_i(\boldsymbol{c}';I')_{\mca{O}'}$.
\end{enumerate}
\end{lem}

By the above arguments, we deduce the following.

\begin{prop}\label{corOO'}
We have 
$\mf{C}_i(\boldsymbol{c}')_{\mca{O}'}
=\mf{C}_i(\boldsymbol{c}){\mca{O}'}$
and 
$\mf{C}_i(\boldsymbol{c})
=\mf{C}_i(\boldsymbol{c}')_{\mca{O}'} \cap \Lambda'
$
for any $i \in \bb{Z}_{\ge 0}$. 
\end{prop}

In the end of this section, we give a remark on 
the behavior of the Selmer group side
along the extension of the rings of coefficients. 
The following holds.

\begin{prop}\label{lemfittscext}
We have
\[
\Fitt_{\Lambda_{\mca{O}'},i}(X( \bb{T}_{\mca{O}'}))
=\Fitt_{\Lambda,i}(X(\bb{T}))\Lambda_{\mca{O}'}.
\]
\end{prop}

\begin{proof}
Since $\mca{O}'$ is 
a free $\mca{O}$-module of finite rank, 
we have an isomorphism
\[
H^i(\pi_{I,[n]}^*\bb{T}\otimes_{\mca{O}} \mca{O}') 
\simeq H^i(\pi_{I,[n]}^*\bb{T}) \otimes_{\mca{O}}\mca{O}' 
\]
of $\Lambda_{\mca{O}'}$-modules
for any $i \in \bb{Z}_{\ge 0}$, 
any $n \in \mca{N}_{\Sigma}$ and 
any ideal $I$ of $\Lambda$.
So, by definition of 
the ideal $\mf{C}_i(\boldsymbol{c})$
and the base change property of Fitting ideals,
we obtain the assertion of Proposition \ref{lemfittscext}. 
\end{proof}

\subsection{Affine transformations}

Here, we introduce affine transformations on
the ring $\Lambda=\Lambda^{(r)}$, 
and show that the ideals $\mf{C}_i(\boldsymbol{c})$
are stable under affine transformations.
(Note that we do not use this property 
in the proof of our main results.)

Let us define affine transformations.
For any $A \in \mathrm{GL}_{r}(\mca{O})$
and any $v \in (\varpi \mca{O})^{\oplus r}$,
we define an automorphism
\[
T(A,v) \colon \Lambda^{(r)}
\xrightarrow{\ \simeq\ } \Lambda^{(r)};\ 
f(\boldsymbol{x}) \longrightarrow f(A\boldsymbol{x}+v),
\]
where we regard $\boldsymbol{x}=(x_i)_{i=1}^r$ and $v$
as column vectors.
We call this automorphism 
\textit{an affine transformation} 
on $\Lambda=\Lambda^{(r)}$.

Now, we introduce certain special affine transformations
called elementary affine transformations.
First, let us recall the definition of elementary matrices.
\begin{enumerate}[(i)]
\item For any $u \in \mca{O}^\times$ and 
any $\nu \in \bb{Z}$ 
with $1 \le \nu \le r$, we define a matrix
\(
P_\nu (u)=(c_{ij})_{i,j} \in \mathrm{GL}_{r}(\mca{O})
\)
by
\[
c_{ij}:=\begin{cases}
u & (\text{if $i=j=\nu$}) \\
1 & (\text{if $i=j \ne \nu$}) \\
0 & (\text{if $i\ne j$}).
\end{cases}
\] 
\item Let $\nu,\mu \in \bb{Z}$ be distinct integers
with $1 \le \nu,\mu \le r$.
Then, we define 
a matrix
\(
Q_{\mu,\nu}=
(c_{ij})_{i,j} \in \mathrm{GL}_{r}(\mca{O})
\)
by
\[
c_{ij}:=\begin{cases}
1 & (\text{if $i=j$, and if $i \notin \{ \mu, \nu \}$}) \\
1 & (\text{if $(i,j)=(\mu,\nu),(\nu,\mu)$}) \\
0 & (\text{otherwise}).
\end{cases}
\] 
\item Let $\nu,\mu \in \bb{Z}$ be integers
with $1 \le \nu,\mu \le r$ and $\mu > \nu$.
For any $a \in \mca{O}$, 
we define a matrix
\(
R_{\mu,\nu} (a)=
(c_{ij})_{i,j} \in \mathrm{GL}_{r}(\mca{O})
\)
by
\[
c_{ij}:=\begin{cases}
1 & (\text{if $i=j$}) \\
a & (\text{if $(i,j)=(\mu,\nu)$}) \\
0 & (\text{otherwise}).
\end{cases}
\] 
\end{enumerate}
The matrices of the form 
$P_{\nu}(u)$, $Q_{\mu,\nu}$ 
or $R_{\mu,\nu}(a)$ 
are called \textit{elementary matrices}.
Note that since $\mca{O}$ is a local ring, 
any element $A \in \mathrm{GL}_{r}(\mca{O})$
is decomposed into a product of elementary matrices. 
For any $a \in \varpi\mca{O}$ and 
any $\nu \in \bb{Z}$ with $1 \le \nu \le r$, 
we define an element $\delta_\nu(a)=(v_i)_{i=1}^{r} 
\in (\varpi\mca{O})^{\oplus r}$ by 
$v_\nu:=a$ and $v_j:=0$ if $j\ne \nu$.
An affine transformation $T(A,v)$ is called 
\textit{an elementary affine transformation}
if the pair $(A,v)$ is one of the following:
\begin{itemize}
\item The matrix $A$ is elementary, and $v=0$
\item The matrix $A$ is the identity matrix, 
and $v=\delta_i(a)$ for some 
$a \in \varpi\mca{O}$ and 
$\nu \in \bb{Z}$ with $1 \le \nu \le r$.
\end{itemize}
Note that any affine transformation is 
a composite of finitely many elementary 
affine transformations.

\begin{prop}\label{lemafftrans}
Let $A \in \mathrm{GL}_{r}(\mca{O})$ and
$v=(v_i)_{i=1}^{r} \in (\varpi \mca{O})^{\oplus r}$ 
be arbitrary elements.
We put 
\[
\boldsymbol{y}=(y_i)_{i=1}^{r}:=A\boldsymbol{x}+v 
\in (\mf{m}_{\Lambda^{(r)}})^{\oplus r}. 
\]
Then, for any $i \in \bb{Z}_{\ge 0}$,
we have 
\[
T(A,v)(\mf{C}_i(\boldsymbol{c};\boldsymbol{x}))
=\mf{C}_i(\boldsymbol{c};\boldsymbol{y})=\mf{C}_i(\boldsymbol{c};\boldsymbol{x}).
\]
\end{prop}

\begin{proof}
By the definition of the ideal 
$\mf{C}_i(\mf{c};\boldsymbol{x})$,
the first equality 
\[
T(A,v)(\mf{C}_i(\boldsymbol{c};\boldsymbol{x}))
=\mf{C}_i(\boldsymbol{c};\boldsymbol{y})
\]
is clear.
Let us prove the second equality.
Since any affine transformation is 
decomposed into a composite of elementary 
affine transformations, we may assume that 
$T(A,v)$ is elementary.

Let us consider the case when $v=0$.
First, we assume that $A$ is equal to $P_\mu(u)$ or $Q_{\mu,\nu}$. 
In this case, we have $I(\boldsymbol{x}^N)=I(\boldsymbol{y}^N)$
for any $N \in \bb{Z}_{\ge 1}$.
So, by definition, we have $\mf{C}_i(\boldsymbol{c};I(\boldsymbol{x}^N))=
\mf{C}_i(\boldsymbol{c};I(\boldsymbol{y}^N))$. 
This implies that
$\mf{C}_i(\boldsymbol{c};\boldsymbol{y})
=\mf{C}_i(\boldsymbol{c};\boldsymbol{x})$.
Next, let us assume that $A=R_{\mu,\nu}(a)$.
Then, the system
\[
\boldsymbol{y}=(\varpi, x_1, \dots, x_{\mu-1}, 
x_\mu+a x_\nu, x_{\mu+1}, \dots x_{r})
\]
becomes a monic parameter system. 
Hence by Proposition \ref{propindepofparam}, 
we obtain $\mf{C}_i(\boldsymbol{c};\boldsymbol{y})=\mf{C}_i(\boldsymbol{c};\boldsymbol{x})$.

Let us consider the case when $A=1$.  
Put $v=\delta_\nu (a)$.
In this case, the system
\[
\boldsymbol{y}=(\varpi, x_1, \dots,x_{\nu-1}, x_\nu+a, 
x_{nu+1}, \dots x_{r})
\]
forms a monic parameter system, so Lemma \ref{propindepofparam} implies that 
$\mf{C}_i(\boldsymbol{c};\boldsymbol{y})
=\mf{C}_i(\boldsymbol{c};\boldsymbol{x})$.
\end{proof}

\section{Reduction of $\mf{C}_i(\boldsymbol{c})$}\label{secred}

Let $\Lambda=\Lambda^{(r)}$ and 
$(\bb{T}, \boldsymbol{c})$ be as in Theorem \ref{thmuncondresults}.
In particular, we assume that 
$\boldsymbol{c}$ satisfies (NV).
In the proof of our main results, namely 
Theorem \ref{thmuncondresults}, 
Theorem \ref{thmcondresults1} 
and Theorem \ref{thmcondresults2}, 
a certain properties of the ideals $\mf{C}_i(\boldsymbol{c})$
related to the specialization of the coefficient ring 
called \textit{the weak/strong specialization compatibility}
become a key. 
Roughly speaking, 
the specialization weak (resp.\ strong) compatibility
says that if the reduction map 
$\pi_I \colon \Lambda \longrightarrow \Lambda/I \simeq \Lambda^{(r-1)}$ 
for a certain ideal $I$ is given, then for any $i \in \bb{Z}_{\ge 0}$, 
the image of $\mf{C}_i(\boldsymbol{c})$ by $\pi_I$ 
is contained in 
(resp.\ coincides with) the ideal $\mf{C}_i(\pi_I^* \boldsymbol{c})$.
$\mf{C}_i(\pi_I^* \boldsymbol{c})$ 
defined by the data 
$(\pi_{I}^*\bb{T}, 
\pi_I^*\boldsymbol{c},\mf{N}(\pi_I^*\bb{T};\boldsymbol{x}_{\le r-1}))$.
In this section, we shall study the specialization compatibilities of ideals 
$\mf{C}_i(\boldsymbol{c})$.

In \S \ref{sswsc}, we shall prove
Proposition \ref{propredonevar}, that is, 
the weak specialization compatibility 
for general multi-variable cases.
In \S \ref{ssrmkononevarcase}, we will show 
the strong compatibility in one variable cases, 
namely Theorem \ref{thmonevarcompletered}.
Note that the proof of Theorem \ref{thmonevarcompletered} 
is the most technical part of this article.
In \S \ref{ssrmkoncyclotcase}, 
we shall prove Theorem \ref{thmcyclotcompletered},
which asserts that the strong compatibilities hold
in the case when $\bb{T}$ is a cyclotomic 
deformation of a one variable deformation.

\subsection{Weak specialization compatibility}\label{sswsc}
Here, let us study the weak specialization compatibilities.
We need the notion of linear elements 
in the sense of Ochiai's article \cite{Oc2}.

\begin{dfn}\label{defnLE}
\textit{A linear element} $g$ in $\Lambda=\Lambda^{(r)}$
is a polynomial written in a form
\[
g= a_0 + \sum_{i=1}^{r} a_i x_i \in \Lambda,
\]
where $a_0 \in \varpi \mca{O}$, 
and $(a_1, \dots, a_{r}) \in \mca{O}^{\oplus r} 
\setminus (\varpi \mca{O})^{\oplus r}$.
\end{dfn}

As in \S Notation, namely, the end of \S \ref{secintro}, 
for each ideal $J$ of $\Lambda$, 
we denote by $\mf{C}_i(\boldsymbol{c})_J$
the image of $\mf{C}_i(\boldsymbol{c})$
in $\Lambda/J$.

\begin{prop}\label{propredonevar}
Let $\boldsymbol{h}=(\boldsymbol{x}_{\le r-1},h)$ be 
a monic parameter system of $\Lambda$ 
such that $h$ is a linear element.
We put 
\(
I:=h\Lambda=
I(\boldsymbol{h}^{(\infty, \dots, \infty, 1)})
\).
Let $\pi_I \colon \Lambda \longrightarrow \Lambda/I 
\simeq \Lambda^{(r-1)}$ be 
the reduction map.
Then, We have 
\[
\mf{C}_i(\boldsymbol{c})_I
\subseteq 
\mf{C}_i(\pi_I^* \boldsymbol{c})
:=
\mf{C}_i(\pi_I^* \boldsymbol{c};\boldsymbol{x}_{\le r-1};
\mf{N}(\pi_I^*\bb{T};\boldsymbol{x}_{\le r-1}))
\]
\end{prop}

\begin{proof}
Note that by definition, we have
\[
\mf{N}(\pi_I^*\bb{T};\boldsymbol{x}_{\le r-1})=
\{
\mf{N}(I(\boldsymbol{h}^{(\boldsymbol{m},1)})) 
\}_{\boldsymbol{m}=(m_0, \dots, m_{r-1}) \in (\bb{Z}_{>0})^{r}}.
\]
For any $N \in \bb{Z}_{>0}$, we put 
$\boldsymbol{m}(N):=(N , \dots, N, 1) \in (\bb{Z}_{>0})^{r+1}$.
Then, by definition, we have 
\[
\mf{C}_i(\pi_I^* \boldsymbol{c};\boldsymbol{x};\pi_I^*
\mf{N}(\bb{T};\boldsymbol{h})) = 
\varprojlim_{N} \mf{C}_i
(\boldsymbol{c};I(\boldsymbol{h}^{\boldsymbol{m}(N)});
\mf{N}(\bb{T};\boldsymbol{h})).
\]
So, Proposition \ref{propindepofparam} and
Corollary \ref{corascdesc} imply our proposition.
\end{proof}

\subsection{Strong compatibility for one variable cases}
\label{ssrmkononevarcase}

Here, we set 
$\Lambda=\Lambda^{(1)}=\mca{O}[[x_1]]$. 
Let prove the following theorem, which says that 
in one variable cases, 
the strong specialization compatibilities  hold.

\begin{thm}\label{thmonevarcompletered}
Let $a \in \mf{m}_{O}=\varpi\mca{O}$ be any element. 
Then, for any $i \in \bb{Z}_{\ge 0}$, 
the image of $\mf{C}_i(\boldsymbol{c})$ in $\Lambda/I(0,x_1-a)$
coincides with  
\[
\mf{C}_i(\pi_{I(0,x_1-a)}^*\boldsymbol{c})
:=\mf{C}_i(\pi_{I(0,x_1-a)}^*\boldsymbol{c}; \boldsymbol{x};
\mca{N}(\pi_{I(0,x_1-a)}^*\bb{T})).
\]
\end{thm}

Let $i \in \bb{Z}_{\ge 0}$ be any element. 
Recall that we assume that 
$\boldsymbol{c}$ satisfies (NV).
So, we have a non-negative integer $c$
such that 
\[
\mf{C}_i(\pi_{I(0,x_1-a)}^*\boldsymbol{c})
\supseteq 
\mf{C}_0(\pi_{I(0,x_1-a)}^*\boldsymbol{c})
=\varpi^c \mca{O}.
\]
By in \S \ref{ssconstC_i}, 
in order to prove Theorem \ref{thmonevarcompletered},
it suffices to show the following Proposition.

\begin{prop}\label{proponevarcompletered}
Fix $a \in \mf{m}_{O}$. 
Let $m_0 \in \bb{Z}_{>0}$ and 
$\boldsymbol{m}'=(m'_0,m'_1) \in \bb{Z}_{> 0}^{2}$
be elements satisfying $m'_0 \ge m_0>c$.
We put $I:=I(\varpi^{m_0},x_1-a)$ and 
$I':=I(\varpi^{m'_0},(x_1-a)^{m'_1})$. 
Then, we have 
\[
\mf{C}_i(\boldsymbol{c};I')_I
\supseteq \mf{C}_i(\boldsymbol{c};I). 
\]
\end{prop}

\begin{proof}
By  (\ref{compC_iI'C_iI}), 
we may replace $\boldsymbol{m}'$ 
with suitable larger one, 
and assume that $m'_1$ is prime to $p$, 
and that $m'_0=m_0m'_1$.
Moreover, by Lemma \ref{lemCiOO'2} we may assume that 
$F$ contains a primitive $m'_1$-th root of unity.
We also assume that 
Fix a generator $\bar{b}$ of the cyclic group $k^\times$, and
let $b \in \mca{O}$ be a lift of $\bar{b}$.
Then, we fix a $m'_1$-th root $\beta \in \overline{F}$ of $b$,
and put $F':=F(\beta)$.
The ring of integers of $F'$ is denoted by $\mca{O}'$.
Note that $F'/F$ is an unramified extension
with $[F':F]=m'_1$.
We define a homomorphism
\[
e \colon \Lambda \longrightarrow \mca{O}';\ 
x_1 \longmapsto a + \varpi^{m_0}\beta
\] 
of $\mca{O}$-algebras. The kernel of $e$ is 
generated by 
the irreducible distinguished polynomial 
\[
g:=(x_1-a)^{m'_1}-\varpi^{m'_0}b
\in \mca{O}[x_1].
\]
By definition, 
we have $I(\varpi^{m'_0},g)=I'$, and $e$ induces an injection
\[
\bar{e}_{m'_0}
\colon \Lambda/I' \hookrightarrow \mca{O}'/\varpi^{m'_0}\mca{O}'; \ 
x_1\ \mathrm{mod}\ I' 
\longmapsto a + \varpi^{m_0}\beta 
\ \mathrm{mod} \varpi^{m'_0}\mca{O}'.
\]
Similarly, we have an injection
$\bar{e}_{m_0}\colon \Lambda/I \hookrightarrow 
\mca{O}'/\varpi^{m_0}\mca{O}'$ 
given by 
\[
\bar{e}_{m_0}(x_1\ \mathrm{mod}I)=
a\ \mathrm{mod}\ \varpi^{m_0}\mca{O}'=
a + \varpi^{m_0}\beta\ \mathrm{mod}\ \varpi^{m_0}\mca{O}'.
\]
Then, by definition, we obtain a commutative diagram
\begin{eqnarray}\label{diagiotam0m1pi12}
\xymatrix@C=15mm{
\Lambda/I'\Lambda \ar@{^{(}->}[r]^(0.42)
{\bar{e}_{m'_0}} \ar[d]_{\pi_1} & 
(\mca{O}'/\varpi^{m'_0}\mca{O}')  
\ar[d]^{\pi_2} \\ 
\Lambda/I\Lambda \ar@{^{(}->}[r]^(0.42){\bar{e}_{m_0}}  & 
(\mca{O}'/\varpi^{m_0}\mca{O}'),
}
\end{eqnarray}
where $\pi_1$ and $\pi_2$ are natural surjections.

We put 
\(
\mca{N}(m):=\mca{N}(e^*\bb{T};\varpi^m\mca{O}')
\) 
for any $m \in \bb{Z}_{>0}$.
Then, the following lemma holds.

\begin{lem}\label{lemNm'0NCTI}
We have 
$\mca{N}(m_0)=\mca{N}(\bb{T};I)$ and
$\mca{N}(m'_0)=\mca{N}(\bb{T};I')$.
\end{lem}

\begin{proof}
First, let us show first equality
$\mca{N}(m_0)=\mca{N}(\bb{T};I)$. 
Let $\ell$ be a prime number 
not contained in $\Sigma$, and put
\[
\bb{T}_\ell:= \bb{T}/(\Frob_\ell -1)\bb{T}.
\]
Since the ring $\mca{O}'/\varpi^{m_0}\mca{O}'$
is faithfully flat over 
$\Lambda/I \simeq \mca{O}/\varpi^{m_0}\mca{O}$,
the $\Lambda/I$-module 
$\bb{T}_\ell/I\bb{T}_\ell$ 
is free of rank one
if and only if 
the $\mca{O}'/\varpi^{m_0}\mca{O}'$-module 
\[
e^*\bb{T}_\ell/\varpi^{m_0}e^*\bb{T}_\ell
\simeq e^*(\bb{T}_\ell/I\bb{T}_\ell)
\]
is free of rank one. 
Hence
the first equality follows from the definition of 
$\mca{N}(m_0)$ and $\mca{N}(\bb{T};I)$.

Let us show the second equality, 
namely $\mca{N}(m'_0)=\mca{N}(\bb{T};I')$. 
We denote 
the composite 
\[
\bar{e}_{m'_0} \circ \pi_{I'} \colon 
\Lambda \longrightarrow \Lambda/I' 
\longrightarrow \mca{O}'/(\varpi^{m'_0})
\]
by $e_{m'_0}$.
Let $\ell$ be a prime number 
not contained in $\Sigma$.
In order to show the second equality of our lemma,
it suffices to show that
the $\Lambda/I$-module
$\pi_{I'}^*\bb{T}_\ell$ is free of rank one 
if and only if the $\mca{O}'/(\varpi^{m'_0})$-module
$e_{m'_0}^*\bb{T}_\ell$ is free of rank one.
Note that by Example \ref{exaFittlocring}, the $\Lambda/I'$-module
$\pi_{I'}^*\bb{T}_\ell$ is free of rank one 
if and only if $\Fitt_{\Lambda/I',0}
(\pi_{I'}^*\bb{T}_\ell)=\{0 \}$
and $\Fitt_{\Lambda/I',1}
(\pi_{I'}^*\bb{T}_\ell)=\Lambda/I'$.
Similarly, the $\mca{O}'/(\varpi^{m'_0})$-module
$e_{m'_0}^*\bb{T}_\ell$ is free of rank one 
if and only if $\Fitt_{\mca{O}'/(\varpi^{m'_0}),0}
(e_{m'_0}^*\bb{T}_\ell)=\{0 \}$
and $\Fitt_{\mca{O}'/(\varpi^{m'_0}),1}
(e_{m'_0}^*\bb{T}_\ell)=\mca{O}'/(\varpi^{m'_0})$.
Since $e_{m'_0}^*\bb{T}_\ell=\bar{e}_{m'_0}^*(\pi_{I'}^*\bb{T})$,
we have
\[
\Fitt_{\mca{O}'/(\varpi^{m'_0}),i}
(e_{m'_0}^*\bb{T}_\ell)
=\bar{e}_{m'_0}
\left(
\Fitt_{\Lambda/I',i}
(\pi_{I'}^*\bb{T}_\ell)
\right)
\mca{O}'/(\varpi^{m'_0})
\]
for any $i \in \bb{Z}_{\ge 0}$.
Note that 
since $\bar{e}_{m'_0}$ is injective,
we have 
$\Fitt_{\Lambda/I',0}
(\pi_{I'}^*\bb{T}_\ell)=\{0 \}$
if and only if $\Fitt_{\mca{O}'/(\varpi^{m'_0}),1}
(e_{m'_0}^*\bb{T}_\ell)=\{ 0 \}$.
Moreover, 
since $\bar{e}_{m'_0}$ is a homomorphism of local rings,
we have 
$\Fitt_{\Lambda/I',1}
(\pi_{I'}^*\bb{T}_\ell)=\Lambda/I'$
if and only if $\Fitt_{\mca{O}'/(\varpi^{m'_0}),1}
(e_{m'_0}^*\bb{T}_\ell)=\mca{O}'/(\varpi^{m'_0})$.
Hence we deduce that 
the $\Lambda/I$-module
$\pi_{I'}^*\bb{T}_\ell$ is free of rank one 
if and only if the $\mca{O}'/(\varpi^{m'_0})$-module
$e_{m'_0}^*\bb{T}_\ell$ is free of rank one.
This implies that $\mca{N}(m'_0)=\mca{N}(\bb{T};I')$.
\end{proof}

We need the following lemma and its corollary.

\begin{lem}\label{lemonevarOO'}
Let $n \in \mca{N}(m'_0)$ be any element.
Then, the following hold.
\begin{enumerate}[{\rm (i)}]
\item $\mf{C}({e}^*\boldsymbol{c};\varpi^{m'_0}\mca{O}';n)=
\bar{e}_{m'_0}(\mf{C}(\boldsymbol{c};I';n))\cdot 
{\mca{O}'/\varpi^{m'_0}\mca{O}'}$.
\item $\mf{C}({e}^*\boldsymbol{c};\varpi^{m_0}\mca{O}';n)=
\bar{e}_{m_0}(\mf{C}(\boldsymbol{c};I;n))
\cdot 
{\mca{O}'/\varpi^{m_0}\mca{O}'}$.
\item $\mf{C}(\boldsymbol{c};I';n)=
\bar{e}_{m'_0}^{-1}(\mf{C}({e}^*\boldsymbol{c};
\varpi^{m'_0}\mca{O}';n))$.
\item $\mf{C}(\boldsymbol{c};I;n)=
\bar{e}_{m_0}^{-1}(\mf{C}({e}^*\boldsymbol{c};
\varpi^{m_0}\mca{O}';n))$.
\end{enumerate}
\end{lem}

\begin{proof}
Let us show the assertion (i) of Lemma \ref{lemonevarOO'}.
First, we shall prove 
\begin{equation}\label{eqe*sup}
\mf{C}(e^*\boldsymbol{c};\varpi^{m'_0}\mca{O}';n)
\supseteq \bar{e}_{m'_0}(\mf{C}(\boldsymbol{c};I';n))\cdot 
{\mca{O}'/\varpi^{m'_0}\mca{O}'}.
\end{equation}
The map $\bar{e}_{m'_0}$ induces  an exact sequence
\[
0 \longrightarrow \pi_{I'}^*\bb{T} \xrightarrow{\ \bar{e}_{m'_0,\bb{T}} \ } 
\bar{e}_{m'_0}^* \pi_{I'}^* \bb{T} \longrightarrow 
\bb{T} \otimes_{\Lambda}\Coker(\bar{e}_{m'_0})
\longrightarrow 0,
\]
where $\bar{e}_{m'_0,\bb{T}}:=\bar{e}_{m'_0} \otimes \mathrm{id}_{\bb{T}}$.
The assumption (A3) implies that we have 
\[
H^0(\bb{Q},\bb{T} \otimes_{\Lambda}\Coker(\bar{e}_{m'_0}))=0.
\]
So, the map
$H^1(\bar{e}_{m'_0, \bb{T}}) \colon 
H^1(\pi_{I'}^*\bb{T}) \longrightarrow 
H^1(\bar{e}_{m'_0}^* \pi_{I'}^* \bb{T})$
induced by $\bar{e}_{m'_0,\bb{T}}$
is injective.
Note that by construction, 
we have 
\[
H^1(\bar{e}_{m'_0, \bb{T}})(
\kappa_n^{\mathrm{univ}}(\boldsymbol{c})_{I'} 
)
=
\kappa_n^{\mathrm{univ}}
({e}^*\boldsymbol{c})_{\varpi^{m'_0}\mca{O}'}.
\]
We regard $H^1(\pi_{I'}^*\bb{T})$ as an $\Lambda/I'$-submodule of 
$H^1(\bar{e}_{m'_0}^*\bb{T})$
via the injection $H^1(\bar{e}_{m'_0, \bb{T}})$, and identify
$\kappa_n^{\mathrm{univ}}
(\bar{e}_{m'_0}^*\boldsymbol{c})_{\varpi^{m'_0}\mca{O}'}$
with $\kappa_n^{\mathrm{univ}}(\boldsymbol{c})_{I'}$.

Let $f \in \Hom_{\Lambda/I'}(
H^1(\pi_{I'}^*\bb{T}),\Lambda/I')$
be any element.
Since $\bar{e}_{m'_0} \circ f$ is 
an $\mca{O}$-linear map, and 
since $F'/F$ is unramified, 
we can define  an $\mca{O}'/\varpi^{m'_0}\mca{O}'$-linear map
\[
f'\colon (\mca{O}'/\varpi^{m'_0}\mca{O}')\kappa_n^{\mathrm{univ}}(\boldsymbol{c})_{I'}
\longrightarrow \mca{O}'/\varpi^{m'_0}\mca{O}';\ 
x\kappa_n^{\mathrm{univ}}(\boldsymbol{c})_{I'} \longmapsto 
x\cdot (\bar{e}_{m'_0} \circ f)(\kappa_n^{\mathrm{univ}}(\boldsymbol{c})_{I'}),
\]
where
$(\mca{O}'/\varpi^{m'_0}\mca{O}')
\kappa_n^{\mathrm{univ}}(\boldsymbol{c})_{I'}$ is 
the $\mca{O}'/\varpi^{m'_0}\mca{O}'$-submodule
of $H^1(\bar{e}_{m'_0}^*\pi_{I'}^*\bb{T})$
generated by $\kappa_n^{\mathrm{univ}}(\boldsymbol{c})_{I'}$.
Since $\mca{O}'/\varpi^{m'_0}\mca{O}'$ is an injective 
$\mca{O}'/\varpi^{m'_0}\mca{O}'$-module,
the map $f'$ can be extended to a homomorphism defined on
$H^1(\bar{e}_{m'_0}^*\pi_{I'}^*\bb{T})$.
So, we have 
\[
\bar{e}_{m'_0} \circ
f(\kappa_n^{\mathrm{univ}}(\boldsymbol{c})_{I'}) 
\in \mf{C}(e^*\boldsymbol{c};\varpi^{m'_0}\mca{O}';n).
\]
Hence we obtain (\ref{eqe*sup}).

Next, let us  prove 
\begin{equation*}
\mf{C}(e^*\boldsymbol{c};\varpi^{m'_0}\mca{O}';n)
\subseteq \bar{e}_{m'_0}(\mf{C}(\boldsymbol{c};I';n))\cdot 
{\mca{O}'/\varpi^{m'_0}\mca{O}'}.
\end{equation*}
Since the ring $\mca{O}'/\varpi^{m'_0}\mca{O}'$ 
is a quotient of a DVR, 
there exists an element 
$\tilde{f}' \in \Hom_{\mca{O}}(
H^1(\bar{e}_{m'_0}^*\pi_{I'}^*\bb{T}),
\mca{O}'/\varpi^{m'_0}\mca{O}')$ such that
$\tilde{f}'(\kappa_n^{\mathrm{univ}}(\boldsymbol{c})_{I'})$
generates the ideal 
$\mf{C}(e^*\boldsymbol{c};\varpi^{m'_0}\mca{O}';n)$.
Since $F'/F$ is unramified, 
we may assume that 
\[
\tilde{f}'(\kappa_n^{\mathrm{univ}}(\boldsymbol{c})_{I'})
\in \mca{O}/\varpi^{m'_0}\mca{O}.
\]
We denote by $f_0$ the restriction of 
$\tilde{f}'$ to 
$(\Lambda/I')\kappa_n^{\mathrm{univ}}(\boldsymbol{c})_{I'}$.
Then, we have 
\[
f_0 \in \Hom_{\Lambda/I'}
\left(
(\Lambda/I')\kappa_n^{\mathrm{univ}}
(\boldsymbol{c})_{I'},\Lambda/I'
\right).
\]
By Lemma \ref{leminjmod}, the $\Lambda/I'$-module $\Lambda/I'$ 
is injective, so we have a homomorphism 
\[
f \in \Hom_{\Lambda/I'}(\pi_{I'}^*\bb{T},\Lambda/I')
\]
which is an extension of $f_0$.
Then, we have 
\begin{align*}
\mf{C}({e}^*\boldsymbol{c};\varpi^{m'_0}\mca{O}';n)
& =  \bar{e}_{m'_0} \circ f(\kappa_n^{\mathrm{univ}}
(\boldsymbol{c})_{I}) \cdot \mca{O}'/\varpi^{m'_0}\mca{O}' \\
& \subseteq \bar{e}_{m'_0}(\mf{C}(\boldsymbol{c};I';n))\cdot 
{\mca{O}'/\varpi^{m'_0}\mca{O}'}.
\end{align*}
Hence we obtain the assertion (i). 
The assertion (ii) follows similarly.

We shall show the assertion (iii).
By the assertion (i), we have 
\[
\mf{C}(\boldsymbol{c};I';n) \subseteq
\bar{e}_{m'_0}^{-1}(\mf{C}(\bar{e}^*\boldsymbol{c};
\varpi^{m'_0}\mca{O}';n)).
\]
Let $y \in \bar{e}_{m'_0}^{-1}(\mf{C}(\bar{e}^*\boldsymbol{c};
\varpi^{m'_0}\mca{O}';n))$ be any element.
Then, we have an element 
\[
\tilde{f}' \in \Hom_{\mca{O}'}(
H^1(\bar{e}_{m'_0}^*
\pi_{I'}^*\bb{T}),
\mca{O}'/\varpi^{m'_0}\mca{O}')
\] 
such that
$\tilde{f}'(\kappa_n^{\mathrm{univ}}(\boldsymbol{c})_{I'})=
\bar{e}_{m'_0}(y)$.
Since the map $e_{\bar{m'}_0}$ is an injection,
we obtain an element 
$f \in \Hom_{\Lambda/I'}(\pi_{I'}^*\bb{T},\Lambda/I')$
satisfying 
$f(\kappa_n^{\mathrm{univ}}(\boldsymbol{c})_{I'})=y$
by similar arguments to those in the previous paragraph.
So, we have $y \in \mf{C}(\boldsymbol{c};I';n) $.
This implies that the assertion (iii) holds.
We also obtain the assertion (iv) by similar manner.
\end{proof}

\begin{cor}\label{coronevarOO'}
For any $i \in \bb{Z}_{\ge 0}$, the following hold.
\begin{enumerate}[{\rm (i)}]
\item $\mf{C}_i({e}^*\boldsymbol{c};\varpi^{m'_0}\mca{O}')=
\bar{e}_{m'_0}(\mf{C}_i(\boldsymbol{c};I'))
\cdot {\mca{O}'/\varpi^{m'_0}\mca{O}'}$.
\item $\mf{C}_i({e}^*\boldsymbol{c};\varpi^{m_0}\mca{O}')=
\bar{e}_{m_0}(\mf{C}_i(\boldsymbol{c};I)) \cdot 
{\mca{O}'/\varpi^{m_0}\mca{O}'}$.
\item $\mf{C}_i(\boldsymbol{c};I')=
\bar{e}_{m'_0}^{-1}(\mf{C}_i({e}^*\boldsymbol{c};
\varpi^{m'_0}\mca{O}'))$.
\item $\mf{C}_i(\boldsymbol{c};I)=
\bar{e}_{m_0}^{-1}(\mf{C}_i({e}^*\boldsymbol{c};
\varpi^{m_0}\mca{O}'))$.
\end{enumerate}
\end{cor}

\begin{proof}
The assertion (i) (resp.\ (ii)) of Corollary \ref{coronevarOO'}
immediately follows from Lemma \ref{lemNm'0NCTI} and 
Lemma \ref{lemonevarOO'} (i) (resp.\  
Lemma \ref{lemNm'0NCTI} and 
Lemma \ref{lemonevarOO'}(ii)).

Let us show the assertion (iii).
By the assertion (i), we have 
\[
\mf{C}_i(\boldsymbol{c};I') \subseteq
\bar{e}_{m'_0}^{-1}(\mf{C}_i({e}^*\boldsymbol{c};
\varpi^{m'_0}\mca{O}')).
\]
Since $\mca{O}'/\varpi^{m'_0}\mca{O}'$ is a quotient 
of a DVR,
there exists an element $n \in \mca{N}(m'_0)$
satisfying 
$\mf{C}_i(\bar{e}^*\boldsymbol{c};\varpi^{m_0}\mca{O}')=
\bar{e}_{m'_0}(\mf{C}(\boldsymbol{c};I';n)){\mca{O}'}$.
So, Lemma \ref{lemonevarOO'} (iii) implies 
\[
\bar{e}_{m'_0}^{-1}(
\mf{C}_i({e}^*\boldsymbol{c};\varpi^{m_0}\mca{O}'))
=\bar{e}_{m'_0}^{-1}(\mf{C}({e}^*\boldsymbol{c};
\varpi^{m'_0}\mca{O}';n))= \mf{C}(\boldsymbol{c};I';n)
\subseteq \mf{C}_i(\boldsymbol{c};I).
\]
Hence we obtain the assertion (iii).
Similarly, the assertion (iv) follows.
\end{proof}

Let us finish the proof of 
Proposition \ref{proponevarcompletered}. 
Recall that we have the commutative diagram (\ref{diagiotam0m1pi12}).
By Corollary \ref{corspecialunivKolyv}, we have 
\[
\pi_2(\mf{C}_i({e}^*\boldsymbol{c};\varpi^{m'_0}\mca{O}'))
=\mf{C}_i({e}^*\boldsymbol{c};\varpi^{m_0}\mca{O}').
\]
Since $\mca{O}'/\varpi^{m'_0}\mca{O}'$ is a quotient 
of a DVR, 
and since we assume that $m'_0 \ge m_0>c$, 
we have 
\begin{eqnarray}\label{eqpi2pi2-1Ci}
\mf{C}_i({e}^*\boldsymbol{c};\varpi^{m'_0}\mca{O}')=
\pi_2^{-1}(\pi_2(\mf{C}_i({e}^*\boldsymbol{c};\varpi^{m'_0}\mca{O}'))).
\end{eqnarray}
So, by Corollary \ref{coronevarOO'} (iii) and (iv), we obtain  
\begin{align*}
\mf{C}_i(\boldsymbol{c};I')_I&=
\pi_1 \left( \bar{e}_{m'_0}^{-1}(
\mf{C}_i({e}^*\boldsymbol{c};\varpi^{m'_0}\mca{O}')) \right) \\
&= \bar{e}_{m_0}^{-1} \left(
\pi_2 \left(
\mf{C}_i({e}^*\boldsymbol{c};\varpi^{m'_0}\mca{O}')
\right) \right) \\
&= \bar{e}_{m_0}^{-1} \left(
\mf{C}_i(\bar{e}^*\boldsymbol{c};\varpi^{m_0}\mca{O}')
\right) \\
&= \mf{C}_i(\boldsymbol{c};I).
\end{align*}
Note that the second equality follows from the equality 
(\ref{eqpi2pi2-1Ci}), the commutativity of (\ref{diagiotam0m1pi12}), 
and the surjectivity of the homomorphism 
$\pi_1\colon \Lambda/I' \longrightarrow \Lambda/I$.
This completes the proof of Proposition \ref{proponevarcompletered}.
\end{proof}

\subsection{Strong compatibility for cyclotomic two variable deformations}
\label{ssrmkoncyclotcase}

Here, we study a special case of $r=2$,
that is, the cyclotomic deformation of one variable deformations.
The goal of this section is Theorem \ref{thmcyclotcompletered},
which is an analogous result to Theorem \ref{thmonevarcompletered}.

In this section, we put $\Lambda_0:=\bb{Z}_p[[x_1]]$.
Let $\bb{T}_0$ be the $\Lambda_0[G_{\bb{Q}}]$-module
studied in the previous subsection,
which is denoted by $\bb{T}$ there.
Recall that we put $\Gamma:=\Gal(\bb{Q}_\infty/\bb{Q})$
and $\Lambda:=\Lambda_0[[\Gamma]]$.
We fix a topological generator $\gamma \in \Gamma$.
Here, we use similar notation to that in 
Definition \ref{dfncycdir}.
In this subsection, we put 
$\bb{T}:=\bb{T}_0^{\mrm{cyc}}$.
Via the isomorphism 
$\mca{O}[[x_1,x_2]] \simeq \Lambda$
defined by $x_2 \longmapsto \gamma-1$, 
and regard $\Lambda$ as the ring of 
two-variable formal power series.
Note that $\bb{T}$ is a $\Lambda$-module 
on which $ \Gal(\bb{Q}_{\Sigma}/\bb{Q}) $ acts via 
$\rho_{\bb{T}_0}\otimes \chi_{\mrm{taut}}$.

Let $h \in \Lambda$ an element of the form
$h = ax_1 +x_2+b$, where $a \in \mca{O}$ and
$b \in \varpi \mca{O}$.
We put $\overline{\Lambda}:=\Lambda/(h)$, 
and let
$\pi:= \pi_{h\Lambda} \colon 
\Lambda \longrightarrow \overline{\Lambda}$ 
be the natural projection. 
Note that $\overline{\Lambda}$ 
is naturally isomorphic to $\Lambda_0$ 
as a $\Lambda_0$-algebra.
Now we can state the main result in \S \ref{ssrmkoncyclotcase}.

\begin{thm}\label{thmcyclotcompletered}
For any $i \in \bb{Z}_{\ge 0}$, we have
\[
\pi(\mf{C}_i(\boldsymbol{c}))=\mf{C}_i(\pi^*\boldsymbol{c}), 
\]
where $\mf{C}_i(\pi^*\boldsymbol{c})$ is the ideal of 
$\overline{\Lambda}$ defined in Definition \ref{dfnC_i}
for the data 
\[
(\Lambda_0, \pi^*\bb{T} , \pi^*\boldsymbol{c}, \boldsymbol{x}_{\le 1}).
\]
\end{thm}

\begin{proof}
For any positive integer $M$, we  define 
a set $\mca{P}_{M}$ of prime numbers by
\[
\mca{P}_{M}:=\{ \ell \notin \Sigma 
\mathrel{\vert} \ell \equiv 1 \mod \varpi^{M}\mca{O}
\}.
\]
Let $m$ be any positive integer.
We put $I:=I(m):= I(\varpi^{m},x_1^{m},h)$ and 
$I':=I'(m):= I(\varpi^{m},x_1^{m},h^{m})$. 
Since $\chi_{\mathrm{taut}}$ factors through the group
$\Gamma \simeq 1+p\bb{Z}_p$, 
the continuity of $\chi_{\mathrm{taut}}$ implies that there is 
a positive integer $N=N(m)$ with $N \ge m$ such that
for any  
$\ell \in \mca{P}_{N}$,
it holds that 
\begin{equation}\label{eqtautfrobtriv}
\chi_{\mathrm{taut}}(\Frob_\ell)\equiv 1 \mod I'.
\end{equation}
We put $I'':=I''(m):= I(\varpi^{N},x_1^{m},h)$.

Let $\ell \in \mca{P}_{N}$ be any element, 
and $J \in \{I, I'\}$.
We denote by $J_0$ the ideal of $\Lambda_0$ 
generated by $\varpi^{m}$ and $x_1^{m}$.
We put $\bb{T}_{0,\ell}:=\bb{T}_0/(\Frob_\ell -1)\bb{T}_0$
and $\bb{T}_{\ell}:=\bb{T}/(\Frob_\ell -1)\bb{T}$.
By the congruence (\ref{eqtautfrobtriv}), we have
\[
\bb{T}_{\ell}\otimes_{\Lambda} \Lambda/J = 
(\bb{T}_{0,\ell}\otimes_{\Lambda_0} \Lambda_0/J_0)
\otimes_{\Lambda_0/J_0} \Lambda/J. 
\]
Since $\Lambda/J$ is faithfully flat over $\Lambda_0/J_0$,
the $\Lambda/J$-module $\bb{T}_{\ell}\otimes_{\Lambda} \Lambda/J$
is free of rank one if and only if the $\Lambda_0/J_0$-module 
$\bb{T}_{0,\ell}\otimes_{\Lambda_0} \Lambda_0/J_0$ is free of rank one.
So we obtain
\[
\mca{N}(\bb{T},I') \cap \mca{P}_{N} 
= \mca{N}(\bb{T},I) \cap \mca{P}_{N}
\supseteq \mca{N}(\bb{T},I'') \cap \mca{P}_{N} .
\]
Hence by Corollary \ref{corascdesc} and 
(\ref{compC_iI'C_iI}) in \S \ref{ssconstC_i},
we have 
\[
\mf{C}_i(\boldsymbol{c};I'')_I \subseteq 
\mf{C}_i(\boldsymbol{c};I')_I \subseteq \mf{C}_i(\boldsymbol{c};I),
\]
and obtain
\begin{align*}
\mf{C}_i(\pi^*\boldsymbol{c})= \varprojlim_{m} \mf{C}_i(\boldsymbol{c};I''(m))_{I(m)}
&\subseteq \pi(\mf{C}_i(\boldsymbol{c}))
=\varprojlim_{m} \mf{C}_i(\boldsymbol{c};I'(m))_{I(m)} \\
&\subseteq \mf{C}_i(\pi^*\boldsymbol{c})= \varprojlim_{m} \mf{C}_i(\boldsymbol{c};I(m))_{I(m)}.
\end{align*}
This completes the proof of 
Theorem \ref{thmcyclotcompletered} (i).
\end{proof}

\section{Proof of main results}\label{secpfmr}

Let $(\bb{T},\boldsymbol{c})$ be as 
in Theorem \ref{thmuncondresults}.
In this section, we prove our main results, 
namely Theorem \ref{thmuncondresults}, 
Theorem \ref{thmcondresults1} and 
Theorem \ref{thmcondresults2}.
In \S \ref{ssasympbeh}, 
we review Mazur--Rubin's observations in 
\cite{MR} \S 5.3 which become 
important ingredients to study 
the one-variable deformations, that is, 
the modules over $\Lambda^{(1)}$.
In \S \ref{ssprf1var}, we prove the results for 
one-variable cases, namely 
Theorem \ref{thmuncondresults} for $\Lambda=\Lambda^{(1)}$
and Theorem \ref{thmcondresults1}.
In \S \ref{ssredtoonevar},
we prove Theorem \ref{thmuncondresults} 
in general setting. 
In \S \ref{sspfthmcond2}, 
we prove Theorem \ref{thmcondresults2}.

\subsection{Asymptotic behavior}\label{ssasympbeh}
In this and the next subsections, 
we set $\Lambda=\Lambda^{(1)}=\mca{O}[[x_1]]$.
Here, let us recall some observations by 
Mazur and Rubin in \cite{MR} \S 5.3
which reduce studies on the pseudo-isomorphism class of 
a $\Lambda$-module to that on 
the asymptotic behaviors of their specializations.

Let $f$ be $\varpi$ or a linear element of  
$\Lambda=\mca{O}[[x_1]]$, 
and put $\mf{p}:=f \Lambda$.
For any $N \in \bb{Z}_{> 0}$, we define
an element $f_N \in \Lambda$ by
\[
f_N:= \begin{cases}
f + \varpi^N & (f \ne \varpi) \\
\varpi + x_1^N & (f =\varpi).
\end{cases}
\]
We also write $f_\infty:=f$.
We put $\mf{p}_N:=f_N \Lambda$ 
for each $N \in \bb{Z}_{>0}$.
Note that $\mf{p}_N$ is a prime ideal of $\Lambda$, 
and the residue ring $\mca{O}_N:=\Lambda/\mf{p}_N$ 
becomes a DVR which is finite flat over $\mca{O}$.
Indeed, by definition, we have 
$\mca{O}_N=\mca{O}[\varpi^{1/N}]$ 
(resp.\ $\mca{O}_N=\mca{O}$)
if $f = \varpi$ (resp.\ $f \ne \varpi$).
Let $\pi_N \colon \Lambda \longrightarrow \mca{O}_N$
be the modulo $\mf{p}_N$ reduction map.

We adopt the following notation.
\begin{dfn}
Let $\{ \alpha_N \}_{N \in \bb{Z}_{ > 0}}$ and 
$\{ \beta_N \}_{N \in \bb{Z}_{N > 0}}$ be sequences of real numbers.
We write $\alpha_N \prec \beta_N$ if and only if we have
$\liminf_{N \to \infty} (\beta_N-\alpha_N)> -\infty$.
Moreover if we have $\alpha_N \prec \beta_N$ and 
$\beta_N \prec \alpha_N$, we write $\alpha_N \sim \beta_N$.
Namely, we write $\alpha_N \sim \beta_N$ 
if and only if the sequence 
$\left\{ \left| \beta_N - \alpha_N \right|  
\right\}_{N \in \bb{Z}_{>0}}$ is bounded.
\end{dfn}

The following elementary lemma becomes a key.

\begin{lem}\label{lemkeyobs}
Let $g \in \Lambda$ be a prime element.
\begin{enumerate}[$(1)$]
\item If $g$ is prime to $f$, 
then there exist a positive integer $N$ such that
for any $n \in \bb{Z}$ with 
$n >N$, the length $\mathrm{length}_{\mca{O}}(\Lambda/(g,f_n))$ of
the  $\mca{O}$-module $\Lambda/(g,f_n)$ is finite.
Moreover, the sequence 
$\{\mathrm{length}_{\mca{O}}(\Lambda/(g,f_n)) \}_{n >N}$
is bounded.
\item Let $e \in \bb{Z}_{>0}$. If $g=f^e$, then 
for any $N$, 
the length of
$\Lambda/(g,f_N)$ as $\mca{O}_N$-module is finite.
Moreover, we have
\[
\mathrm{length}_{\mca{O}_N}(\Lambda/(g,f_N)) \sim eN.
\]
\end{enumerate}
\end{lem}

Let $M$ be a finitely generated $\Lambda$-module, and
fix an integer $i \in \bb{Z}_{\ge 0}$. 
We define an integer $\alpha$ by
\[
\Fitt_{\Lambda_{\mf{p}},i}(M)
=\mf{p}^\alpha\Lambda_{\mf{p}}.
\]
For any $N \in \bb{Z}_{\ge 0}$, 
we define $\bar{\alpha}(N) \in \overline{\bb{N}}$ by
\[
\Fitt_{\Lambda_{\mf{p}},i}(\pi_N^* M)
=\varpi_N^{\bar{\alpha}(N)}\mca{O}_N,
\]
where $\varpi_N$ is a prime element of $\mca{O}_N$.
Then, by Lemma \ref{lemkeyobs} and 
the structure theorem of finitely 
generated torsion $\Lambda$-modules, 
we obtain the following corollary.

\begin{cor}\label{corkeyobs}
There exist a positive integer $N(M,\mf{p})$ such that
$\bar{\alpha}(N)<\infty$ for any $N \in \bb{Z}$ with $N>N(M,\mf{p})$.
Moreover, we have
\[
\bar{\alpha}(N) \sim \alpha \cdot N.
\]
\end{cor}

\subsection{Proof of results on one variable cases}\label{ssprf1var}

Here, we prove Theorem \ref{thmuncondresults} for $\Lambda=\Lambda^{(1)}$
and Theorem \ref{thmcondresults1}.
Let $f \in \Lambda$ be a prime element of $\Lambda$  
and put $\mf{p}:=f\Lambda$.
Fix any $i \in \bb{Z}_{\ge 0}$. 
We define integers $\alpha_i, \beta_i \in \bb{N}_{>0}$ by
\begin{align*}
\Fitt_{\Lambda_{\mf{p}},i}
(X(\bb{T})_{\mf{p}})
&= f^{\alpha_i} \Lambda_{\mf{p}}, \\
\mf{C}_i(\boldsymbol{c})\Lambda_{\mf{p}}
&= f^{\beta_i}\Lambda_{\mf{p}}.
\end{align*}

In order to show Theorem \ref{thmuncondresults}
for the case when $\Lambda=\Lambda^{(1)}$ 
(resp.\ Theorem \ref{thmcondresults1}),
it suffices to show $\alpha \le \beta$
(resp.\ $\alpha=\beta$ under the assumption (MC)).

If $f=\varpi$, then we put $\mca{O}':=\mca{O}$, and $\tilde{f}:=f$.
$\mca{O}$
If $f$ is a distinguished polynomial, 
namely if $f \ne \varpi$
then we  define the ring $\mca{O}'$ and a prime element 
$\tilde{f} \in \mca{O}'[[x_1]]$
as follows.
\begin{itemize}
\item[] Let $F'$ be a finite extension field of $F$
such that in $F'[x_1]$, the polynomial $f(x_1)$
is decomposed into the product of linear factors.
We denote by $\mca{O}'$ the ring of integers of $F'$, 
and fix a root $a \in \mca{O'}$ of $f$.
We put $\tilde{f}:=x_1-a \in \mca{O}'[[x_1]]$.
\end{itemize}
We put $\Lambda':=\mca{O}'[[x_1]]$,
and $\tilde{\mf{p}}:=\tilde{f}\mca{O}'[[x_1]]$. 
Let $e$ be the ramification index of 
the extension $\Lambda'_{\tilde{\mf{p}}}/\Lambda_{\mf{p}}$ of DVRs.
By Proposition \ref{corOO'} and Proposition \ref{lemfittscext}, 
we have 
\begin{align*}
\Fitt_{\Lambda_{\mf{p}},i}
(X(\bb{T}\otimes_{\mca{O}}\mca{O}')_{\tilde{\mf{p}}})
&= \tilde{f}^{e\alpha_i} \Lambda'_{\tilde{\mf{p}}}, \\
\mf{C}_i(\boldsymbol{c}\otimes \mca{O}')
\Lambda'_{\tilde{\mf{p}}}
&= \tilde{f}^{e\beta_i}\Lambda'_{\tilde{\mf{p}}}.
\end{align*}
So, In order to show Theorem \ref{thmuncondresults}
for the case when $\Lambda=\Lambda^{(1)}$ 
and resp.\ Theorem \ref{thmcondresults1}),
we may replace $(\mca{O}, \bb{T},\boldsymbol{c})$
with $(\mca{O}', \bb{T}\otimes_{\mca{O}}\mca{O}', 
\boldsymbol{c}\otimes \mca{O}')$, and 
assume that  $f=f'=x_1-a$.

Let us apply the observation in \S \ref{ssasympbeh}
to our situation.
For any $N \in \overline{\bb{N}}_{>0}$, we put
\[
f_N:= \begin{cases}
f + \varpi^N=x_1- a +\varpi^N & (f \ne \varpi), \\
\varpi + x_1^N & (f =\varpi).
\end{cases}
\]
We set $\mf{p}_N:=f_N\Lambda$, and $\mca{O}_N:=\Lambda/\mf{p}_N$.
Let 
\(
\pi_{N} \colon \Lambda \longrightarrow 
\mca{O}_N
\)
be the modulo $\mf{p}_N$ reduction map.
For each 
$N \in \bb{Z}_{>0}$,
we fix a uniformizer $\varpi_N$
of the DVR $\mca{O}_N$, and 
define $\bar{\alpha}_{i}(N), \bar{\beta}_i(N)
 \in \overline{\bb{N}}$ by
\begin{align*}
\Fitt_{\mca{O},i}
( X(\pi_N^* \bb{T}) )
&= \varpi_N^{\bar{\alpha}_i(N)} \mca{O}_N, \\
\mf{C}_i(\pi_N^* \boldsymbol{c})
&= \varpi_N^{\bar{\beta}_i(N)}\mca{O}_N.
\end{align*}

\begin{proof}[Proof of Theorem \ref{thmuncondresults}
for the case when $\Lambda=\Lambda^{(1)}$]
By Corollary \ref{corcontthm}, 
Theorem \ref{thmonevarcompletered}
and Corollary \ref{corkeyobs}, 
we obtain
\begin{align}
\label{aligonevarstr1}
\bar{\alpha}_i(N)  &\sim \alpha_i \cdot N,\\ 
\label{aligonevarstr2}
\bar{\beta}_i(N)  &\prec \beta_i \cdot N.
\end{align}
Since Theorem \ref{thmMRstrthm}
implies $\bar{\alpha}_i(N) \le \bar{\beta}_i(N)$,
we have
\[
\alpha_i \cdot N \sim \bar{\alpha}_i(N)
\le \bar{\beta}_i(N)
\prec \beta_i \cdot N.
\]
Hence we obtain $\alpha_i \le \beta_i$. 
\end{proof}

For the proof of Theorem \ref{thmcondresults1},
we need a bit more careful arguments.
Assume that (MC) for $(\bb{T} ,\boldsymbol{c})$ holds.
Let us show $\alpha_i=\beta_i$.
By (\ref{aligonevarstr1}) and (\ref{aligonevarstr2}),
it suffices to show 
$\bar{\beta}_i(N) \prec 
\bar{\alpha}_i(N)$.
We need the following lemma
which describes a relation between the ideals
$\mathrm{Ind}(\boldsymbol{c})$ and
$\mf{C}_0(\boldsymbol{c})$.

\begin{lem}\label{LemC0Ind}
We denote by $X_{\mathrm{pn}}$ 
the maximal pseudo-null $\Lambda$-submodule of $X$.
Then, we have
\[
\mathrm{ann}_{\Lambda}(X_{\mathrm{pn}})\mathrm{Ind}(\boldsymbol{c})
\subseteq \mf{C}_0(\boldsymbol{c}).
\]
\end{lem}

\begin{proof}[Proof of Lemma \ref{LemC0Ind}]
Let $\psi \in \Hom_{\Lambda}(H^1_{\Sigma}(\bb{T}),\Lambda)$
be any element, and $m \in \bb{Z}_{>0}$ 
any positive integer.
In order to prove Lemma \ref{LemC0Ind}, 
it suffices to show that for any 
$a \in \mathrm{ann}_{\Lambda}(X_{\mathrm{pm}})$,
we have $\pi_{I(\boldsymbol{x}^m)}(a \psi(c(1))) \in 
\mf{C}_0 (\pi_I(\boldsymbol{x}^m)^*\boldsymbol{c})$.
Note that we have a projective resolution
\[
0 \longrightarrow 
\Lambda \xrightarrow{\ d_2 \ } \Lambda^2 \xrightarrow{\ d_1 \ }
\Lambda \xrightarrow{\ \pi_{I(\boldsymbol{x}^m)}\ } 
\Lambda/I(\boldsymbol{x}^m) \longrightarrow 0
\]
of the $\Lambda$-module $\Lambda/I(\boldsymbol{x}^m)$
where the map $d_1$ is defined by 
$d_1(z_1,z_2)=\varpi^mz+ x_1^m z_2$ for each 
$(z_1,z_2) \in \Lambda^2$, and 
the map $d_2$ is defined by 
$d_2(z)=(-x^mz,\varpi^mz)$ for each $z \in \Lambda$.
So, we have
\[
\mathrm{Tor}_2(H^1_{\Sigma}(\bb{T}), \Lambda/I(\boldsymbol{x}^m))
\simeq \Ker \left(
d_2 \otimes \mathrm{id}_{H^2_{\Sigma}(\bb{T})}
\right)
=H^2_{\Sigma}(\bb{T})[I(\boldsymbol{x}^m)]
\subseteq X_{\mathrm{pn}}.
\]
Let $a \in \mathrm{ann}_{\Lambda}(X_{\mathrm{pn}})$
be any element.
By Corollary \ref{corcontthm}, the element $a$
annihilates the kernel of
the natural map
$P_m \colon
\pi_{I(\boldsymbol{x}^m)}^*H^1_{\Sigma}(\bb{T}) \longrightarrow 
H^1_{\Sigma}(\pi_{I(\boldsymbol{x}^m)}^*\bb{T})$.
This implies that the map 
\[
(a \psi) \otimes \pi_{I(\boldsymbol{x}^m)}
\colon \pi_{I(\boldsymbol{x}^m)}^*H^1_{\Sigma}(\bb{T}) 
\longrightarrow \Lambda/I(\boldsymbol{x}^m))
\]
factors through $\mathrm{Im}(P_m)$.
Since the $\Lambda/I(\boldsymbol{x}^m))$ is an injective
$\Lambda/I(\boldsymbol{x}^m))$-module,
there exists a $\Lambda/I(\boldsymbol{x}^m)$-linear map
$\bar{\psi}_{a,m} \colon 
H^1_{\Sigma}(\pi_{I(\boldsymbol{x}^m)}^*\bb{T})
\longrightarrow 
\Lambda/I(\boldsymbol{x}^m) $
which makes the diagram
\[
\xymatrix{
\pi_{I(\boldsymbol{x}^m)}^*H^1_{\Sigma}(\bb{T}) 
\ar[rr]^{(a \psi) \otimes \pi_{I(\boldsymbol{x}^m)}} 
\ar[d]_{P_m}
& &
\Lambda/I(\boldsymbol{x}^m)) \\
H^1_{\Sigma}(\pi_{I(\boldsymbol{x}^m)}^*\bb{T}
\ar@{-->}[rru]_{\bar{\psi}_{a,m}} &&
}
\]
commute.
Hence we obtain
\[
\pi_{I(\boldsymbol{x}^m)}(a \psi(c(1)))
=\bar{\psi}_{a,m}(\kappa^{\mathrm{univ}}_1
(\boldsymbol{c})_{I(\boldsymbol{x}^m)})
 \in 
\mf{C}_0 (\pi_I(\boldsymbol{x}^m)^*\boldsymbol{c}).
\]
This completes the proof of 
Lemma \ref{LemC0Ind}.
\end{proof}

\begin{proof}[Proof of Theorem \ref{thmcondresults1}]
Fix any $i \in \bb{Z}_{\ge 0}$. Let us show that 
$\bar{\alpha}_i(N)
\sim \bar{\beta}_i(N)$.
By Theorem \ref{thmuncondresults}
for one-variable cases,
it suffices to show that 
$\bar{\beta}_i(N) \prec 
\bar{\alpha}_i(N)$.
By (MC) for $(\bb{T}, \boldsymbol{c})$
and  Lemma \ref{LemC0Ind}, we have
\[
\mathrm{ann}_{\Lambda}(X_{\mathrm{pn}}) 
\Fitt_{\Lambda,0}(X) \subseteq
\mathrm{ann}_{\Lambda}(X_{\mathrm{pn}})  
\cha_{\Lambda}(X) 
= 
\mathrm{ann}_{\Lambda}(X_{\mathrm{pn}})  
\mathrm{Ind}(\boldsymbol{c})
\subseteq 
\mf{C}_0(\boldsymbol{c}).
\]
By Corollary \ref{corcontthm}
and Theorem \ref{thmonevarcompletered}, we obtain
\[
\pi_N\left(
\mathrm{ann}_{\Lambda}(X_{\mathrm{pn}}) 
\right)
\cdot \Fitt_{\mca{O}_N,0}(X(\pi_N^* \bb{T}))
\subseteq  \mf{C}_0(\pi_N^*\boldsymbol{c})
\]
for any $N \in \bb{Z}_{>0}$.
Since the height of the ideal
$\mathrm{ann}_{\Lambda}(X_{\mathrm{pn}}) $
of the ring $\Lambda$ is at least two,
there exists a positive integer $L_0 \in \bb{Z}_{>0}$
such that for any $N \in \bb{Z}_{>0}$, 
we have
$\pi_N(
\mathrm{ann}_{\Lambda}(X_{\mathrm{pn}}) 
)) \supseteq \varpi_N^{L_0}\mca{O}_N$.
So we obtain
\[
\varpi_N^{L_0}
\Fitt_{\mca{O}_N,0}(X(\pi_N^* \bb{T}))
\subseteq \mf{C}_0(\pi_N^*\boldsymbol{c}).
\]
for any $N \in \bb{Z}_{>0}$.
By Corollary \ref{corspecialunivKolyv}, we have
\[
\varpi_N^{L_0}
\Fitt_{\mca{O}_N,i}(X(\pi_N^* \bb{T}))
\subseteq \mf{C}_i(\pi_N^*\boldsymbol{c}).
\]
for any $N \in \bb{Z}_{>0}$.
Hence we obtain 
$\bar{\beta}_i(N) \prec 
\bar{\alpha}_i(N)$.
\end{proof}

\subsection{Proof of Theorem \ref{thmuncondresults}}\label{ssredtoonevar}

In \S \ref{ssprf1var}, we have already proved the assertions of 
Theorem \ref{thmuncondresults} for one-variable cases.
Here, let us complete the proof of Theorem \ref{thmuncondresults}
by induction on the number $r$ of the variables in 
the coefficient ring $\Lambda$.
For each $r \in \bb{Z}_{\ge 1}$, 
we consider the following induction hypothesis $(I)_r$, 
which claims that the assertion of Theorem \ref{thmuncondresults} 
for the $r$-variable cases holds:
\begin{itemize}
\item[$(I)_{r}$] \textit{
Let $F$ be a finite extension field of $\bb{Q}_p$.
We denote $\mca{O}$ the ring of integers of $F$, 
and put $\Lambda^{(r)}: =\mca{O}[[x_1, \dots ,x_r]]$.
Let $\bb{T}'$ be an arbitrary free 
$\Lambda^{(r)}$-module of finite rank equipped 
a continuous $\Lambda^{(r)}$-linear action of $G_{\bb{Q},\Sigma}$
satisfying the conditions {\rm (A1)--(A8)}. 
Let $\boldsymbol{c}'$ be an Euler system for $\bb{T}'$ 
satisfying the condition {\rm (NV)}
which can be extended to cyclotomic direction.
Then, for any height one prime ideal $\mf{p}$ of $\Lambda^{(r)}$ 
and for any $i \in \bb{Z}_{\ge 0}$,
we have 
\[
\Fitt_{\Lambda^{(r)}_{\mf{p}},i}
(X(\bb{T}')_{\mf{p}}) \supseteq 
\mf{C}_i(\boldsymbol{c}')\Lambda_{\mf{p}}. 
\]}
\end{itemize}
Note that in the assertion of $(I)_r$, 
we vary the field $F$.

Here, we fix any finite extension field $F$ of $\bb{Q}_p$.
Let $r \in \bb{Z}_{\ge 2}$, and put 
$\Lambda=\Lambda^{(r)}$.
In order to prove
Theorem \ref{thmuncondresults},
it suffices to show the following assertion:
\begin{prop}\label{propind}
Suppose that the hypothesis $(I)_{r-1}$ holds.
Let $(\bb{T},\boldsymbol{c})$ be
a pair over $\Lambda$
satisfying the assumptions in Theorem \ref{thmuncondresults}.
Then, for any height one prime ideal $\mf{p}$ of $\Lambda$ 
and for any $i \in \bb{Z}_{\ge 0}$,
we have 
\[
\Fitt_{\Lambda_{\mf{p}},i}
(X(\bb{T})_{\mf{p}}) \supseteq
\mf{C}_i(\boldsymbol{c})\Lambda_{\mf{p}}. 
\]
\end{prop}

Now let $(\bb{T},\boldsymbol{c})$ 
be the pair over $\Lambda=\Lambda^{(r)}$
satisfying the assumptions in 
Theorem \ref{thmuncondresults}, and
fix a height one prime ideal $\mf{p}$ of $\Lambda$ 
and an integer $i \in \bb{Z}_{\ge 0}$.
By the structure theorem, we have pseudo-isomorphisms
\[
\iota_X \colon X:=X(\bb{T}) \longrightarrow
\bigoplus_{a=1}^s \bigoplus_{b=1}^{t_a}
\Lambda/\mf{p}_a^{e_{ab}(X)}
\]
and 
\[
\iota_C \colon \Lambda/\mf{C}_i(\boldsymbol{c}) \longrightarrow
\bigoplus_{a=1}^s 
\Lambda/\mf{p}_a^{e_{a}(C)},
\]
where $\mf{p}_1, \dots, \mf{p}_s$ are 
distinct height one prime ideals of $\Lambda$, 
and $\{e_{ab}(X)\}_{b=1}^{t_a}$ (resp.\ $e_a(C)$) is 
an increasing sequence of non-negative 
integers (resp.\ a non-negative integer)
for any $a \in \bb{Z}$ with $1 \le a \le s$.
We may (and do) assume $\mf{p}_1=\mf{p}$, 
and $e_{11}(X)>0$. We may also assume that $t_1>i$.
We put 
$\delta:=\sum_{a=1}^{t_1-a} e_{1a}(X)$.
Then, in oder to prove the assertion (1) of 
Proposition \ref{propind}, 
it suffices to show the inequality
\begin{equation}\label{ineqlocpiv2}
\delta \le e_1(C)
\end{equation}
under the hypothesis $(I)_{r-1}$.

Here, we shall prove the inequality (\ref{ineqlocpiv2})
via the reduction arguments 
by using principal ideals of $\Lambda$ generated 
by \textit{a linear element}, 
which are developed in the Ochiai's work \cite{Oc2}.

Recall that in Definition \ref{defnLE}, 
we have introduced the notion
of linear elements in the sense of 
Ochiai's article \cite{Oc2}.
A principal ideal of $\Lambda$ generated by 
a linear element is called 
\textit{a linear ideal} of $\Lambda$.
Ochiai introduced the following sets 
consisting of certain linear ideals.

\begin{dfn}[\cite{Oc2} Definition 3.3]
Recall that we put 
$\Lambda=\Lambda^{(r)}=\mca{O}[[x_1 , \dots, x_r]]$.
\begin{enumerate}[(i)]
\item We denote by $\mca{L}_\mca{O}^{(r)}$ by the set of 
all linear ideals of $\Lambda=\mca{O}[[x_1, \dots , x_r]]$.
\item For any finite set $\mf{I}$ of 
ideals of $\Lambda$, 
we denote by $\mca{L}_{\mca{O}}^{(r)}[\mf{I}]$ 
the set of all linear ideals of $\Lambda$
not contained in any ideal belonging to $\mf{I}$.
\item For any finitely generated torsion $\Lambda$-module $M$, 
we denote by $\mca{L}_{\mca{O}}^{(r)}(M)$
the set of all linear ideals $I$ of $\Lambda$ satisfying 
the following two conditions.
\begin{enumerate}[(a)]
\item The quotient $\pi_I^*M=M/IM$ is a torsion $\Lambda /I$-module, 
where $\pi_I \Lambda \longrightarrow \Lambda/I$ denotes the residue map.
\item It holds that
\(
\cha_{\Lambda /I}( \pi_I^* M)= \pi_I \left(
\cha_{\Lambda} (M)
\right)
\).
\end{enumerate}
\end{enumerate}
\end{dfn}

We put 
$\overline{\Lambda}:=\Lambda^{(r-1)} 
=\mca{O}[[x_1, \dots, x_{r-1}]]$
and $\Lambda:=\Lambda^{(r)}$.
For any linear element
$g= a_0 + \sum_{i=1}^{r} a_i x_i \in \Lambda$
with $a_r \in \mca{O}^\times$,
we have a natural isomorphism
$\overline{\Lambda} \simeq \Lambda/g\Lambda$.

Let $F'$ be a finite extension field of $F$, and 
$\mca{O}'$ the ring of integers.
We fix a uniformizer $\varpi'$ of $\mca{O}'$, and 
denote by $k'$ the residue field of $\mca{O}'$.
We have a bijection 
\[
P_{\mca{O}'}=(P_{\mf{m},\mca{O}'}, P_{\bb{P},\mca{O}'})
\colon \mca{L}^{(r)}_{\mca{O}'} \longrightarrow 
\varpi'\mca{O}' \times \bb{P}^{r-1}(\mca{O}')
\]
defined by 
\[
\left( a_0 + \sum_{i=1}^{r} a_ix^i \right)\Lambda
\longmapsto \left( a_0a_{i_0}^{-1}, 
(a_1: a_2 : \cdots : a_{r})
\right)
\]
where we put $i_0:= \min 
\{ i \in \bb{Z} \mathrel{\vert} 
1\le i \le r,\ a_i \in {\mca{O}'}^\times \}$.
In the article \cite{Oc2}, 
Ochiai defined a map
$\mathrm{Sp}_{\mca{O}'} \colon 
\mca{L}_{\mca{O}'} \longrightarrow 
\bb{P}^{r-1}(k')$ 
to be the composite of the map $P_{\bb{P},\mca{O}'}$
and the reduction map
$\bb{P}^{r-1}(\mca{O}') \longrightarrow 
\bb{P}^{r-1}(k')$.

Let us recall the following two lemmas
proved in \cite{Oc2}.

\begin{lem}[\cite{Oc2} Lemma 3.5]\label{lemOc3.51}
Let $\mf{a}$ be a height two prime ideal of $\Lambda$.
\begin{enumerate}[{\rm (i)}]
\item The set
$\mca{L}^{(r)}_{\mca{O}} \setminus 
\mca{L}^{(r)}_{\mca{O}}[\{ \mf{a} \}]$
is infinite if and only if
the ideal $\mf{a}$ contains at least two elements 
of $\mca{L}^{(r)}_{\mca{O}}$.
Moreover, if the set 
$\mca{L}^{(r)}_{\mca{O}} \setminus 
\mca{L}^{(r)}_{\mca{O}}[\{ \mf{a} \}]$ is infinite, 
there exist two distinct linear element $g_1, g_2 \in \Lambda$
satisfying $\mf{a}=(g_1,g_2)$.
\item Assume that there exist 
two distinct element $g_1, g_2 \in \Lambda$
satisfying $\mf{a}=(g_1,g_2)$.
If $\varpi \in \mf{a}$, 
then there exists an element $\bar{x} \in \bb{P}^{r-1}(k)$
such that the set 
$\mca{L}^{(r)}_{\mca{O}} \setminus 
\mca{L}^{(r)}_{\mca{O}}[\{ \mf{a} \}]$
is contained in $\mathrm{Sp}_{\mca{O}}^{-1}(\bar{x})$.
\item Assume that there exist 
two distinct element $g_1, g_2 \in \Lambda$
satisfying $\mf{a}=(g_1,g_2)$.
If $\varpi \notin \mf{a}$, 
then there exists a section 
$s \colon \bb{P}^{r-1}(\mca{O}) \longrightarrow 
\mca{L}^{(r)}_{\mca{O}}$ of $P_{\bb{P},\mca{O}'}$
such that the set
$\mca{L}^{(r)}_{\mca{O}} \setminus 
\mca{L}^{(r)}_{\mca{O}}[\{ \mf{a} \}]$
is contained in the image of $s$.
\end{enumerate}
\end{lem}

\begin{lem}[\cite{Oc2} Lemma 3.5]\label{lemOc3.52}
Let $M$ be a finitely generated pseudo-null 
$\Lambda^{(r)}$-module, 
and let $\mathrm{Assoc}_{\Lambda^{(r)}}^2(M)$ be 
the set of all height two associated primes of $M$.
Then, We have 
\[
\mca{L}^{(r)}_{\mca{O}'}(M)
= \bigcap_{\mf{a} \in \mathrm{Assoc}_{\Lambda^{(r)}}^2(M)}
\mca{L}^{(r)}_{\mca{O}'} (\Lambda^{(r)}/\mf{a})
= \bigcap_{\mf{a} \in \mathrm{Assoc}_{\Lambda^{(r)}}^2(M)}
\mca{L}^{(r)}_{\mca{O}'} [\{ \mf{a} \} ]. 
\]
\end{lem}

\begin{proof}[Proof of Theorem \ref{thmuncondresults}]
We denote by $\mf{I}_{\mathrm{char}}$ the set of height one primes 
of $\Lambda$ containing the ideal $\cha_{\Lambda}(X)
\cha_{\Lambda}(\Lambda/\mathrm{Ind}_{C}(\boldsymbol{c}))$.
We define a set $\mf{I}$ of prime ideals of $\Lambda$ by
\begin{align*}
\mf{I}:= & \mf{I}_{\mathrm{char}} \cup 
\bigcup_{a=2}^{s} \mathrm{Assoc}^2_{\Lambda}(
\Lambda/(\mf{p}+\mf{p}_a)) \\
& \cup \mathrm{Assoc}^2_{\Lambda}(\Ker \iota_X)
\cup \mathrm{Assoc}^2_{\Lambda}(\Coker \iota_X) \\
& \cup \mathrm{Assoc}^2_{\Lambda}(\Ker \iota_{C})
\cup \mathrm{Assoc}^2_{\Lambda}(\Coker \iota_{C}). 
\end{align*}
By definition, the set $\mf{I}$ is finite.
So Lemma \ref{lemOc3.51} and Lemma \ref{lemOc3.52} imply that 
after replacing $F$ by a suitable finite extension field
of it,
there exists a linear element 
$g= a_0 + \sum_{i=1}^{r} a_i x_i \in \Lambda$
which is not contained in any $\mf{a} \in \mf{I}$.
By Proposition \ref{lemafftrans},
Moreover, Lemma \ref{lemOc3.51} 
implies that we may assume that $a_r \in \mca{O}^\times$.

For any  
pseudo-isomorphism
$\iota \colon M_1 \longrightarrow M_2$
of $\Lambda$-modules, 
we consider the following condition $P_\iota (h)$ for 
a linear element $h \in \Lambda$.
\begin{itemize}
\item[$P_\iota(h)$] \textit{``For 
any height two prime ideal $\mf{a} \in 
\mathrm{Assoc}_{\Lambda}(\Ker \iota) \cup 
\mathrm{Assoc}_{\Lambda}(\Coker \iota)$, 
the ideal $h \Lambda$ is not contained in $\mf{a}$."}
\end{itemize}
By the definition of $\mf{I}$, 
the linear element $g$ satisfies 
the following properties:
\begin{enumerate}[{\rm (LE1)}]
\item \textit{The ideal $g\Lambda$ does not contain the ideal 
$\cha_{\Lambda}(X)\cha_{\Lambda}
(\Lambda/\mathrm{Ind}_{\Lambda}(\boldsymbol{c}))$.}
\item \textit{The linear element $g \in \Lambda$ satisfies 
the conditions $P(\iota_X)$ and $P(\iota_C)$.}
\item \textit{For any $a \in \bb{Z}$ with $2 \le a \le s$, 
the height of the ideal $\mf{p}+\mf{p}_a+g\Lambda$
of $\Lambda$ is three.}
\end{enumerate}
Let $\pi_g \colon \Lambda \longrightarrow \Lambda/g\Lambda 
= \overline{\Lambda}$
be the reduction map. 
It follows from the property (LE1) that   
$\pi_g(\mf{p})$ becomes a non-zero principal ideal 
of the UFD $\overline{\Lambda}$.
We fix a height one prime $\bar{\mf{p}}$ of $\overline{\Lambda}$
containing $\pi_g(\mf{p})$, and define a positive integer $m$
by $\bar{\mf{p}}^m\overline{\Lambda}_{\bar{\mf{p}}}=
\pi_g(\mf{p})\overline{\Lambda}_{\bar{\mf{p}}}$.
Note that by the above property (LE2),
the maps
\[
\bar{\iota}_X= \iota_X \otimes \pi_g 
\colon X(\bb{T}) \otimes_{\Lambda,\pi_g} 
\overline{\Lambda}
\longrightarrow
\bigoplus_{a=1}^s \bigoplus_{b=1}^{s_a}
\overline{\Lambda}/\pi_g(\mf{p}_a)^{e_{ab}(X)}
\]
and 
\[
\bar{\iota}_C =\iota_C\otimes \pi_g
\colon \overline{\Lambda}/
\pi_{g}(\mf{C}_i(\boldsymbol{c})) \longrightarrow
\bigoplus_{a=1}^s 
\overline{\Lambda}/\pi_g(\mf{p}_a)^{e_{a}(C)},
\]
induced by $\pi_g$
are pseudo-isomorphisms of $\overline{\Lambda}$-modules.
By Corollary \ref{corcontthm}, 
we have natural isomorphism
$X(\bb{T}) \otimes_{\Lambda,\pi_g}\overline{\Lambda} 
\simeq X(\pi_g^*\bb{T})$.
Since $g$ satisfies (LE3), and since $\bar{\iota}_X$ is a pseudo-isomorphism, 
we obtain
\[
\Fitt_{\overline{\Lambda}_{\bar{\mf{p}}},i}(
X(\pi_g^*(\bb{T}))_{\bar{\mf{p}}})
=\bar{\mf{p}}^{m\delta}\overline{\Lambda}_{\bar{\mf{p}}}
\]
By Proposition \ref{propredonevar}, we have
$\pi_g(\mf{C}_i(\boldsymbol{c})) \subseteq
\mf{C}_i(\pi_g^*\boldsymbol{c})$.
So, the property (LE3) and 
the pseudo-isomorphism $\bar{\iota}_C$ imply that
\[
\pi_g(\mf{C}_i(\boldsymbol{c}))
\overline{\Lambda}_{\bar{\mf{p}}}
= \bar{\mf{p}}^{me_1(C)} 
\overline{\Lambda}_{\bar{p}}
\subseteq \mf{C}_i(\pi_g^*\boldsymbol{c})
\overline{\Lambda}_{\bar{p}}.
\]
Hence the induction hypothesis $(I_r)$ implies that we have
\[
\bar{\mf{p}}^{m\delta}\overline{\Lambda}_{\bar{\mf{p}}}
=
\Fitt_{\overline{\Lambda}_{\bar{\mf{p}}},i}(
X(\pi_g^*(\bb{T}))_{\bar{\mf{p}}})
\supseteq \mf{C}_i(\pi_g^*\boldsymbol{c})
\supseteq
\bar{\mf{p}}^{me_1(C)} 
\overline{\Lambda}_{\bar{p}},
\]
and we obtain the inequality (\ref{ineqlocpiv2}), namely $\delta \le e_1(C)$.
\end{proof}

\subsection{Proof of Theorem \ref{thmcondresults2}}\label{sspfthmcond2}

Here, let us prove Theorem \ref{thmcondresults2}.

First, we recall some notation. 
We denote the cyclotomic $\bb{Z}_p$-extension 
by $\bb{Q}_\infty/\bb{Q}$, and 
put $\Gamma:=\Gal(\bb{Q}_\infty/\bb{Q})$.
We fix a topological generator $\gamma \in \Gamma$.
We put $\Lambda_0:=\Lambda^{(1)}=\mca{O}[[x_1]]$ and 
identify the completed group ring $\Lambda_0[[\Gamma]]$
with $\Lambda:=\Lambda^{(2)}=\mca{O}[[x_1,x_2]]$
via the isomorphism $\Lambda \simeq \Lambda[[\Gamma]]$
of $\Lambda_0$-algebras defined by $1+x_2 \longmapsto \gamma$.
Let $\bb{T}_0$ be a free $\Lambda_0$-module
of finite rank $d$ with 
a continuous $\Lambda_0$-linear action of 
$G_{\bb{Q},\Sigma}$
satisfying the conditions (A1)--(A8). 
We denote by $\bb{T}$ the cyclotomic deformation of $\bb{T}_0$.
Then, the $\Lambda$-module $\bb{T}$ also satisfies 
the conditions conditions (A1)--(A8). 
Let $\boldsymbol{c}$ be an Euler system for 
the $\Lambda$-module $\bb{T}$ satisfying 
the condition (MC).
Note that since $\bb{T}$ is a cyclotomic deformation,
the Euler system $\boldsymbol{c}$ can be extended to
the cyclotomic direction. (See Lemma \ref{lempropcycext}.)

\begin{proof}[{\bf Proof of Theorem \ref{thmcondresults2}}]
In order to prove Theorem \ref{thmcondresults2},
it suffices to show that 
for any $i \in \bb{Z}_{\ge 0}$ and 
any height one prime $\mf{p}$ of $\Lambda$,
we have
\[
\Fitt_{\Lambda_{\mf{p}},i}(X(\bb{T}))
=\mf{C}_i(\boldsymbol{c}).
\]

As in \S \ref{ssredtoonevar}, 
we take pseudo-isomorphisms
\[
\iota_X \colon X:=X(\bb{T}) \longrightarrow
\bigoplus_{a=1}^s \bigoplus_{b=1}^{s_a}
\Lambda/\mf{p}_a^{e_{ab}(X)}
\]
and 
\[
\iota_C \colon \Lambda/\mf{C}_i(\boldsymbol{c}) \longrightarrow
\bigoplus_{a=1}^s 
\Lambda/\mf{p}_a^{e_{a}(C)}.
\]
We may assume $\mf{p}=\mf{p}_1$.
Put $\delta:=\sum_{a=1}^{t_1-a} e_{1a}(X)$.
In order to show Theorem \ref{thmcondresults2},
it suffices to show the equality
\begin{equation}\label{eqd=e}
\delta =e_1(C).
\end{equation}

Let $\mf{I}$ be the finite set of prime
ideals of $\Lambda$ defined in \S \ref{ssredtoonevar}.
Let $\mathrm{pInd}_{\Lambda}(\boldsymbol{c})$ be 
the minimal prime ideal of $\Lambda$ 
containing $\mathrm{Ind}_{\Lambda}(\boldsymbol{c})$,
and $\mf{A}_{\mathrm{Ind}}$ the ideal of $\Lambda$
satisfying 
\[
\mathrm{Ind}_{\Lambda}(\boldsymbol{c})
=\mf{A}_{\mathrm{Ind}} \cdot 
\mathrm{pInd}_{\Lambda}(\boldsymbol{c}).
\]
Note that since
$\Lambda$ is a UFD, the ideals 
$\mathrm{pInd}_{\Lambda}(\boldsymbol{c})$
and $\mf{A}_{\mathrm{Ind}}$ do exist.
We denote by $\mf{I}_{\mathrm{Ind}}^1$ 
the set of height one prime ideals of $\Lambda$
containing $\mathrm{pInd}_{\Lambda}(\boldsymbol{c})$, 
and by $\mf{I}_{\mathrm{Ind}}^2$
the set of height two prime ideals of $\Lambda$
containing $\mf{A}_{\mathrm{Ind}}$.
Let $X_{\mathrm{pn}}$ be the maximal pseudo-null
$\Lambda$-submodule of $X(\bb{T})$.
We define a finite set $\widetilde{\mf{I}}$ of 
prime ideals of $\Lambda$ by
\begin{align*}
\widetilde{\mf{I}}:=  \mf{I}_{\mathrm{Ind}}^1
\cup \mf{I}_{\mathrm{Ind}}^2 \cup 
\mathrm{Assoc}^2_{\Lambda}(X_{\mathrm{pn}}) 
\cup \mf{I}. 
\end{align*}
By Lemma \ref{lemOc3.51} and Lemma \ref{lemOc3.52}, 
after replacing $F$ by a suitable finite extension field
of it,
we can take a linear element 
$g= a_0 +  a_1 x_1+ x_2 \in \Lambda$
which is not contained in any $\mf{a} \in \mf{I}$.
Since $\widetilde{\mf{I}}$ contains $\mf{I}$,
the linear element $g$ satisfies 
the properties (LE1)--(LE3) 
for our $(\bb{T},\boldsymbol{c})$.
Let us show that the linear element $g$ satisfies 
the following additional property:
\begin{itemize}
\item[(LE4)] \textit{Let 
$\pi_g \colon \Lambda \longrightarrow \Lambda/(g)
\simeq \Lambda_0$ be the projection.
Then, we have 
\[
\mathrm{pInd}_{\Lambda_0}(\pi_g^*\boldsymbol{c})
=\pi_g(\mathrm{pInd}_{\Lambda}(\boldsymbol{c})).
\]
}
\end{itemize}
First, we shall prove
\begin{equation}\label{pindsupset}
\pi_g(\mathrm{pInd}_{\Lambda}(\boldsymbol{c}))
\supseteq \mathrm{pInd}_{\Lambda}(\pi_g^*\boldsymbol{c}).
\end{equation}
Since the pair $(\bb{T},\boldsymbol{c})$ satisfies (MC),
we have
\[
\pi_g(\mathrm{pInd}_{\Lambda}(\boldsymbol{c}))
=\pi_g(\cha_{\Lambda} (X(\bb{T}))).
\]
By the properties (LE1) and (LE2), 
the $\Lambda_0$-module $X(\pi_g^*\bb{T})$ is torsion, 
and we have 
\[
\pi_g(\cha_{\Lambda} (X(\bb{T})))=
\cha_{\Lambda_0}(X(\pi_g^*\bb{T}) ).
\]
By Ochiai's results on Euler systems for Galois deformations
(see \cite{Oc2} Theorem 2.4), 
it holds that
\[
\cha_{\Lambda_0}(X(\pi_g^*\bb{T}) )
\supseteq \mathrm{pInd}_{\Lambda}(\pi_g^*\boldsymbol{c}).
\]
Hence we obtain (\ref{pindsupset}).
Next, let us show
\begin{equation}\label{pindsubset}
\pi_g(\mathrm{pInd}_{\Lambda}(\boldsymbol{c}))
\subseteq \mathrm{pInd}_{\Lambda}(\pi_g^*\boldsymbol{c}).
\end{equation}
Let $\psi \in \Hom_{\Lambda}(\bb{H}^1_{\Sigma}(\bb{T}),\Lambda)$
be any homomorphism.
We denote by 
\[
\bar{\psi}\colon \pi_g^* H^1_\Sigma(\bb{T})
\longrightarrow \Lambda/(g) \simeq \Lambda_0
\]
the homomorphism induced by $\psi$.
By the cohomological exact sequence induced by
\[
0 \longrightarrow \bb{T} \xrightarrow{\ \times g \ }
\bb{T} \longrightarrow \bb{T}/g\bb{T} 
\longrightarrow 0,
\] 
we deduce that the natural map
\(
\pi_g^*H^1_{\Sigma}(\bb{T})
\longrightarrow H^1(\pi_g^*\bb{T})
\)
is injective, and its cokernel is annihilated 
by $\ann_{\Lambda_0}(X(\bb{T})[g])$.
So, for any $\bar{h} \in \ann_{\Lambda_0}(X(\bb{T})[g])$,
there exists a homomorphism $\phi_{\bar{h}}
\in \Hom_{\Lambda_0}(\bb{H}^1_{\Sigma}(\pi_g^*\bb{T}),\Lambda_0)$
which makes the diagram
\[
\xymatrix{
\pi_g^*H^1_{\Sigma}(\bb{T})
\ar[r]^(0.6){\bar{h} \cdot \psi} \ar[d] & \Lambda_0 \\
H^1_\Sigma(\pi_g^*\bb{T}) \ar@{-->}[ur]_{\phi_{\bar{h}}} &
}
\]
commutes.
This implies that we have 
\[
\ann_{\Lambda_0}(X(\bb{T})[g])\cdot
\pi_g(\mathrm{Ind}_{\Lambda}(\boldsymbol{c}))
\subseteq \mathrm{Ind}_{\Lambda}(\pi_g^*\boldsymbol{c}).
\]
In order to show (\ref{pindsubset}), 
it suffices to prove that the height of the ideal 
$\ann_{\Lambda_0}(X(\bb{T})[g])$ is at least two.
Since $g$ is prime to $\cha_{\Lambda}(X(\bb{T}))$, 
the $\Lambda$-module $X(\bb{T})[g]$ is a pseudo-null.
Namely, we have $X(\bb{T})[g] \subseteq X_{\mathrm{pn}}$.
Moreover, since $g \not\in \mf{a}$ for any   
$\mf{a} \in \mathrm{Assoc}^2_{\Lambda}(X_{\mathrm{pn}})$,
we deduce that $X(\bb{T})[g]$ is also pseudo-null 
as $\Lambda_0$-module.
So the height of $\ann_{\Lambda_0}(X(\bb{T})[g])$
is at least two, and we obtain (\ref{pindsubset}).
Hence we deduce that
the linear element $g$ satisfies the property (LE4).

Let $\bar{\mf{p}}$ be a height one prime ideal of $\Lambda_0$
containing $\pi_g(\mf{p})$,
and $m$ an integer satisfying 
$\bar{\mf{p}}^m=
\pi_g(\mf{p})\Lambda_{0,\bar{\mf{p}}}$.
It The property (LE3) and Corollary \ref{corcontthm} 
imply that we have
\[
\Fitt_{\Lambda_{0,\bar{\mf{p}}},i}(X(\pi_g^*\bb{T}))
=\Fitt_{\Lambda_{0,\bar{\mf{p}}},i}(\pi_g^*X(\bb{T}))
=\bar{\mf{p}}^{m \delta}\Lambda_{0,\bar{\mf{p}}},
\]
and Theorem \ref{thmcyclotcompletered} implies 
\[
\pi_g(\mf{C}_i(\boldsymbol{c}))\Lambda_{0,\bar{\mf{p}}}=
\mf{C}_i(\pi_g^*\boldsymbol{c})\Lambda_{0,\bar{\mf{p}}}=
\bar{\mf{p}}^{m e_1(C)}\Lambda_{0,\bar{\mf{p}}}.
\]
By the assumption (MC) for the pair 
$(\bb{T},\boldsymbol{c})$ and (LE4),
the pair $(\pi_g^*\bb{T}, \pi_g^*\boldsymbol{c})$ 
satisfies (MC).
So, we can apply Theorem \ref{thmcondresults1} for
the pair
$(\pi_g^*\bb{T}, \pi_g^*\boldsymbol{c})$, and we obtain
\[
\bar{\mf{p}}^{m \delta}\Lambda_{0,\bar{\mf{p}}}
=\Fitt_{\Lambda_{0,\bar{\mf{p}}},i}(
X(\pi_g^*(\bb{T}))_{\bar{\mf{p}}})
= \mf{C}_i(\pi_g^*\boldsymbol{c}) \Lambda_{0,\bar{\mf{p}}}
=\bar{\mf{p}}^{m e_1(C)}\Lambda_{0,\bar{\mf{p}}}.
\]
Hence we obtain the equality (\ref{eqd=e}), that is, 
$\delta=e_1(C)$. 
This completes the proof of Theorem \ref{thmcondresults2}.
\end{proof}

\section{Application to (nearly) ordinary Hida deformations}\label{ssHida}

Here, we apply our main results
to ordinary and nearly ordinary
Hida deformations.

First, Let us fix our notation.
Here, suppose that $p \ge 5$, and fix 
an isomorphism $\overline{\bb{Q}}_p \simeq \bb{C}$.
As in \S \ref{ssintro}, 
let $F$ be a finite extension field, 
and $\mca{O}:=\mca{O}_F$ the ring of integers of $F$.
Fix a positive integer $N$ prime to $p$. 
We put $\Sigma:=\mathrm{Prime}(pN)$, 
namely the set of all prime divisors of $pN$.
Let $D_\infty$ be a pro-finite group equipped 
with an isomorphism
\(
1+p\bb{Z}_p \xrightarrow{\ \simeq \ }D_\infty;\ 
a \longrightarrow \langle a \rangle 
\).
Note that $D_\infty$ is regarded as 
the projective limit of 
the direct factor 
\(
D_m \simeq (\bb{Z}/p^{m+1}\bb{Z})^\times
\otimes_{\bb{Z}} \bb{Z}_p
\) 
of the group naturally
isomorphic to $(\bb{Z}/Np^{m+1} \bb{Z})^\times$
consisting of the 
diamond operators
acting on the modular curve  
$Y_1(N p^{m+1})$.
We denote by 
$\chi_D \colon D_\infty \longrightarrow 1+ p\bb{Z}_p$
the inverse of $\langle \cdot \rangle$.
In \S \ref{ssHida}, 
we set 
$\Lambda:=\mca{O}[[D_\infty]]
\simeq \Lambda^{(1)}$.
Let $\mf{m}_{\Lambda}$ be the maximal ideal of $\Lambda$. 
Then, we have 
$k:= \mca{O}/\varpi \mca{O} \simeq 
\Lambda/\mf{m}_{\Lambda}$.

Let
$\omega \colon G_{\bb{Q},\Sigma}
\longrightarrow \mu_{p-1} \subseteq \bb{Z}_p^\times$
be the Teichm\"uller character.
We regard $\omega$ as a Dirichlet character modulo $pN$
defined by $\omega(\ell):=\omega (\Frob_\ell^{-1} )$
for each prime number $\ell$ not dividing $p$. 
Let 
$\psi=\psi_0\omega^j$ be an $\mca{O}^\times$-valued 
Dirichlet character modulo $pN$, 
where $\psi_0$ denotes a character modulo $N$, 
and $j$ is an integer satisfying $0 \le j \le p-2$.
We take a Hida family
\[
\mca{F}=\sum_{n=1}^\infty A(n,\mca{F}) 
q^n \in \Lambda [[q]]
\]
of cuspforms of tame level $N$
with Dirichlet character $\psi$.
For any $k \in \bb{Z}_{\ge 2}$ and 
any character 
$\varepsilon \colon D_\infty  \longrightarrow
\overline{\bb{Q}}_p^\times$ of finite order, 
the power series 
\[
\mca{F}_{\chi_D^{k-2}\varepsilon} := 
\sum_{n=1}^\infty \chi_D^{k-2}\varepsilon(A(n,\mca{F})) 
q^n \in \overline{\bb{Q}}_p [[q]] 
\]
becomes the $q$-expansion of 
a $p$-ordinary Hecke eigen cuspform of weight $k$
at the cusp $\infty$.
We say that $\mca{F}_{\chi_D^{k-2}\varepsilon}$
is the specialization of $\mca{F}$ at 
the arithmetic point $\chi_D^{k-2}\varepsilon$.
Note that here, we assume that the coefficients
$A(n, \mca{F})$ are contained in $\Lambda$.

\begin{rem}
Let $k$ be an integer with $k \ge 2$, 
and 
$f \in S_k(\Gamma_1(Np),\psi\omega^{2-k}; \overline{\bb{Q}}_p)$
be a $p$-stabilized newform  of tame level $N$.
Then, we have a Hida family
\[
\mca{F}'=
\sum_{n=1}^\infty A(n,\mca{F}') 
q^n 
\in \bb{I}[[q]]
\]
of tame level $N$
with Dirichlet character $\psi$
such that $f$ becomes a specialization of $\mca{F}'$,
where $\bb{I}$ is a  domain finite flat 
over $\bb{Z}_p[[D_\infty]]$.
(See \cite{Hi1} Corollary 3.2 and Corollary 3.7.)
A sufficient condition on $(p,f')$
to ensure that $\mca{F}' \in \mca{O}'[[D_\infty]][[q]]$
for some DVR $\mca{O}'$ finite flat over $\bb{Z}_p$
is studied in \cite{Go}. 
See \cite{Go} Proposition 8 and Corollary 9.
\end{rem}

By \cite{Hi2} Theorem 2.1, 
we have a free $\Lambda$-module $\bb{T}(\mca{F})$  
of rank $2$ with a 
continuous $\Lambda$-linear action $\rho_{\bb{T}(\mca{F})}$ 
of $G_{\bb{Q}, \Sigma}$
satisfying 
\[
\det (1-x \Frob_\ell^{-1} \vert \bb{V}(\mca{F}))=
1-A(\ell,\mca{F}) x +\ell \cdot \psi( \ell )
\langle \mathrm{pr}(\ell) \rangle
\Lambda [x]
\]
for any prime number $\ell$ not dividing $pN$, 
where $\Frob_\ell$ denotes the arithmetic Frobenius at $\ell$, 
and $\mathrm{pr} \colon \bb{Z}_p^\times 
=\mu_{p-1} \times 1+p\bb{Z}_p
\longrightarrow 1+p\bb{Z}_p$
denotes the projection.
(In this article, for the description of $\bb{T}(\mca{F})$, 
we use the cohomological convention.) 
We have an exact sequence  
\[
0 \longrightarrow F^+ \bb{T}(\mca{F})
\longrightarrow 
\bb{T}(\mca{F})
\longrightarrow 
F^- \bb{T}(\mca{F})
\longrightarrow 0
\]
of $\Lambda [G_{\bb{Q}_p}]$-modules, 
where $F^+ \bb{T}(\mca{F})$ 
and $F^- \bb{T}(\mca{F})$
are free $\Lambda$-modules of 
rank one, and 
the $G_{\bb{Q}_p}$-action on 
$F^+ \bb{T}(\mca{F})$ is unramified.
We put  
\[
\bb{T}'_{\mca{F}}:=\Hom_{\Lambda}(\bb{T}(\mca{F}),\Lambda)
\otimes_{\bb{Z}_p} \varprojlim_{m} \mu_{p^{m}}, 
\]
and denote by $\rho_{\bb{T}'_{\mca{F}}} 
\colon G_{\bb{Q},\Sigma}
\longrightarrow \mathrm{Aut}_{\Lambda}(\bb{T}'_{\mca{F}}) 
\simeq \mathrm{GL}_2(\Lambda)$ 
the Galois action on $\bb{T}'_{\mca{F}}$.
Note that $\bb{T}'_{\mca{F}}$
satisfies the assumption (A8).
Moreover, since we assume that $p \ge 5$,
the assumption (A4) for $\bb{T}'_{\mca{F}}$ is satisfied.
Here, we assume the following hypothesis: 
\begin{itemize}
\item[(Full)] The image of  $\rho_{\bb{T}(\mca{F})}$
contains the special linear subgroup 
$\mathrm{SL}_{\Lambda}(\bb{T}(\mca{F}))
\simeq \mathrm{SL}_2(\Lambda)$.
\end{itemize}
Since the commutator subgroup 
$\mathrm{Aut}_{\Lambda}(\bb{T}(\mca{F}))$
is $\mathrm{SL}_{\Lambda}(\bb{T}(\mca{F}))$,
we have 
\[
\rho_{\bb{T}(\mca{F})}( 
\Gal(\bb{Q}_{\Sigma}/\bb{Q}(\mu_p^{\infty})) ) \supseteq 
\mathrm{SL}_2(\Lambda)
\]
if $\bb{T}(\mca{F})$ satisfies (Full).
So, under the assumption (Full),
The conditions (A1)--(A3) for $\bb{T}'_{\mca{F}}$ 
are satisfied obviously.
Moreover, by Proposition \ref{propH1vanish} below,
we can deduce  that
$\bb{T}'_{\mca{F}}$ also satisfies
the condition (A5) 
under the assumption (Full).
(Note that Proposition \ref{propH1vanish}
easily follows from the induction on $i$.)

\begin{prop}\label{propH1vanish}
For any $i \in \bb{Z}_{>0} $, 
we have 
$H^1(\mathrm{SL}_2(\Lambda/\mf{m}_{\Lambda}^i), k ^2)=0$, 
where $k ^2$ 
is regarded as 
a $\Lambda[\mathrm{SL}_2(\Lambda/\mf{m}_{\Lambda}^i)]$-module
equipped with 
the standard matrix action of 
$\mathrm{SL}_2(\Lambda/\mf{m}_{\Lambda}^i)$.
\end{prop}

Let $\widetilde{\bb{T}}'_{\mca{F}}
:=\bb{T}'_{\mca{F}}\otimes_{\Lambda} \Lambda[[ \Gamma ]]$ 
be the cyclotomic deformation of $\bb{T}'_{\mca{F}}$.
We fix a basis $d$ of 
the free $\Lambda$-module 
$\mca{D}:=
(F^+ \bb{T}(\mca{F}) \hat{\otimes}_{\bb{Z}_p} 
W(\overline{\bb{F}}_p))^{G_{\bb{Q}_p}}$
of rank one, and a basis
$b$ of the free $\Lambda$-module 
$\mca{B}^{(-1)^j}$ of rank one 
consisting of $\Lambda$-adic modular symbols
in the sense of \cite{Ki}.
Then, we  have an Euler system 
$\mca{Z}^{\mathrm{Ki}}_{b,d}
= \{ \mca{Z}^{\mathrm{Ki}}_{b,d}(n) \}_{n \in \mca{N}_{\Sigma}}$
for $\widetilde{\bb{T}}'_{\mca{F}}$ 
such that 
$\mca{Z}^{\mathrm{Ki}}_{b,d}(1)$ maps to the two-variable
$p$-adic $L$-function $L^{\mathrm{Ki}}_{p,b}$ 
attached to $\mca{F}$
constructed by Kitagawa in \cite{Ki} 
via the generalized Coleman map 
\(
\Xi_d \colon 
H^1(\bb{Q}_p, \widetilde{\bb{T}}'_{\mca{F}})/
H^1_f(\bb{Q}_p, \widetilde{\bb{T}}'_{\mca{F}})
\longrightarrow 
\Lambda[[\Gamma]]
\)
constructed by Ochiai in \cite{Oc1} Theorem 3.13.
(For details on $\mca{Z}$, see \cite{Oc3} Theorem 6.11.)
Note that $\mca{Z}$ is constructed by gluing 
the Euler system of Beilinson--Kato elements
defined in \cite{Ka}.)
Let $\boldsymbol{c}_{\mca{F}}$ be the Euler system 
for $\widetilde{\bb{T}}'_{\mca{F}}$ 
corresponding to $\mca{Z}^{\mathrm{Ki}}_{b,d}$
in the sense of Proposition \ref{propESmodify}.
The Euler system $\boldsymbol{c}_{\mca{F}}$ 
satisfies the condition (NV).
Moreover,  
the pair 
$(\widetilde{\bb{T}}'_{\mca{F}},
\boldsymbol{c}_{\mca{F}})$ satisfies the condition (MC)
if and only if 
the (two-variable) Iwasawa--Greenberg main conjecture for 
nearly ordinary Hida deformation $\widetilde{\bb{T}}(\mca{F})$
proposed by Greenberg \cite{Gr} holds.
(For the precise statement of the two-variable 
Iwasawa--Greenberg main conjecture for 
$\widetilde{\bb{T}}(\mca{F})$, 
see also \cite{Oc3} Conjecture 2.4.)
In our setting, the two-variable 
Iwasawa--Greenberg main conjecture for 
$\widetilde{\bb{T}}(\mca{F})$
also implies 
$(\bb{T}'_{\mca{F}},
\mathrm{aug}^*\boldsymbol{c}_{\mca{F}})$ 
satisfies the condition (MC).
(See \cite{Oc3} Corollary 7.5.)

\begin{thm}
Suppose that $\bb{T}'_{\mca{F}}$
satisfies the hypotheses {\rm (A6), (A7)} and {\rm (Full)}.
We denote by $\mathrm{aug}_{\Gamma}\colon 
\Lambda[[\Gamma]] \longrightarrow \Lambda$
the augmentation map.
Take any $i \in \bb{Z}_{\ge 0}$.
Let $\widetilde{\mf{p}}$
be a height one prime ideal
of $\Lambda[[\Gamma]]$, and 
$\mf{p}$ a height one prime ideal of $\Lambda$. 
Then, we have 
\(
\Fitt_{\Lambda[[\Gamma]]_{\widetilde{\mf{p}}},i}
(X(\widetilde{\bb{T}}'_{\mca{F}}))
\supseteq \mf{C}_i(\boldsymbol{c}_{\mca{F}})
\Lambda[[\Gamma]]_{\widetilde{\mf{p}}}
\), 
and 
\(
\Fitt_{\Lambda_{\mf{p}},i}
(X(\bb{T}'_{\mca{F}}))
\supseteq \mf{C}_i(\mathrm{aug}^*
\boldsymbol{c}_{\mca{F}})
\Lambda_{\mf{p}}
\).
Moreover, if the two-variable Iwasawa--Greenberg main conjecture for 
nearly ordinary Hida deformation $\widetilde{\bb{T}}(\mca{F})$ holds, 
we have the equalities
\(
\Fitt_{\Lambda[[\Gamma]]_{\widetilde{\mf{p}}},i}
(X(\widetilde{\bb{T}}'_{\mca{F}}))
= \mf{C}_i(\boldsymbol{c}_{\mca{F}})
\Lambda[[\Gamma]]_{\widetilde{\mf{p}}}
\), and 
\(
\Fitt_{\Lambda_{\mf{p}},i}
(X(\bb{T}'_{\mca{F}}))
= \mf{C}_i(\mathrm{aug}^*
\boldsymbol{c}_{\mca{F}})
\Lambda_{\mf{p}}
\).
\end{thm}

\begin{rem}
Some sufficient conditions
for (Full) are studied, for instance, 
in  \cite{Bo} and \cite{MW}.
In our setting, the condition (Full) 
for $\bb{T}(\mca{F})$ is satisfied 
if $\mca{O}=\bb{Z}_p$, and 
if the residual representation 
$\bb{T}(\mca{F}) \otimes_{\Lambda} k$
contains is $\mathrm{SL}_2(k)$.
(For details, see \cite{MW} \S 10 Proposition.)
\end{rem}

\begin{rem}\label{remA67}
Here, we give some remarks on the conditions
(A6) and (A7) on $\bb{T}'_{\mca{F}}$.
Let $k \in \bb{Z}_{\ge 2}$, and 
suppose that the conductor of 
$\psi_0:=\psi\omega^{2-k}$ divides $N$.
Let $f = \sum_{n=1} a(n,f)q^n 
\in S_k(\Gamma_1(N), \varepsilon ; \overline{\bb{Q}}_p)$ 
be a $p$-ordinary 
normalized Hecke eigen newform with $r \in \bb{Z}_{\ge 0}$.
We denote by $f^*$ the $p$-stabilization ordinary newform 
of tame level $N$ attached to $f$. 
Suppose that $f^*$ is a specialization
of $\mca{F}$ at an arithmetic point $\eta$. 
We denote by $\mca{O}_\eta$ 
the normalization of $\eta(\Lambda)$. 
Note that the residue field of $\mca{O}_\eta$ 
is naturally isomorphic to $k=\Lambda/\mf{m}_{\Lambda}$.
We fix an $\mca{O}_\eta$-lattice $T(f)$ of 
the $p$-adic Galois representation $V(f)$
attached to $f$.
\begin{enumerate}[(i)]
\item Obviously, the condition (A6) on $\bb{T}(\mca{F})$
holds if and only if the following  property
$(\mathrm{A6})_{f,\ell}$
holds for for any prime divisor $\ell$ of $N$. 
\begin{itemize} \vspace{3mm}
\setlength{\leftskip}{0.8cm}
\item[$(\mathrm{A6})_{f,\ell}$] We have 
$H^0(I_\ell, T(f) \otimes_{\mca{O}}k)=0$. 
\end{itemize} \vspace{3mm}
A sufficient condition for $(\mathrm{A6})_{f,\ell}$.
is given in Proposition \ref{propsuffcond} below.
\item By Deligne's unpublished work, 
the continuous modulo $p$ Galois representation 
$(T(f) \otimes_{\mca{O}_\eta} k, 
\bar{\rho}_{T(f)}  \vert_{G_{\bb{Q}_p}} )$
is given by 
\[
\bar{\rho}_{T(f)} \vert_{G_{\bb{Q}_p}} \simeq 
\begin{pmatrix}
\lambda \left( \bar{a}(p,f) \right) 
 &* \\
0 &  
\bar{\chi}_{\mathrm{cyc}}^{k} \cdot 
\lambda \left( 
\overline{\psi_0(p)} / \bar{a}(p,f)
\right)
\end{pmatrix}
\]
where we denote the image of an element $x \in \mca{O}_\eta$ 
in $k$ by $\bar{x}$, and
for any $\bar{a} \in k^\times$, 
we define an unramified character 
$\lambda(\bar{a}) \colon 
G_{\bb{Q}_p} \longrightarrow k^\times$ by 
\(
\lambda(\bar{a})(\Frob_p^{-1}):=\bar{a}
\).
(For the proof of this fact under assumption $k \le p+1$, 
see \cite{Gro} Proposition 12.1. 
Note that Galois representations in \cite{Gro} 
are written in the homological convention.)
In particular, if we have
$A(\mca{F},p) \not\equiv \psi_0^{-1}(p) \mod \mf{m}_{\Lambda}$, 
then $\bb{T}'_{\mca{F}}$ satisfies (A7).
\item Recall that here, we assume that $p \ge 5$. 
For any Hida family $\mca{F}$ of ordinary cuspidal  
Hecke eigen
newforms of tame level $N$
with Dirichlet character $\psi$, 
there exists an integer $i \in \bb{Z}$ with $0 \le i \le p-2$
such that the representation 
$\bb{T}(\mca{F}\otimes \omega^i) \simeq 
\bb{T}(\mca{F}) \otimes \omega^i$
corresponding to the Hida family 
\[
\mca{F}\otimes \omega^i=
\sum_{n=1}^\infty A(n,\mca{F}) \omega^i(\ell)
q^n \in \Lambda [[q]]
\]
of cuspforms of tame level $N$
with Dirichlet character $\psi\omega^{2i}$
satisfies (A7).
\end{enumerate}
\end{rem}

\begin{prop}\label{propsuffcond}
Let $f = \sum_{n=1} a(n,f)q^n 
\in S_k(\Gamma_1(N); \mca{O}_\eta)$ 
be as in Remark \ref{remA67}. 
Suppose that $\mca{O}_\eta$ 
is unramified over $\bb{Z}_p$.
Let $\ell$ be a prime divisor of $N$. 
If we have $a(\ell,f) \ne 0$, and if 
$\ell^2-1$ is prime to $\# k^\times$,
then $f$ satisfies
$(\mathrm{A6})_{f,\ell}$.
\end{prop}

\begin{proof}
We denote by
$\pi_f= \bigotimes'_v \pi_{f,v}$ 
the automorphic representation of 
$\mathrm{GL}_2(\bb{A}_\bb{Q})$
attached to $\ell$. 
We put $G_\ell:= \mathrm{GL}(\bb{Q}_\ell)$.
Let $T_\ell$ be the maximal torus of $G_\ell$
consisting of diagonal matrices, 
and $B_\ell$ the Borel subgroup of $G_\ell$
consisting of upper triangle matrices.
In order to prove Proposition \ref{propsuffcond},
we treat the following three cases.
\begin{enumerate}[(I)]
\item The representation 
$\pi_{f,\ell}$ of $G_\ell$ is principal series, 
but not special in the sense of 
\cite{BH} \S 9.11 Classification Theorem.
\item The representation 
$\pi_{f,\ell}$ of $G_\ell$ is special. 
\item The representation 
$\pi_{f,\ell}$ of $G_\ell$ is supercuspidal.
\end{enumerate}

First, let us consider the case (I). 
Since we assume that $a(\ell,f)=0$, 
the representation 
$\pi_{f,\ell}$ of $G_\ell$ is not 
$p$-primitive 
in the sense of \cite{AL} p.\ 236.
(See also \cite{LW} Proposition 2.8.)
This fact and (the proof of) \cite{BH} \S 33.3 Theorem
imply that  
there exists two ramified characters
$\chi_1$ and $\chi_2$ of $G_{\bb{Q}_\ell}$
such that 
we have $\rho_{T(f)} \vert_{G_{\bb{Q}_\ell}} 
=\chi_1 \oplus \chi_2$.
This implies that 
$f$ satisfies
$(\mathrm{A6})_{f,\ell}$.

Next, let us consider the case (II).
The assumption that $a(\ell,f)$ that is equal to $0$
implies that $\pi_{f,\ell}$ is 
a twist of the Steinberg representation by a ramified character.
From this fact and \cite{BH} \S 33.3 Theorem, 
it follows that 
$f$ satisfies
$(\mathrm{A6})_{f,\ell}$.

Let us consider the case (III).
In this case, the representation 
$\rho_{V(f)} \vert_{G_{\bb{Q}_\ell}}$
is irreducible.
(See \cite{BH} \S 33.4 Theorem.)
Suppose that $T(f) \otimes_{\mca{O}_\eta} k$ 
has a non-zero fixed vector 
$\bar{v}$ under the action of $I_\ell$.
We denote by $I_\ell^w$
the wild inertia subgroup of $G_{\bb{Q}_\ell}$, 
and put $I_w^{t}:=I_\ell/I_\ell^w 
\simeq \varprojlim_n \bb{F}_{\ell^n}^\times$.
Since the kernel of the reduction map
$\mathrm{GL}_2(\mca{O}_\eta) 
\longrightarrow \mathrm{GL}_2(k)$
is pro-$p$, and since $I_\ell^w$
is a pro-$\ell$ group with $\ell \ne p$,
we deduce that 
there exists a lift $v \in T(f)$ of $\bar{v}$
fixed by $I_\ell^w$.
The group $I_\ell$ acts on $v$
via a character $\chi$ of $I_\ell^t$.
Since $I_\ell$ fixes $\bar{v}$, 
the order of $\chi$ is a power of $p$.
By the assumption that $p \ge 5$, 
and that $\mca{O}_\eta$ is unramified over $\bb{Z}_p$, 
we do not have an element of 
$\mathrm{GL}(\mca{O}_\eta)$
whose order is a positive power of $p$.
So, the character $\chi$ is trivial.
This contradicts the assertion that 
$\rho_{V(f)} \vert_{G_{\bb{Q}_\ell}}$
is irreducible.
Hence $T(f) \otimes_{\mca{O}_\eta} k$ 
does not have a non-zero fixed vector 
under the action of $I_\ell$.
\end{proof}

\begin{rem}
By Skinner and Urban, 
For the precise statement of the two-variable 
Iwasawa--Greenberg main conjecture for 
nearly ordinary deformations are proved in many cases.
Indeed, 
the ``three-variable" conjectures
which imply two--variable ones (after combined 
with \cite{Oc3} \S 2 Theorem 3 under suitable assumptions)  
are also proved under certain hypotheses. 
(See \cite{SU} Theorem 3.6.6.) 
However, in order to apply Skinner--Urban's work, 
we need to assume that the tame level $n$ 
has a prime divisor whose square does not divide $d$.
Under this assumption, the condition (A6)
$\widetilde{\bb{T}}'_{\mca{F}}$ does not hold.
So, we cannot apply the results in \cite{SU}
in our setting.
\end{rem}

\end{document}